\documentclass[11pt]{amsart}
\usepackage{xypic}
\usepackage{amsmath,amssymb,amscd,amsthm,graphicx,enumerate}
 \usepackage{graphicx}  
 \usepackage{float}
\usepackage{mathrsfs}
\usepackage{hyperref}

\usepackage{xcolor}
\usepackage{color}
\usepackage{mathtools}
\usepackage[normalem]{ulem}
\mathtoolsset{showonlyrefs}
\hypersetup{breaklinks=true}
\hypersetup{
    colorlinks=true,
    linkcolor=red,
    filecolor=magenta,      
    urlcolor=cyan,
    citecolor=green,
}

\numberwithin{equation}{section}
\theoremstyle{plain}

\newcommand{\bowl}{ \text{Bowl}^{n+1-k}}
\newcommand{\gra }{ \eta }

\newcommand{\cE}{\mathcal{E}}

\newcommand{\bY}{{\bar Y}}

\newcommand{\bry}{{\bf y}}

  \newcommand{\hD}{\mathfrak{D}}

\newcommand{\vp}{\varphi}

\newcommand{\mm}{\, e^\mu  d\omega    dv}

\newcommand{\collar}{\mathcal{K}}  
\newcommand{\pd}{\partial}

\newcommand{\dist}{{\mathrm{dist}}}
 
\newcommand{\eps}{\varepsilon}

\newcommand{\be}{\begin{equation}}
\newcommand{\bee}{\begin{equation*}}
\newcommand{\bea}{\begin{equation}\begin{aligned}}
\newcommand{\eea}{\end{aligned}\end{equation}}
\newcommand{\ee}{\end{equation}}
\newcommand{\eee}{\end{equation*}}
\newcommand{\bsp}{\begin{split}}
\newcommand{\esp}{\end{split}}

\newcommand{\R}{{\mathbb R}}

\newtheorem{theorem}{Theorem}[section]
\newtheorem{proposition}[theorem]{Proposition}
\newtheorem{lemma}[theorem]{Lemma}
\newtheorem{corollary}[theorem]{Corollary}

\newtheorem{claim}[theorem]{Claim}

\theoremstyle{remark}
\newtheorem{remark}[theorem]{Remark}

\theoremstyle{definition}
\newtheorem{definition}[theorem]{Definition}

\newcommand{\cC}{\mathcal{C}}
\renewcommand{\cD}{\mathcal{D}}
\newcommand{\cT}{\mathcal{T}}
\renewcommand{\cH}{\mathcal{H}}

\newcommand{\fp}{\mathfrak{p}}

\graphicspath{{./mathdraw/}}

\usepackage{fancyhdr}
\subjclass{53E10, 35K55, 35A21}
\keywords{ancient oval, singularity analysis of mean curvature flow, spectral uniqueness, reflection symmetry, spectral stability.}
\begin{document}

\title{Rigidity of ancient ovals in higher dimensional mean curvature flow}

\author[Beomjun Choi,  Wenkui Du, Jingze Zhu]{Beomjun Choi,  Wenkui Du, Jingze Zhu}

\begin{abstract}
In this paper, we consider the classification of compact ancient noncollapsed mean curvature flows of hypersurfaces in arbitrary dimensions. More precisely, we study $k$-ovals in $\mathbb{R}^{n+1}$, defined as ancient noncollapsed solutions whose tangent flow at $-\infty$ is given by $\mathbb{R}^k \times S^{n-k}(\sqrt{2(n-k)|t|})$ for some $k \in \{1,\dots,n-1\}$, and whose fine cylindrical matrix has full rank. A  significant advance achieved recently by Choi and Haslhofer suggests that the shrinking $n$-sphere and $k$-ovals together account for all compact ancient noncollapsed solutions in $\mathbb{R}^{n+1}$. We prove that $k$-ovals are $\mathbb{Z}^{k}_2 \times \mathrm{O}(n+1-k)$-symmetric and are uniquely determined by $(k-1)$-dimensional spectral ratio parameters. This result is sharp in view of the $(k-1)$-parameter family of $\mathbb{Z}^{k}_2 \times \mathrm{O}(n+1-k)$-symmetric ancient ovals constructed by Du and Haslhofer, as well as the conjecture of Angenent, Daskalopoulos and Sesum concerning the moduli space of ancient solutions. We also establish a new spectral stability theorem, which suggests the local $(k-1)$-rectifiability of the moduli space of $k$-ovals modulo space-time rigid motion and parabolic rescaling. In contrast to the case of $2$-ovals in $\mathbb{R}^4$, resolved by Choi, Daskalopoulos, Du, Haslhofer and Sesum, the general case for arbitrary $k$ and $n$ presents new challenges beyond increased algebraic complexity. In particular, the quadratic concavity estimates in the collar region and the absence of a global parametrization with regularity information pose major obstacles. To address these difficulties, we introduce a novel test tensor that produces essential gradient terms for the tensor maximum principle, and we derive a local Lipschitz continuity result by parameterizing $k$-ovals with nearly matching spectral ratio parameters.

\end{abstract}
\vspace{-10pt}
\maketitle

\thispagestyle{empty}
\vspace{-20pt}
\tableofcontents
\vspace{-20pt}
\section{Introduction}
Ancient ovals are compact non-selfsimilar ancient noncollapsed mean curvature flows of hypersurfaces in Euclidean space.  We recall that a mean curvature flow is called ancient if it is defined for all sufficiently negative time, and is called noncollapsed if it is mean convex and there is an $\alpha>0$ such that every point $p\in M_t$ admits interior and exterior  balls of radius at least $\alpha/H(p)$ c.f. \cite{ShengWang,Andrews_noncollapsing,HaslhoferKleiner_meanconvex, Brendle_inscribed, HK_inscribed}. For $1\le k\le n-1$, the existence of ancient oval, which has $\mathrm{O}(k)\times\mathrm{O}(n+1-k)$ symmetry and has tangent flow $\mathbb{R}^{k}\times S^{n-k}(\sqrt{2(n-k)|t|})$ at time $-\infty$\footnote{This is not an assumption, but rather a consequence. By \cite[Theorem 1.14]{HaslhoferKleiner_meanconvex} every  ancient noncollapsed flow has unique cylindrical tangent flow at $-\infty$. c.f. \cite{brendle2024local} showed convex ancient flow with cylindrical tangent flow at $-\infty$ must be noncollapsed.},  has been obtained by White \cite{White_nature} and Hershkovits-Haslhofer  \cite{HaslhoferHershkovits_ancient}. Later, Haslhofer and the second author \cite{DH_ovals}  constructed $(k-1)$-parameter family of $\mathbb{Z}_2^{k}\times\mathrm{O}(n+1-k)$-symmetric ancient ovals in $\mathbb{R}^{n+1}$, whose tangent flow at $-\infty$  is $\mathbb{R}^{k}\times S^{n-k}(\sqrt{2(n-k)|t|})$ and whose geometric width ratio parameters in coordinate directions of $\mathbb{R}^k$ are prescribed. 

Ancient ovals are related to the fine structure of singularities such as the accumulation of neckpinch singularities and the finiteness of singular times \cite{CM_arrival, CHH_ovals,sun2022generic, sun2025passing}. They arise as potential singularity models in the canonical neighborhood theorem, which is crucial for surgery of mean curvature flows c.f. \cite{HuiskenSinestrari_surgery, BH_surgery, HaslhoferKleiner_surgery}\footnote{We mention there are studies on compact ancient collapsed flows, known as pancakes, \cite{BLT1, BLT2}. Formally, they correspond to the case $k=n$. Collapsed solutions may appear as singularity models for immersed mean curvature flow as in \cite{BLT1, angenent1991formation}.}.

In recent studies, ancient ovals have gained significant attention due to their connection to conjectures regarding the selfsimilarity of limit flows (Ilmanen \cite[Conjecture 6]{Ilmanen_problems} and White \cite[Conjecture 1 and Conjecture 3]{White_nature}) and the fact that they are the only non-selfsimilar ancient solutions that possibly arise as limit flows of closed mean-convex solutions in $\mathbb{R}^3$ and $\mathbb{R}^4$ (see  \cite{BC1,ADS2,CH-classification-r4} and further references therein). We also note that ancient ovals and their classification also appear in the resolution of mean convex neighborhood conjecture \cite{Ilmanen_problems, CHH, CHHW} and the classification of noncollapsed translators in $\mathbb{R}^4$ \cite{CHH_translator}. The work \cite{CHH_translator}, in particular, shows that the cross sections of the noncollpased translators in $\mathbb{R}^4$ exhibit the same asymptotics as that of the ancient ovals in $\mathbb{R}^3$, suggesting a grater postulate that compact ancient noncollapsed flows are building blocks of noncompact ancient noncollapsed flows in general dimensions.

For $1 \leq k \leq n-1$, among ancient noncollapsed flows in general dimensions $\mathbb{R}^{n+1}$ with the tangent flow $\mathbb{R}^{k}\times S^{n-k}(\sqrt{2(n-k)|t|})$ at $-\infty$, \emph{$k$-ovals} are those ancient noncollapsed flows exhibiting inward quadratic bending toward each axis of the asymptotic cylinder\footnote{The concept of $k$-ovals was introduced by the second and third authors \cite[Definition 1.7] {DZ_spectral_quantization}.  See \cite{DH_hearing_shape, CDDHS} where $2$-ovals, known as bubble-sheet ovals, were studied.} (see precise description in Definition \ref{def-koval}).
Note that $k$-ovals are compact (thus they are ancient ovals) and $\mathrm{O}(n+1-k)$ symmetric \cite[Theorem 1.8]{DZ_spectral_quantization}. Experts in the field believe that   the class of $k$-ovals is not simply  a subclass of ancient ovals. Indeed, a series of recent breakthroughs have confirmed that these two classes are identical for $2$d and $3$d ancient ovals, leading to the complete classification modulo space-time rigid motion and parabolic dilation, summarized below:

\textbf{Classification of ancient ovals in $\mathbb{R}^3$ and $\mathbb{R}^4$} \cite{CH-classification-r4, ADS2, CDDHS}.

{\setlength{\leftmargini}{1.5em}
\begin{itemize}
    \item {(Choi-Haslhofer)} Every ancient oval in $\mathbb{R}^4$ is either a $1$-oval  or a $2$-oval. 
    \item {(Angenent-Daskalopoulos-Sesum)} 
The only ancient oval in $\mathbb{R}^3$ is the $\mathbb{Z}_2\times \mathrm{O}(2)$-symmetric oval from \cite{White_nature, HaslhoferHershkovits_ancient}, which is indeed a $1$-oval, and the only $1$-oval in $\mathbb{R}^4$ is the $\mathbb{Z}_2\times \mathrm{O}(3)$-symmetric oval from \cite{White_nature, HaslhoferHershkovits_ancient}.
    \item {(Choi-Daskalopoulos-Du-Haslhofer-Sesum)} Every $2$-oval in $\mathbb{R}^4$ is either the unique $\mathrm{O}(2)\times \mathrm{O}(2)$-symmetric oval from \cite{White_nature, HaslhoferHershkovits_ancient} or a member of the $1$-parameter family of $\mathbb{Z}_2^2\times \mathrm{O}(2)$-symmetric ovals from \cite{DH_ovals}.
\end{itemize}
}
 In higher dimensions, it is analogously expected that (i) $k$-ovals and the shrinking sphere account for all compact ancient noncollapsed mean curvature flows, and (ii) the class of $k$-ovals coincides with the previously discussed $(k-1)$-parameter family of ancient ovals constructed in \cite{DH_ovals}. In \cite{DH_ovals}, Haslhofer and the second author also obtained the uniqueness of $\mathrm{SO}(k)\times \mathrm{SO}(n+1-k)$ symmetric ancient oval, which is a $k$-oval. However, characterizing the rigidity of general $k$-ovals remains widely open except for $1$-ovals \cite{ADS2} and $2$-ovals in $\mathbb{R}^4$ \cite{CDDHS}. The main difficulty comes from the symmetry breaking phenomenon illustrated in the construction \cite{DH_ovals}: the solutions have only $\mathbb{Z}_2^{k}\times\mathrm{O}(n+1-k)$ symmetry. Thus, the crucial challenge is twofold: first, to identify $ k-1 $ quantities that capture the distinct geometric features among the solutions, and second, to establish a rigidity theorem proving that these quantities uniquely characterize each $ k $-oval.

Motivated by these challenges, our  first main result Theorem~\ref{thm:uniqueness_eccentricity_intro} (spectral uniqueness) and its Corollary~\!\ref{reflection symmetry} (symmetry of $k$-ovals) prove that, after suitable rigid motion in space-time, $k$-ovals are $\mathbb{Z}_2^{k}\times\mathrm{O}(n+1-k)$-symmetric and they are characterized via the spectral projection of the solution onto the neutral eigenspace of the linearized operator, which has dimension $k$.  Taking into account the additional identification due to parabolic dilation, this shows that $ k $-oval is characterized by the spectral ratio parameters  among these $ k $ components, which is $(k-1)$-dimensional (see \eqref{spectral ratio parameters}). We make this assertion more precise in the other main Theorem~\ref{spectral_stability_intro} (spectral stability), where we prove that properly space-time recentered and parabolically dilated $k$-ovals can  be locally parametrized by the $(k-1)$-dimensional spectral ratio parameters. Moreover, this local parametrization is Lipschitz, suggesting a $(k-1)$-rectifiability of the moduli space of $k$-ovals modulo symmetry. As a final remark, note that a similar spectral uniqueness theorem will also be an essential component in the classification of noncollapsed translators, whose cross sections exhibit asymptotic behavior analogous to that of ancient ovals defined in one lower spatial dimension (e.g., see the spectral uniqueness theorem for translators in $\mathbb{R}^4$ in \cite[Theorem 1.6]{CHH_translator}). The techniques developed here for $k$-ovals are expected to be employed in the classification of noncompact ancient noncollapsed mean curvature flows in general dimensions and in ancient $\kappa$-noncollapsed Ricci flows\footnote{In a parallel story of classification of noncompact ancient noncollapsed mean curvature flows, by works in \cite{BC1, BC2, CHH_translator, DH_hearing_shape, CH-classification-r4}, $\mathbb{R}^{j}\times S^{2-j}$ for $j=1, 2$ and $2d$-round bowl are the only noncompact ancient noncollapsed  flows in $\mathbb{R}^3$; $\mathbb{R}^{j}\times S^{3-j}$ for $j=1, 2, 3$,  $\mathbb{R}\times 2d$-round bowl translator,  $\mathbb{R}\times 2d$-round oval, and noncollapsed  $3d$-oval-bowl translators from \cite{HIMW} are the only noncompact ancient noncollapsed flows in $\mathbb{R}^4$. For a parallel classification story in Ricci flow,  the readers can refer \cite{brendle2020ancient, ABDS, Haslhofer_4dRicci, Lai}.}.

\subsection{Spectral uniqueness and stability}
Let us discuss the precise formulation and details of the main theorems. For $1\le k\le n-1$, let $\mathcal{M}=\{M_t\}$ be an ancient noncollapsed flow having the tangent flow $\mathbb{R}^k\times S^{n-k}(\sqrt{2(n-k)|t|})$ at $-\infty$. This means 
\begin{equation}\label{bubble-sheet_tangent_intro}
\lim_{\tau\to -\infty}\bar{M}_\tau=\lim_{\tau\to -\infty}e^{\frac{\tau}{2}} M_{-e^{-\tau}}=\mathbb{R}^{k}\times S^{n-k}(\sqrt{2(n-k)}).
\end{equation} 
We express the renormalized flow $\bar M_\tau$ by the profile function $ v(\mathbf{y}, \vartheta,\tau) $ defined on the domains $(\mathbf{y},\vartheta) \in \Gamma_\tau$ that exhaust the cylinder $\mathbb{R}^{k}\times S^{n-k}$ via
\begin{equation}\label{profile v def}
    \bar{M}_\tau =\{(\mathbf{y},v(\mathbf{y},\vartheta,\tau)\vartheta )\in \mathbb{R}^k \times \mathbb{R}^{n+1-k}\,:\, (\mathbf{y},\vartheta)\in \Gamma_\tau  \}.
\end{equation}
The asymptotics \eqref{bubble-sheet_tangent_intro} reads to the smooth local convergence of $v(\mathbf{y},\vartheta,\tau)$ to $\sqrt{2(n-k)}$ as $\tau$ approaches $-\infty$. The $k$-oval is defined in terms of the next-order asymptotics which describes $v(\tau)$ up to an error of order $o(|\tau|^{-1})$. 

\begin{definition}[$k$-oval, c.f.{\cite[Definition 1.7]{DZ_spectral_quantization}}]\label{def-koval} An ancient noncollapsed flow $\mathcal{M}$ in $\mathbb{R}^{n+1}$ is a $k$-oval  if it has the tangent flow $\mathbb{R}^{k}\times S^{n-k}(\sqrt{2(n-k)|t|})$ at $-\infty$ and the renormalized profile $v(\mathbf{y},\vartheta,\tau)$ in \eqref{profile v def} has the asymptotics 
\[\lim_{\tau\to -\infty} \bigg \Vert |\tau|(v(\mathbf{y},\vartheta,\tau) -\sqrt{2(n-k)})+ \frac{\sqrt{2(n-k)}}{4}(|\mathbf{y}|^2-2k)\bigg \Vert_{C^m(\{|\mathbf{y}|<R\})} =0,\] for all $m\in \mathbb{N}$ and $0<R<+\infty$\footnote{According to \cite[Theorem 1.2]{DZ_spectral_quantization},  ancient asymptotically-$\mathbb{R}^k\times {S}^{n-k}$-type ancient flow exhibits the asymptotics $v(\mathbf{y},\vartheta,\tau)-\sqrt{2(n-k)}=[\mathbf{y}^\top Q\mathbf{y} -2\mathrm{tr}(Q)]|\tau|^{-1}+o(|\tau|^{-1})$, for some symmetric matrix $Q$ having eigenvalues $-{\sqrt{2(n-k)}}/{4}$ or $0$. Definition \ref{def-koval} is the case when this fine cylindrical matrix $Q$ has full rank $k$.}. 
%
%
\end{definition}
As a consequence of \cite[Theorem 1.8] {DZ_spectral_quantization}, every $ k$-oval indeed has $ \mathrm{O}(n+1-k) $-symmetry in the last $n+1-k$ coordinates. Therefore, we will write the profile function of $k$-oval by $v(\mathbf{y},\tau)$.

Next, we consider the  $k$-dimensional Ornstein-Ulenbeck operator 
\begin{equation}\label{OU_operator}
\mathcal{L}=\Delta_{\bf{y}}-\frac{1}{2}{\bf{y}}\cdot  \nabla_{\bf{y}}+1,
\end{equation}
which appears in the linearization of the evolution equation of $v(\mathbf{y},\tau)$ in \eqref{eqn-v}. The operator $\mathcal{L}$   is a self-adjoint  on the Gaussian weighted $L^2$-space $\mathcal{H}$ and it admits a spectral decomposition 
\begin{equation}
\mathcal{H}=L^2\big(\mathbb{R}^k,e^{-\frac{|{\bf{y}}|^2}{4}}d{\bf{y}})= \mathcal{H}_+\oplus \mathcal{H}_0\oplus \mathcal{H}_-,
\end{equation}
where the positive eigenspace and neutral eigenspace are explicitly given by
\begin{equation}\begin{aligned}\label{H0+}
 &   \mathcal{H}_+=\textrm{span}\big\{1, y_1,\dots, y_k\}, \\ 
&\mathcal{H}_0=\textrm{span}\big\{y_1^2-2, \dots, y_k^2-2, y_{i}y_{j}, 1\leq i<j\leq k\big\}.\end{aligned}
\end{equation}
Let $\mathfrak{p}_{0}$ and $\mathfrak{p}_{\pm}$ be the orthogonal projections onto $\mathcal{H}_0$ and $\mathcal{H}_\pm$, respectively.

Using a smooth cut-off function $\chi_C:[0,\infty)\to [0,1]$, we consider truncated profile function $v_{\cC}=v\chi_{\cC}(v)$ supported on the cylindrical region $\big\{ v \ge \theta  \big\} $ (see \eqref{eq-cutoffC} for conditions of $\chi_\mathcal{C}$.) The following definition is a quantitative measure to the sharp asymptotics  in Definition \ref{def-koval}. To achieve main results, it is crucial that the estimates and theorems are not established for individual solutions but uniformly across the entire $\kappa$-quadratic family defined below.

\begin{definition}[$\kappa$-quadraticity at $\tau_0$]\label{k_tau00_intro}
A $k$-oval in $\mathbb{R}^{n+1}$ 
is called \emph{$\kappa$-quadratic at time $\tau_{0}$}\footnote{Compared with \cite[Definition 1.8]{CDDHS}, the positive mode orthogonality condition 
$\mathfrak{p}_{+}(v_{\cC}(\cdot,\tau_{0})-\sqrt{2(n-k)})=0$ is removed from the definition (as in \cite[Definition 1.4]{CHH_translator}). Instead, new conditions on positive mode are imposed in Theorem~\ref{thm:uniqueness_eccentricity_intro} (spectral uniqueness) and Theorem~\ref{spectral_stability_intro} (spectral stability).}, where $\kappa>0$ and $\tau_0<0$, if its truncated profile $v_{\cC}$ satisfies the following non-degenerate quadraticity condition and 
\begin{equation}\label{condition1intro}
    \left\| v_{\cC}(\mathbf{y}, \tau_{0})-\sqrt{2(n-k)}+\frac{\sqrt{2(n-k)}}{4|\tau_{0}|}(|\mathbf{y}|^2-2k)\right\|_{\mathcal{H}}\leq \frac{\kappa}{|\tau_{0}|}
\end{equation}
and graphical radius condition
\begin{equation}\label{condition3}
        \sup_{\tau\in [2\tau_{0}, \tau_{0}]} |\tau|^{\frac{1}{50}} \|v(\cdot ,\tau)-\sqrt{2(n-k)}\|_{C^{4}(B(0, 2|\tau|^{{1}/{100}}))}\leq 1.
\end{equation}
\end{definition}
The nonuniform asymptotics in Lemma \ref{lem-quadraticity} shows that  every $k$-oval $\mathcal{M}$ becomes $\kappa$-quadratic at all sufficiently negative time $\tau_0\le \tau_*=\tau_*(\mathcal{M},\kappa)$ for all $\kappa>0$. Namely, $k$-oval class is exhausted by families of $\kappa$-quadratic solutions at a sequence of time points going to $-\infty$.  With this preparation, we state the spectral uniqueness theorem.
\begin{theorem}[spectral uniqueness]\label{thm:uniqueness_eccentricity_intro} There exist $\kappa>0$ and $\tau_{\ast}>-\infty$ with the following significance:
If $\mathcal{M}^1$ and $\mathcal{M}^2$ are $k$-ovals in $\mathbb{R}^{n+1}$ that are $\kappa$-quadratic at time $\tau_0$, where $\tau_0 \leq \tau_{\ast}$, and if their trancated cylindrical profile functions $v^1_{\cC}$ and $v^2_{\cC}$ satisfy the following matching mode conditions
\begin{equation}
    \mathfrak{p}_{+}\big(v^1_{\cC}(\cdot , \tau_{0})\big)= \mathfrak{p}_{+}\big(v^2_{\cC}(\cdot , \tau_{0})\big)\quad\textrm{and}\quad  
\fp_{0}\big(v^1_{\cC}(\cdot , \tau_{0})\big)=\fp_{0}(v^2_{\cC}\big(\cdot , \tau_{0})\big),
\end{equation}
then two solutions are identical, namely  $\mathcal{M}^1=\mathcal{M}^2$.
\end{theorem}
The theorem requires  matching  in both positive and neutral eigenspaces. However, by applying a suitable translation in space-time, one can force the positive component to vanish at sufficiently negative times (see Proposition \ref{prop_orthogonality} and Remark \ref{remark_orthogonality}). This means that matching in the positive part can always be achieved via isometries in the ambient space. Furthermore, by comparing every $k$-oval $\mathcal{M}$ with its reflections across appropriately chosen coordinate hyperplanes, i.e., allowing additional rotations, we deduce that $k$-ovals possess additional reflection symmetries beyond the $\mathrm{O}(n+1-k)$ symmetry from \cite[Theorem 1.8]{DZ_spectral_quantization}.
\begin{corollary}[symmetry of $k$-ovals]\label{reflection symmetry}
 Every $k$-oval exhibits $\mathbb{Z}_2^k \times \mathrm{O}(n+1-k)$ symmetry after an appropriate rigid motion.
\end{corollary}
This, in particular, rules out the possibility of pathological $k$-ovals that appear to rotate within the $\mathbb{R}^k$ cylindrical factor as $\tau \to -\infty$. See \cite{DH_no_rotation} for the corresponding result in $\mathbb{R}^4$.

Theorem \ref{thm:uniqueness_eccentricity_intro} and Corollary \ref{reflection symmetry}  extend the rigidity results for $\textrm{SO}(k)\times\textrm{SO}(n+1-k)$ symmetric ancient oval in \cite[Theorem 1.4]{DH_ovals} and  $2$-ovals in $\mathbb{R}^4$ \cite{CDDHS}  to ancient ovals with all types of asymptotic cylinders in general dimensions. Moreover, the result is remarkable for following reasons: 
{\setlength{\leftmargini}{1em}
\begin{itemize}
 \item It establishes strong backward uniqueness based on  neutral eigenspace for ancient mean curvature flows. It is a purely nonlinear phenomenon that the dynamics within the kernel of linearized operator determines the behavior. This has sharp contrast with linear parabolic equations, whose ancient solutions are characterized by positive eigenspace. 

 \item It is a local-to-global result: matching on the truncated profiles at one time slice is sufficient to ensure global rigidity both in space and time. Additional subtlety lies in the fact that two different $k$-ovals even share the identical asymptotics as described in \cite{ADS1, DZ_spectral_quantization}.

 \item In view of the reflection symmetry in Corollary \ref{reflection symmetry}, there are $k$ eigenfunctions, namely $y_i^2-2$ for $i=1,\ldots,k$, in the neutral eigenspace such that $\langle v_{\cC}(\cdot, \tau_{0}),  y^2_{i}-2 \rangle_\cH$ , for $i=1,\ldots,k$, characterize $k$-ovals. Our results Theorem \ref{thm:uniqueness_eccentricity_intro} and Corollary \ref{reflection symmetry}  are sharp because  the $\mathbb{Z}_2^{k}\times\mathrm{O}(n+1-k)$ symmetry and the number of eigenfunctions in Theorem \ref{thm:uniqueness_eccentricity_intro} are optimal in light of the $(k-1)$-parameter family of $\mathbb{Z}_2^{k}\times\mathrm{O}(n+1-k)$-symmetric  ancient ovals constructed in \cite[Section 4]{DH_ovals}. Difference between $k$ and $k-1$ comes from the parabolic dilation. Note a parabolic dilation scales the magnitude of quadratic bending  at sufficiently negative time $\tau_0$
\[\Big\langle v_{\mathcal{C}}(\cdot, \tau_{0}),|\mathbf{y}|^2-2k \Big\rangle_{\mathcal{H}}=\sum_{j=1}^k\Big\langle v_{\mathcal{C}}(\cdot, \tau_{0}),y_j^2-2 \Big\rangle_{\mathcal{H}} \] 
(c.f. Remark \ref{remark_orthogonality}),
while none of two ancient ovals constructed in \cite{DH_ovals} are the same after parabolic dilations.  In particular, we can say the $k$-ovals are uniquely determined by the following $(k-1)$-dimensional  spectral ratio parameters at  a  sufficiently negative time $\tau_0$
\begin{equation}\label{spectral ratio parameters}
    \frac{\langle v_{\cC}(\cdot, \tau_{0}), y^2_{j}-2 \rangle_\cH}{\langle v_{\cC}(\cdot, \tau_{0}), |{\bf{y}}|^2-2k \rangle_\cH}\quad j=1,\dots, k,
\end{equation}
 and they are interpreted as  analytic version of the $(k-1)$-dimensional geometric width ratio parameters discussed in  \cite{DH_ovals}. 
\item The corollary on symmetry together with \cite[Theorem 1.10]{DZ_spectral_quantization} provides strong evidence to the expectation that every ancient oval has $\mathbb{Z}^k_2\times \mathrm{O}(n+1-k)$ symmetry for some $k=1, \dots, n-1$. This symmetry also supports the natural symmetry assumption in the conjecture of Angenent-Daskalopoulos-Sesum \cite{ADS_dynamics}  that  the space of all convex ancient solutions with a point symmetry and Huisken's Gaussian area lying in $(0, 2)$ is homeomorphic to an $(n-1)$-dimensional simplex.\footnote{In this conjecture, cylinders $\mathbb{R}^{k}\times S^{n-k}$ for $k=0,\dots, n-1$  are critical points of Huisken's Gaussian  area. They are regarded as the vertices of the above $(n-1)$-dimensional moduli space simplex. The $k$-ovals are the connecting orbits between these critical points.
}
\end{itemize}
}
In order to parametrize $k$-ovals with regularity information, we introduce a novel spectral stability theorem. To this end, given $\kappa>0$ and $\tau_0>-\infty$, we let $\bar{\mathcal{A}}_{\kappa}(\tau_0)$ be the set of all $k$-ovals which are $\kappa$-quadratic at time $\tau_0$, and the corresponding truncated profile function  $v_{\cC}$  satisfies following orthogonality conditions\footnote{This is always possible due to  Lemma \ref{lem-quadraticity} and Proposition \ref{prop_orthogonality} (orthogonality).}
\begin{equation}\label{condition_positive}
    \mathfrak{p}_{+}\big(v_{\cC}(\cdot,  \tau_{0})-\sqrt{2(n-k)}\big)=0,
\end{equation}
    \begin{equation}\label{OC11}
      \Big\langle v_{\cC}(\cdot, \tau_0), y_{i}y_{j}\Big\rangle_\cH=0,\quad 1\leq i<j\leq k, 
    \end{equation}
    \begin{equation}\label{OC21}
         \Big\langle v_{\cC}(\cdot, \tau_0) +\frac{\sqrt{2(n-k)}(|\mathbf{y}|^2-2k)}{4|\tau_0|},|\mathbf{y}|^2-2k\Big\rangle_\cH=0.
    \end{equation}
Then  we introduce  spectral ratio map  $\mathcal{E}(\tau_0)$ for $k$-oval $\mathcal{M}\in \bar{\mathcal{A}}_{\kappa}(\tau_0)$,   which is given by
\begin{equation}\label{spectral map def0}
\mathcal{E}(\tau_0)(\mathcal{M})=  \left(\frac{\langle v_{\cC}(\cdot, \tau_{0}), y^2_{1}-2 \rangle_\cH}{\langle v_{\cC}(\cdot, \tau_{0}), |{\bf{y}}|^2-2k \rangle_\cH}, \dots,\frac{\langle v_{\cC}(\cdot, \tau_{0}),  y^2_{k}-2 \rangle_\cH}{\langle v_{\cC}(\cdot, \tau_{0}), |{\bf{y}}|^2-2k \rangle_\cH}\,\right),
\end{equation}
and when the $\tau_0$ is clearly fixed, we also use the abbreviated notion
\begin{equation}\label{spectral map def1}
    \mathcal{E}(\mathcal{M})=\mathcal{E}(\tau_0)(\mathcal{M}).
\end{equation}
Then our spectral stability theorem is stated as following.
\begin{theorem}[spectral stability]\label{spectral_stability_intro}
    There exist constants $\kappa>0$, $\tau_{*}>-\infty$, $C<\infty$  with the following significance: for $\tau_0\le \tau_*$, suppose $\mathcal{M}^1, \mathcal{M}^2\in \bar{\mathcal{A}}_{\kappa}(\tau_0)$, namely they are $k$-ovals $\kappa$-quadratic at $\tau_0$ and satisfy orthogonality conditions \eqref{condition_positive}--\eqref{OC21}.
  For each fixed $0 < l < +\infty$, there is a constant $\delta(\tau_0, l)>0$ such that  if  the  spectral closeness condition
    \begin{equation}\label{spectral closedness condition}
        |\mathcal{E}(\tau_0)(\mathcal{M}^1)-\mathcal{E}(\tau_0)(\mathcal{M}^2)|\leq \delta(\tau_0, l)
    \end{equation} holds at $\tau_0$, then the renormalized flow $\bar {M}^1_\tau$ is a normal graph over the other renormalized flow $\bar{M}^2_\tau$ for $\tau \in [\tau_0 - l, \tau_0]$. Moreover, for each $m\in \mathbb{N}$, the graph function $\zeta(\tau)$ satisfies the following estimates 
    \begin{align}\label{bilip stability}
         c\sup_{\tau\in [\tau_0-l, \tau_0]}\|\zeta(\tau)\|_{C^{m}}\leq |\mathcal{E}(\tau_0)(\mathcal{M}^1)-\mathcal{E}(\tau_0)(\mathcal{M}^2)|\leq C\|\zeta(\tau_0)\|_{C^{m}},
    \end{align}
    holds, where  $c=c(\tau_0, l, m)>0$ is a small constant.
\end{theorem} 
Theorem \ref{spectral_stability_intro} (spectral stability)  is a quantitative version of Theorem \ref{thm:uniqueness_eccentricity_intro} (spectral uniqueness), which allows us to parametrize families of $k$-ovals with almost matching spectral ratio parameters in a Lipschitz continuous way. It also suggests that   the  moduli space of $k$-ovals has local Lipschitz-regularity and $(k-1)$-rectifiability (see \cite[Section 3.2.14]{federer2014geometric} or  \cite[Chapter 7]{Rectifiability} for definition of rectifiability).

%


\subsection{Major new challenges}\label{challenge} 
In this subsection, we outline several significant new challenges distinct from previous works. To introduce them, we define
the \emph{collar region}
\begin{equation}
\label{eq-collar-region0}
\collar = \bigl\{  L/\sqrt{|\tau|} \le v \le 2  \theta \bigr\},
\end{equation}
which interpolates asymptotic behavior of $k$-ovals in the \emph{cylindrical region}
\begin{equation}\label{cylindrical_ region}
\mathcal{C}=\big\{ v \ge \theta  \big\}    
\end{equation}
 and the \emph{soliton region} \begin{equation}\label{soliton_ region}
     \mathcal{S}=\bigl\{v \le L/\sqrt{|\tau|} \bigr\}
 .\end{equation}
In addition, the \emph{tip region} refers
\begin{equation}\label{eq-tip-region0}
\mathcal{T}= \big\{v \le 2   \theta \big\}=\mathcal{K}\cup\mathcal{S}.
\end{equation}
Throughout the paper,  $\theta>0$ is a fixed small enough constant and $0<L<+\infty$ is a fixed large enough constant to be determined later.

First, the almost Gaussian collar Theorem~\ref{prop-great} represents the most challenging aspect of our analysis. This estimate in the collar region bridges the gap on asymptotics between the {cylindrical region} in \eqref{cylindrical_ region} and the {soliton region} in \eqref{soliton_ region}--regions where the solutions exhibit markedly different asymptotic behaviors (see  asymptotics in intermediate and tip regions in Theorem \ref{strong_uniform0}). In effect, this allows us to combine energy estimates from both regions to establish the spectral uniqueness theorem. The Gaussian collar estimate arises from a concavity property of $v^2({\bf y}, \tau)$ along the radial direction in $\mathbb{R}^k$, a property we refer to as \emph{quadratic almost concavity}.  In the following detailed discussion about quadratic almost concavity, $g$ denotes the metric on the manifold induced from the embedding $M_t$ in $\mathbb{R}^{n+1}$, the Hessian of function $Q$ is computed with respect to the Riemannian connection of $g$, and $\partial_{\vartheta}$ is any vector field on $M_t$ along $S^{n-k}$ direction. Previously, this quadratic almost concavity was established for $2$-ovals in $\mathbb{R}^4$ through a argument in the proof of tensor maximum principle  of Hamilton \cite{Ham_pco} on the intrinsic manifold.

 \begin{theorem}[{quadratic almost concavity, \cite[Theorem 3.9]{CDDHS}}]
      \label{prop-concavity_introstated}
There exist constants $\kappa>0$ and $\tau_*>-\infty$ with the following significance. If  $\mathcal{M}$ is a $2$-oval  (bubble-sheet oval)  in $\mathbb{R}^4$ which is $\kappa$-quadratic at time $\tau_0 \le \tau_*$, then for all $t\le -e^{-\tau_0}$ we have
\begin{equation}\label{concavity_coincidence}
\bar{A}(X, X):=\nabla^2 Q(X,X) -\gamma g(X,X)\leq 0
\end{equation}
for every tangent vector $X$ of $M$ with $X\perp \partial_\vartheta$. Here, $Q=V^2$ where $V$ is the profile function of the unrenormalized flow $M_t$ (viewed as an intrinsic function in the manifold) and $\gamma= \Big (\frac{\sqrt{|t|}}{\sqrt{Q\log |t|}}\Big )^{3}$.
\end{theorem}
However, for general $k$-ovals in $\mathbb{R}^{n+1}$, the quadratic almost concavity estimate for $Q = V^2$ fails unless the dimension of the sphere factor in the tangent flow $\mathbb{R}^k\times S^{n-k}(\sqrt{2(n-k)|t|})$ at $-\infty$ satisfies $n-k = 1$. Following the proof in \cite[Section 3]{CDDHS}, let $X_p $ be a null-vector of $\bar A$ satisfying $X_p\perp \partial_\vartheta $ which happens to exist for the first time. In a neighborhood of $p$, extend $X_p$ to a vector field $X\perp \partial_{\vartheta} $, e.g., as in Claim \ref{claim_extension} (extension),  and consider the function $A(X,X)$. (Hamilton's original argument does not have such a step, but here this is required as the maximum principle is applied to a subbundle of tangent bundle that is perpendicular to ${S}^{n-k}$.) A direct computation shows that the evolution equation for $\bar{A}(X, X) = \nabla^2 Q(X, X) - \gamma\, g(X, X)$
actually takes the form
\begin{equation}\label{difficulty}
    (\partial_t-\Delta+\nabla_Z)(\bar{A}(X, X))\leq({\frac{2(n-k)-3}{4}}+\zeta+\gamma)\frac{|\nabla_X Q|^2}{Q^2}(\frac{|\nabla Q|^2}{Q}-2\gamma),
\end{equation}
where $\zeta > 0$ denotes an arbitrarily small constant  and $Z$ is a smooth vector field (see \eqref{eq-defZ} for its precise definition). Notably, although $(\frac{|\nabla Q|^2}{Q}-2\gamma)\ge 0$ and $\gamma+\zeta$ is small, the reaction term on the right-hand side does not have a preferred negative sign for $n-k\ge 2$, thereby precluding the application of the maximum principle. Consequently, it becomes necessary to construct a novel, finely tuned test tensor for the maximum principle argument, which is essential for deriving the Gaussian collar estimate in Theorem~\ref{prop-great}.

Another major challenge arises from additional terms that complicate the energy estimates for higher-dimensional $k$-ovals when $k>2$. It is convenient to work with the polar coordinates \begin{equation}\label{v_radial_representation}
v(y,\omega,\tau):=v(\mathbf{y},\tau),
\end{equation}
where $\mathbf{y}=y\,\omega$, $y=|\mathbf{y}|\ge0$ and $\omega \in S^{k-1}$. To work in the tip region, let us also  define the inverse profile function $Y(v, \omega, \tau)$  of $v$  by the equation.
\begin{equation}\label{inverse profile}
    Y(v(y, \omega, \tau), \omega, \tau)=y.
\end{equation}
Then, extra terms emerging from non-radial directions in $\mathbb{R}^k$ appear in the evolution equation of $Y$
\begin{align}
    Y_{\tau}
    \!&\!=\frac{(Y^2 + |\nabla_{S^{k-1}}Y|^2)\nabla^2_{vv}Y - 2 \nabla^2 Y ( \nabla_{S^{k-1}}Y, \nabla_{\mathbb{R}}Y ) + (1 +|\nabla_\mathbb{R} Y|^2)\Delta_{S^{k-1}}Y}{Y^2\, (1 + |\nabla_\mathbb{R} Y|^2)+|\nabla_{S^{k-1}}Y|^2} \nonumber\\
    &+\Biggl(\frac{n-k}{v}-\frac{v}{2}\Biggr)\nabla_v Y -\frac{|\nabla_{S^{k-1}}Y|^2}{Y\Bigl[Y^2(1+|\nabla_\mathbb{R} Y|^2)+|\nabla_{S^{k-1}}Y|^2\Bigr]}+\frac{Y}{2}-\frac{k-1}{Y} \nonumber\\
    &+ \frac{1}{Y^2}\frac{\Delta_{S^{k-1}}Y\,|\nabla_{S^{k-1}}Y|^2-\nabla^2 Y( \nabla_{S^{k-1}}Y, \nabla_{S^{k-1}}Y)}{Y^{2}(1+|\nabla_\mathbb{R} Y|^2)+|\nabla_{S^{k-1}}Y|^2}.
\end{align}
Here, $|\cdot|$ and $\nabla$ denote the norm and the connection induced by the cylindrical metric $dv^2+g_{S^{k-1}}$, $\nabla_{S^{k-1}}V$ and $\nabla_{\mathbb{R}}V$ represent the projections of $\nabla V$ onto the spherical and $\mathbb{R}$ components, respectively, and $\Delta_{S^{k-1}}$ is the Laplacian on the sphere $S^{k-1}$. These additional terms (for example, the new terms in the last line were unseen in \cite[(4.67)]{CDDHS} for the $2$-oval case in $\mathbb{R}^4$.) must be handled with exceptional care.

The final obstacle concerns the Lipschitz parametrization required for proving Theorem~\ref{spectral_stability_intro} (spectral stability). The difficulty arises from the fact that we only have almost matching spectral ratio parameters $\mathcal{E}(\mathcal{M}^1)$ and $\mathcal{E}(\mathcal{M}^2)$ at a single time $\tau_0$, rather than a global continuous parametrization via the inverse spectral ratio map $\mathcal{E}^{-1}$, where the spectral ratio map $\mathcal{E}$ is defined in \eqref{spectral map def0} and \eqref{spectral map def1}.

\subsection{Outline of the proof}
\label{sec-outline}
According to main steps, we organize our argument in the next four sections as follows:
{\setlength{\leftmargini}{1.5em}
\begin{itemize}
\item uniform sharp asymptotics,
\item quadratic almost concavity,
\item spectral uniqueness theorem,
\item spectral stability theorem.
\end{itemize}
}

In Section~\ref{Uniform sharp asymptotics}, we establish uniform asymptotics for $\kappa$-quadratic $k$-ovals in three distinct regions, corresponding to the \emph{parabolic region} $\mathcal{P} = \{{\bf y} : |{\bf y}| \leq \varepsilon^{-1}\}$, the \emph{cylindrical region} (intermediate region) $\mathcal{C} = \{ v \ge \theta \}$, and the \emph{soliton region} (piece of tip region) $\mathcal{S} = \{ v \le L/\sqrt{|\tau|} \}$.
Note also that we often represent functions in the polar coordinates. i.e. $ v(y,\omega,\tau)=v(\mathbf{y},\tau)$ as in \eqref{v_radial_representation} for $y=|{\bf y}|\ge0$ and  $\omega \in S^{k-1}$.
\begin{theorem}[uniform sharp asymptotics c.f. {\cite[Section 2]{CDDHS}, \cite[Section 5]{DZ_spectral_quantization}}]\label{strong_uniform0} For every $\varepsilon>0$, there exists $\kappa>0$ and $\tau_{*}>-\infty$ such that if a $k$-oval $\mathcal{M}=\{M_t\}$ in $\mathbb{R}^{n+1}$ is $\kappa$-quadratic at time $\tau_{0}\leq \tau_{*}$, then{\setlength{\leftmargini}{1.5em} 
\begin{itemize}
\item Parabolic region:    for every $\tau\leq \tau_{0}$ we have:
     \begin{equation}
    \sup_{|{\bf{y}}|\leq \eps^{-1}}\left| v({\bf{y}}, \tau)-\sqrt{2(n-k)}+\frac{\sqrt{2(n-k)}}{4}\frac{|{\bf{y}}|^2-2k}{|\tau|} \right|\leq \frac{\varepsilon}{|\tau|} ;
    \end{equation}

\item Intermediate region:  for every $\omega\in  S^{k-1}$ and  $\tau\leq \tau_{0}$ we have:
  \begin{equation}
       \sup_{z\leq \sqrt{2}-\eps}\left|\bar{v}(z,\omega,\tau)-\sqrt{(n-k)(2-z^2)}\right|\leq \varepsilon,
\end{equation}
where $\bar{v}(z,\omega,\tau) = v(|\tau|^{\frac{1}{2}}z, \omega, \tau)$;

\item Tip region:  for every $\omega \in S^{k-1}$ and every $s \leq -e^{- \tau_{0}}$, we have that the rescaled flow  $\widetilde{M}^{\omega,s}_t$ obtained by centering $\mathcal{M}$ to tip point and parabolically dilating a factor $\sqrt{|s|^{-1}\log|s|}$ (see \eqref{widetildeM})
is $\eps$-close in $C^{\lfloor 1/\eps\rfloor}$ in $B_{\eps^{-1}}(0) \times (-\eps^{-2},\eps^{-2})$ to $N_t\times\mathbb{R}^{k-1}$, where $N_t$ is the unique rotationally symmetric translating bowl soliton $\bowl$ in $\mathbb{R}^{n-k+2}$ with tip $0\in N_0$ that translates in negative $\bar{\omega} = (\omega, 0)$ direction with speed $1/\sqrt{2}$.
\end{itemize}}
\end{theorem}

We begin by establishing uniform sharp asymptotics under the stronger assumption of strong $\kappa$-quadraticity (see Definition~\ref{strong}). We then upgrade this by showing the estimate holds under a weaker assumption that the ancient solutions are merely $\kappa$-quadratic at a single time. This includes a quantitative Merle–Zaag type argument, as in \cite[Section~3]{CHH_translator} and \cite[Section~2]{CDDHS}. However, in our setting the ODE dynamics among spectral components is considerably more intricate (cf. \cite{DH_hearing_shape, DZ_spectral_quantization}).

In Section \ref{Quadratic_almost_concavity_section},
we derive the important piece of asymptotic information in the
\emph{collar region} $\collar = \bigl\{  L/\sqrt{|\tau|} \le v \le 2  \theta \bigr\}$ as defined in \eqref{eq-collar-region0}.
\begin{theorem}[almost Gaussian collar\footnote{It is referred to as almost Gaussian collar since \eqref{almost gauss collar derivative} implies $Y(v, \omega, \tau)\sim Ce^{-\frac{v^2}{n-k}}$ in the collar region. c.f. \cite[(4.64)]{CDDHS} and \cite[Corollary 5.9]{CHH_translator}.}] \label{prop-great}
For every $\varepsilon > 0$, there exist constants $\kappa>0$, $\tau_*>-\infty$, $L<\infty$, and $\theta >0$ with the following significance. If  $\mathcal{M}$ is a $k$-oval which is $\kappa$-quadratic at time $\tau_0 \le \tau_*$, then for all $\tau \le \tau_0$ and for the profile function $v$ in \eqref{v_radial_representation} and its inverse profile function $Y$ in \eqref{inverse profile}, the following radial derivative estimates
\begin{equation}\label{gaussian_collar_estimate}
\left|\nabla(v^2)\cdot {\bf y} + 4(n-k)\right| < \varepsilon
\end{equation}
and equivalently the almost Gaussian collar estimates
\begin{equation}\label{almost gauss collar derivative}
    \left|1+\frac{vY}{2(n-k)Y_{v}}\right| < \varepsilon
\end{equation}
hold in the collar region $\mathcal{K}= \bigl\{  L/\sqrt{|\tau|} \le v \le 2  \theta \bigr\}$.
 \end{theorem} 

The almost Gaussian collar theorem follows from a quadratic almost concavity estimate and is crucial in establishing derivative estimates required for Theorem \ref{thm:uniqueness_eccentricity_intro}. To state and proof the quadratic almost concavity, we view unrescaled profile function $V(x, t)=\sqrt{|t|}v({\bf y}, \tau)$, for ${\bf y}=|\tau|^{-1}x, t=-e^{-\tau}$, as an intrinsic function on a manifold $S^n$ via the pull-back through the embedding of evolving hypersurface (see Section \ref{sec_intrinsic}) and consider $Q=V^2$. To overcome the new challenge that was unseen in the quadratic almost concavity estimate in the case $n-k=1$ as discussed in Section \ref{challenge}, we use the following test tensor in Theorem \ref{prop-concavity0} by adding two novel gradient type terms to the previously employed tensor $\bar{A}=\nabla^2 Q-\gamma g$ in \cite{CDDHS}. Consequently the following novel version of almost concavity estimates below
holds.
\begin{theorem}[quadratic almost concavity]\label{prop-concavity0}
There exist constants $\kappa>0$ and $\tau_*>-\infty$ with the following significance. If  $\mathcal{M}$ is a $k$-oval which is $\kappa$-quadratic at time $\tau_0 \le \tau_*$, then for all $\tau\leq\tau_0$ we have
\begin{equation}\label{almost_quadratic concavity_estimates_intro}
\left[{\nabla^2} Q- \gamma g-\frac{|\nabla Q|^2}{4Q}\left(\frac{\nabla Q\otimes\nabla Q}{2Q}-\gamma g\right)\right]\!\!(X, X)\leq 0
\end{equation}
for every vector field $X$ perpendicular to all  vector fields along $S^{n-k}$, where
\begin{equation}
    \gamma=\Bigg (\frac{\sqrt{|t|}}{\sqrt{Q\log |t|}}\Bigg )^{3}.
\end{equation}
\end{theorem}      
\begin{remark}
At first glance, the new tensor appears significantly more complicated than the one in Theorem~\ref{prop-concavity_introstated} (\cite[Theorem 3.9]{CDDHS}). However, the final two gradient-type terms are motivated by \eqref{difficulty} and by the computations in the proof of almost Gaussian collar estimate in Theorem~\ref{prop-great} (compare \eqref{quadratic concavity ode inequality} with (3.106) in \cite{CDDHS}). This adjusted tensor resolves the main obstacle in applying the argument of proving the tensor maximum principle discussed in Section~\ref{challenge}. 
\end{remark}

In Section~\ref{spectral_uniqueness}, we prove  our main result Theorem~\ref{thm:uniqueness_eccentricity_intro} (spectral uniqueness) and Corollary \ref{reflection symmetry} (symmetry of $k$-ovals). We follow the approach in \cite{CDDHS}, to derive and combine energy estimates in distinct regions as explained in Section \ref{challenge}. To this end, we carefully handle the algebraic complexity introduced by new non-radial terms in the evolution equations of the inverse profile function $Y$. Moreover, we establish energy estimates in both the cylindrical region $\mathcal{C}$ defined in \eqref{cylindrical_ region} and the tip region $\mathcal{T}$ defined in \eqref{eq-tip-region0} by employing suitably adapted weighted norms.

For these energy estimates, besides the Gaussian $L^2$-norm $\| \,\, \|_{\mathcal{H}}$, we employ the Gaussian $H^1$-norm
\begin{equation}
\|f\|_{\hD} :=\left(\int_{\R^{k}} \big( f(\bry)^2 +|Df(\bry)|^2 \big) \, e^{-|\bry|^2/4}d\bry\right)^{1/2}
\end{equation}
and its dual norm $\| \,\, \|_{\hD^\ast}$. Moreover, for time-dependent functions this induces the parabolic norms 
\begin{equation}
\|f \|_{\mathcal{X},\infty}:=\sup_{\tau\leq \tau_0 }\left( \int_{\tau-1}^\tau \| f(\cdot,\sigma)\|^2_{\mathcal{X}} \, d\sigma \right)^{1/2},
\end{equation}
where $\mathcal{X}=\mathcal{H},\hD$ or $\hD^\ast$. Furthermore, in the tip region we work with the norm
\begin{equation}
\| F\|_{2,\infty}:= \sup_{\tau\leq \tau_0} \frac{1}{|\tau|^{1/4}} \left( \int_{\tau-1}^\tau \int_0^{2\theta}\int_{S^{k-1}} F^2 e^{\mu}\, d\omega\, dv\, d\sigma \right)^{1/2},
\end{equation}
where $\mu=\mu(v,\omega,\sigma)$ is a weight function interpolating between the Gaussian weight in the cylindrical region and $\mathbb{R}^{k-1}\times\bowl$ related weight in the tip region (see \eqref{eqn-weight1} for precise definition).  Theorem \ref{prop-great} (almost Gaussian collar) will be employed to ensure that our weight function $\mu$ has the desired derivative estimates in tip region for the energy estimates (see \eqref{eqn-imp} for more details).

For Proposition \ref{prop-cyl-est} (energy estimate in cylindrical region), we consider the difference of two truncated profile functions 
\begin{equation}\label{wcutc}
w_{\cC} := v_{\cC}^1 - v_{\cC}^2=v_1\chi_\cC(v_1)-v_2\chi_\cC(v_2),
\end{equation}
and for Proposition \ref{prop-tip}(energy estimate in tip region) we consider the truncated difference of inverse profile functions
\begin{equation}\label{WcutT}
W_{\mathcal{T}}:=\chi_{\mathcal{T}} W=\chi_{\mathcal{T}} (Y_1-Y_2),
\end{equation}
where $Y_i(\cdot,\omega,\tau)$,  for $i=1, 2$,  is defined as the inverse function of $v_i(\cdot,\omega,\tau)$ as in \eqref{inverse profile} and  $\chi_{\mathcal{T}}$ is a cut-off function that localizes in the tip region.  In the tip region, we need more attention to analyze the new terms appearing in evolution equation of $W$ in Proposition \ref{lemma-ev-W-appendix} to settle the second difficulty mentioned in Section \ref{challenge}.

Combining two energy estimates, we prove that 
\begin{theorem}[coercivity energy estimate, c.f. {\cite[Proposition 4.15]{CDDHS}}]
For every $\varepsilon > 0$ there exist $\kappa > 0$ small enough, $\theta > 0$  small enough in definitions of cylindrical region $\mathcal{C}$ and tip region $ \mathcal{T}$ and $\tau_* > -\infty$ negative enough 
such that if $\mathcal{M}^{1}$ and $\mathcal{M}^{2}$ are $k$-ovals which  are $\kappa$-quadratic at  time $\tau_0\leq \tau_{*}$ and satisfy  $\mathfrak{p}_{+}\big(v^1_{\cC}(\cdot , \tau_{0})\big)= \mathfrak{p}_{+}\big(v^2_{\cC}(\cdot , \tau_{0})\big)$, then we have the following coercive energy estimates
\be\label{eq-coercive1}
\|w_\cC-\mathfrak{p}_0 w_\cC \|_{\hD,\infty}+\|W_\cT\|_{2,\infty} \leq  \eps \|\mathfrak{p}_0 w_\cC \|_{\hD,\infty}\, .
\ee
\end{theorem}

Finally, by applying the estimate \eqref{eq-coercive1} and referring the norm equivalence in the transition region $\{\theta\leq v\leq 2\theta\}$, we deduce that $w_{\cC}=0$ and $W_{\mathcal{T}}=0$, hence $\mathcal{M}^1=\mathcal{M}^2$, which proves Theorem~\ref{thm:uniqueness_eccentricity_intro} (spectral uniqueness). Moreover, Corollary~\ref{reflection symmetry} (symmetry of $k$-ovals) follows by applying suitable space-time shift, rotation, and parabolic dilation via Proposition~\ref{prop_orthogonality}, and then invoking Theorem~\ref{thm:uniqueness_eccentricity_intro}.

\bigskip

In Section \ref{stability section}, we  prove Theorem \ref{spectral_stability_intro} (spectral stability) for the spectral ratio map at time $\tau_0$ 
\begin{equation}
   \mathcal{E}(\mathcal{M})=\mathcal{E}(\tau_0)(\mathcal{M})=\left(\frac{\langle v^{\mathcal{M}}_{\cC}(\tau_{0}), y^2_{1}-2 \rangle_\cH}{\langle v^{\mathcal{M}}_{\cC}(\tau_{0}), |{\bf{y}}|^2-2k \rangle_\cH}, \dots,\frac{\langle v^{\mathcal{M}}_{\cC}(\tau_{0}),  y^2_{k}-2 \rangle_\cH}{\langle v^{\mathcal{M}}_{\cC}(\tau_{0}), |{\bf{y}}|^2-2k \rangle_\cH}\,\right).
\end{equation}
The proof is based on showing the locally bi-Lipschitz continuity of spectral ratio map $\mathcal{E}$ for  $k$-ovals via quantifying the proof of Theorem \ref{thm:uniqueness_eccentricity_intro} (spectral uniqueness). Namely, we prove that if $\mathcal{M}^1, \mathcal{M}^2$ are $k$-ovals that are $\kappa$-quadratic at sufficiently negative time $\tau_0$ for small enough $\kappa>0$ and that satisfy orthogonality conditions \eqref{condition_positive}, \eqref{OC11} and \eqref{OC21}, then $w_\cC$ and $W_\cT$ defined in \eqref{wcutc} and \eqref{WcutT} satisfy
\begin{align}
        \|w_\cC\|_{\hD,\infty} +\| W_\cT \|_{2,\infty} \leq C\|\mathfrak{p}_0 w_{\mathcal{C}}(\tau_0)\|_{\mathcal{H}}.
    \end{align}
To finish the proof of Theorem \ref{spectral_stability_intro} (spectral stability), we also carefully apply an implicit function argument  under the condition \eqref{spectral closedness condition}. Then we use the derivative estimates and parabolic regularity theory to show that for $\tau_0$ negative enough, we have that the renormalized flow $\bar{M}^1_{\tau}$ can be written as a graph over $\bar{M}^2_{\tau}$
    of a function $\zeta({\tau})$ satisfying
    \begin{align}\label{part2}
        \sup_{\tau \in [\tau_0-l, \tau_0]}\|\zeta(\tau)\|_{C^{m}}\leq C(\tau_0, l, m)\|\mathfrak{p}_0 w_{\mathcal{C}}(\tau_0)\|_{\mathcal{H}   },
    \end{align}
where $0<C(\tau_0, l, m)<+\infty$ is a constant depending on $\tau_0$, $l$ and $m$. Then  by definition of spectral ratio map $\mathcal{E}(\tau_0)(\mathcal{M}^{i})$ at $\tau_0$,  Theorem \ref{spectral_stability_intro} (spectral stability) follows from letting $c(\tau_0, l, m)=C(\tau_0, l, m)^{-1}>0$ and  that 
 \begin{align}
       |\mathfrak{p}_0 w_{\mathcal{C}}(\tau_0)\|_{\mathcal{H}} \leq |\mathcal{E}(\tau_0)(\mathcal{M}^1)-\mathcal{E}(\tau_0)(\mathcal{M}^2)| \quad \text{for}\,\, \tau_0\,\,\text{negative enough}
    \end{align}
due to \eqref{OC21}.
On the other hand, 
\begin{equation}
    |\mathcal{E}(\tau_0)(\mathcal{M}^1)\!-\!\mathcal{E}(\tau_0)(\mathcal{M}^2)|\!\leq \!C\|\zeta(\tau_0)\|_{C^{m}}
\end{equation}
follows from the definitions of spectral ratio map $\mathcal{E}(\tau_0)(\mathcal{M}^{i})$ at $\tau_0$ and truncated profile functions $v^i_{\mathcal{C}}$ for $i=1, 2$, along with the  $\tau_0$ being very negative and orthogonality condition \eqref{OC21}.

\textbf{Acknowledgments.}
The first author has been supported by the National Research Foundation of Korea (NRF) grants funded by the Korean government (MSIT)  NRF-2022R1C1C1013511, RS-2023-00219980 and by Samsung Science \& Technology Foundation grant SSTF-BA2302-02.
The second author has been supported by the postdoctoral positions of University of Toronto and  MIT. The third author  has been supported by the postdoctoral position of  MIT. The authors appreciate the communication with Professor Panagiota Daskalopoulos, Professor Robert Haslhofer and Professor Natasa Sesum on the current paper.

\section{Uniform sharp asymptotics}\label{Uniform sharp asymptotics}

In this section, we establish uniform sharp asymptotics for our $k$-ovals. Our scheme of proof, similarly to the one for translators from \cite[Section 3]{CHH_translator}, is to first derive uniform sharp asymptotics under a stronger a priori  assumption, called strong $\kappa$-quadraticity, and then to use quantitative Merle-Zaag type arguments to show that  $\kappa$-quadracity at one time implies strong $\kappa$-quadraticity.
Throughout this section $\mathcal{M}=\{M_t\}$ denotes a $k$-oval in $\mathbb{R}^{n+1}$, where by \cite[Theorem 1.8]{DZ_spectral_quantization} we can always assume that we have $\textrm{SO}(n+1-k)$-symmetry in the $x_{k+1}...x_{n+1}$-plane centered at  the origin. Since the tangent flow at $-\infty$ is given by \eqref{bubble-sheet_tangent_intro}, for $\tau \to -\infty$ the renormalized flow
\begin{equation}
\bar M_\tau = e^{\frac{\tau}{2}}  M_{-e^{-\tau}}
\end{equation}
converges smoothly on compact subsets to the static cylinder
\begin{equation}
\Gamma:=\mathbb{R}^k\times S^{n-k}(\sqrt{2(n-k)}).
\end{equation}
We denote points in $\mathbb{R}^k$ by
\begin{equation}
{\bf y} = (y_1,y_2,...,y_k)=y \omega, \quad \mathrm{ where } \quad y=|{\bf y}|  \text{ and } \omega \in S^{k-1}(1)\subset \mathbb{R}^{k} .
\end{equation}
Let $\bar{\Omega}_\tau$ be the set of points ${\bf y}\in\mathbb{R}^k$ such that $({\bf y},r\vartheta)\in \bar{M}_\tau$ for some $r\geq 0$, and define $u({\bf y},\tau)$, where ${\bf y}\in\bar{\Omega}_\tau$, by
\begin{equation}
({\bf y},(\sqrt{2(n-k)}+u({\bf y},\tau))\vartheta) \in \bar M_\tau\, .
\end{equation}
Note that the above cylindrical graphical function $u$ and the profile function $v$ given in \eqref{profile v def} are related by
\begin{equation}\label{rel_v_u}
v({\bf y},\tau)=\sqrt{2(n-k)}+u({\bf y},\tau).
\end{equation}
In addition, it is also convenient to often represent profile functions in the polar coordinates. i.e., 
\begin{equation}
    v(y,\omega,\tau)=v(\mathbf{y},\tau), \text{\, for \, }y\ge0 \text{\, and \,} \omega \in S^{k-1}.
\end{equation}
Since $\bar{M}_\tau$ evolves by renormalized mean curvature flow by \cite{DH_hearing_shape}[Proposition A.1], $u$ satisfies
\begin{align}\label{equation_u} 
     u_\tau
    = \left(\delta_{ij}-\frac{u_{y_i}u_{y_j}}{1+|Du|^2}\right) u_{y_iy_j}-\frac{1}{2}   y_i u_{y_i}
    + \frac{ \sqrt{2(n-k)}+u}2 -\frac{n-k}{\sqrt{2(n-k)}+u}\, ,
\end{align}
where the summation convention is used over all indices $1\leq i, j \leq k$.  
Furthermore, fixing a smooth cut-off function with $\chi(s)=1$ for $s\leq 1$ and $\chi(s)=0$ for $s\geq 2$, we often consider the truncated graphical function
\begin{equation}
\hat{u}({\bf y},  \tau)=u({\bf y},  \tau)\chi\left(\frac{|{\bf y}|}{\rho(\tau)}\right).
\end{equation}
We say $\rho(\tau)$ is an {\em admissible graphical radius} for $\tau\le \tau_0$ if \begin{equation}
\label{univ_fns}
\lim_{\tau \to -\infty} \rho(\tau)=\infty,  \quad\text{and}\quad   -\rho(\tau) \leq \rho'(\tau) \leq 0,  
\end{equation}and
\begin{equation}\label{small_graph_admissible}
\|u(\cdot,\tau)\|_{C^4(\Gamma \cap B_{2\rho(\tau)}(0))} \leq  \rho(\tau)^{-2},
\end{equation}for $\tau\le \tau_0$.
We recall that our Gaussian inner product is given by the formula
\begin{equation}
\langle f,g\rangle_{\mathcal{H}}=\int_{\mathbb{R}^{k}} f({\bf y})g({\bf y})e^{-\frac{|{\bf y}|^2}{4}}\, d{\bf y}\, ,
\end{equation}
and that there is the well-known weighted Poincar\'e inequality
\begin{equation}\label{easy_Poincare}
\big\|  (1+|{\bf{y}}|) f \big\|_{\mathcal{H}} \leq C \big( \| f \|_{\mathcal{H}}+\| Df \|_{\mathcal{H}}\big).
\end{equation}
Indeed, by approximation it is enough to check this for smooth compactly supported functions $f$, and for such functions this follows by computing
\begin{align}
\int_{\mathbb{R}^{k}} \left(\tfrac12 |{\bf{y}}|^2f^2 -2f^2\right)\, e^{-\frac{|{\bf y}|^2}{4}}\, d{\bf y}&=\int_{\mathbb{R}^{k}} D(f^2)\cdot {\bf{y}}\, e^{-\frac{|{\bf y}|^2}{4}}\, d{\bf y}\nonumber\\
&\leq \int_{\mathbb{R}^{k}} \left(\tfrac14 |{\bf{y}}|^2 f^2 +4|Df|^2 \right)\, e^{-\frac{|{\bf y}|^2}{4}}\, d{\bf y}.
\end{align}
We also recall and summarize the regions we discussed in the introduction for later use:  \emph{cylindrical region}
\begin{equation}
\label{eq-cyl-region}
\cC= \big\{ v \ge \theta  \big\},
\end{equation}
and the \emph{tip region}
\begin{equation}
\label{eq-tip-region}
\mathcal{T}= \big\{v \le 2   \theta \big\},
\end{equation}
where we further subdivide the latter into  the \emph{collar region}
\begin{equation}
\label{eq-collar-region}
\collar = \bigl\{  L/\sqrt{|\tau|} \le v \le 2  \theta \bigr\},
\end{equation}
and the \emph{soliton region}
\begin{equation}
\label{eq-soliton-region}
\mathcal{S} = \bigl\{v \le L/\sqrt{|\tau|} \bigr\}.
\end{equation}
\subsection{\texorpdfstring{Uniform sharp asymptotics assuming strong $\kappa$-quadraticity}{Uniform sharp asymptotics assuming strong kappa-quadraticity}}\label{sec2.1}

In this subsection, we establish uniform sharp asymptotics under the following a priori assumption:
\begin{definition}[{strong $\kappa$-quadraticity, c.f. \cite[Definition 3.7]{CHH_translator}} ]\label{strong}
We say that an ancient noncollapsed flow $\mathcal{M}$ in $\mathbb{R}^{n+1}$ (with coordinates and tangent flow $-\infty$ as above) is \emph{strongly $\kappa$-quadratic from time $\tau_{0}$},  if 
\begin{enumerate}[(i)]
\item $\rho(\tau)=|\tau|^{1/10}$ is an admissible graphical radius for $\tau\leq \tau_{0}$, and
\item  the truncated graphical function $\hat{u}(\cdot,  \tau)=u(\cdot,  \tau)\chi\left(\frac{|\cdot|}{\rho(\tau)}\right)$ satisfies 
\begin{equation}
    \left\| \hat{u}({\bf{y}},  \tau)+\frac{\sqrt{2(n-k)}}{4|\tau|}{(|{\bf{y}}|^2-2k)} \right\|_{\mathcal{H}}\leq \frac{\kappa}{|\tau|}\quad \text{for}\,\,\tau\leq \tau_{0}.
\end{equation}
\end{enumerate}
\end{definition}

It is of independent interest to verify that a given solution becomes strongly $\kappa$-quadratic from a sufficiently negative time. This is established in the following lemma, and we will use it in later sections.
\begin{lemma}\label{lem-quadraticity}
For a given $k$-oval $\mathcal{M}$ in $\mathbb{R}^{n+1}$ and $\kappa>0$, there exists $\tau_*(\mathcal{M},\kappa)>\infty$ such that $\mathcal{M}$ is strongly $\kappa$-quadratic from time $\tau_*$ and $\kappa$-quadratic at time $\tau_0$ for all $\tau_0\in (-\infty,\tau_*]$.

\begin{proof}We first prove the strong $\kappa$-quadraticify assertion of the lemma. By \cite[Proposition 2.8]{DZ_spectral_quantization}, there exists universal constant $\gamma>0$ such that $\rho_0(\tau):=|\tau|^{\gamma}$ becomes an admissible radius function of $\mathcal{M}$ for for $\tau\le \tau_0(\mathcal{M})$. Let us we define $    \rho_1(\tau):=\beta(\tau)^{-\frac{1}{5}}$, where
\begin{equation}\label{def_beta}
\beta(\tau):=\sup_{\sigma\leq \tau}\left(\int_{\mathbb{R}^k}{u}^2({\bf y},\sigma)\chi^2\left(\frac{|{\bf y}|}{\rho_{0}(\sigma)}\right)e^{-\frac{|{\bf y}|^2}{4}}d{\bf y}\right)^{1/2}.
\end{equation}
Note that $\rho_1(\tau)$ is also an admissible graphical radius for sufficiently negative times by \cite[Proposition 2.6]{DZ_spectral_quantization}.

By \cite[Theorem 3.3]{DZ_spectral_quantization} (sharp asymptotics in $\mathcal{H}$-norm),  \begin{equation}\label{eq-thm33DZ} \left\| {u}({\bf{y}},  \tau)\chi \Big (\frac{|\mathbf{y}|}{\rho_0(\tau)}\Big )+\frac{\sqrt{2(n-k)}}{4|\tau|}{(|{\bf{y}}|^2-2k)} \right\|_{\mathcal{H}}=o(|\tau|^{-1}).\end{equation}
Note this implies that  $\beta(\tau)\le C|\tau|^{-1}$ for sufficiently negative times and thus $\rho(\tau)=|\tau|^{1/10}\le \rho _1(\tau)$ is an admissible graphical radius. Next, in view of the Gaussian weight in $\mathcal{H}$,  \eqref{eq-thm33DZ} implies the condition (ii) in 
Definition \ref{strong}. This finishes the assertion on the strong quadraticity.

Observe that the two truncation functions employed in the definitions of $\kappa$-quadraticity and strong $\kappa$-quadraticity are different. Nevertheless, by the Gaussian weight in $\mathcal{H}$-norm, there is $\kappa'=\kappa'(\kappa)$ and $\tau'=\tau'(\kappa)>-\infty $ such that if $\mathcal{M}$ is strongly $\kappa'$-quadratic from time $\tau_* \le \tau'$, then it is $\kappa$-quadratic at all times $\tau_0\le \tau_*$.  
\end{proof}

\end{lemma}

Let $\lambda(s) = \sqrt{|s|^{-1}\log|s|}$ and fix  spherical $\omega\in S^{k-1}(1)\subset \mathbb{R}^k$, let
\begin{equation}\label{R polar}
\tilde{p}_s^\omega=(R(\omega,s)\omega,0^{n+1-k}),
\end{equation}
where the radius in direction $\omega$ can be expressed as
\begin{equation}
R(\omega,s)=|s|^{1/2}\sup \{ y\geq 0 : v(y,\omega,-\log(-s)) > 0\},
\end{equation}
and we consider the flow
\begin{equation} \label{widetildeM}
\widetilde{M}^{\omega,s}_t:= \lambda(s)\cdot (M_{s+\lambda(s)^{-2}t} - \tilde{p}_s^\omega).
\end{equation}

We will now upgrade the sharp asymptotics from \cite[Theorem 1.8]{DZ_spectral_quantization} for every single $k$-oval, which was used to derive $\textrm{O}(n+1-k)$-symmetry, to uniform sharp asymptotics for families of strongly  $\kappa$-quadratic solutions.

\begin{proposition}[c.f.{\cite[Section 2]{CDDHS}, \cite[Theorem 1.8]{DZ_spectral_quantization}}]\label{strong_sharp_asymptotics} 
For every $\varepsilon>0$, there exists $\kappa>0$ and $\tau_{*}>-\infty$ such that if a $k$-oval $\mathcal{M}=\{M_t\}$ in $\mathbb{R}^{n+1}$ is strongly $\kappa$-quadratic from time $\tau_{0}\leq \tau_{*}$, then {\setlength{\leftmargini}{1.5em}
\begin{itemize}
\item Parabolic region:  for every $\tau\leq \tau_{0}$ we have:
     \begin{equation}
    \sup_{|{\bf{y}}|\leq \eps^{-1}}\left| u({\bf{y}}, \tau)+\frac{\sqrt{2(n-k)}}{4}\frac{|{\bf{y}}|^2-2k}{|\tau|} \right|\leq \frac{\varepsilon}{|\tau|} ;
    \end{equation}

\item Intermediate region:    for every $\omega\in  S^{k-1}$ and time $\tau\leq \tau_{0}$ we have:
  \begin{equation}
       \sup_{z\leq \sqrt{2}-\eps}\left|\bar{v}(z,\omega,\tau)-\sqrt{(n-k)(2-z^2)}\right|\leq \varepsilon,
\end{equation}
where $\bar{v}(z,\omega,\tau) = v(|\tau|^{\frac{1}{2}}z, \omega, \tau)$;

\item Tip region:    for every $\omega \in S^{k-1}$ and  $s \leq -e^{- \tau_{0}}$, we have that the flow $\widetilde{M}^{\omega,s}_t$ 
is $\eps$-close in $C^{\lfloor 1/\eps\rfloor}$ in $B_{\eps^{-1}}(0) \times (-\eps^{-2},\eps^{-2})$ to $N_t\times\mathbb{R}^{k-1}$, where $N_t$ is the unique rotationally symmetric translating $\bowl$ in $\mathbb{R}^{n-k+2}$ with tip $0\in N_0$ that translates in negative $\bar{\omega} = (\omega, 0)$ direction with speed $1/\sqrt{2}$.
\end{itemize}}
\end{proposition}
\begin{proof}
We follow closely to the proof in \cite{CDDHS, DZ_spectral_quantization}, and for the readers' convenience, we include the proof here.

In the parabolic region as discussed in as in \cite[Propostion 2.2]{CDDHS}, using  the strong $\kappa$-quadraticity from time $\tau_0$ assumption, property of graphical radius $\rho$
\eqref{small_graph_admissible} and  the parabolic regularity estimates of equation   \eqref{equation_u}  for
\begin{align}
    \mathcal{D}({\bf{y}}, \tau) =  \hat{u}({\bf{y}}, \tau)+\frac{\sqrt{2(n-k)}}{4}\frac{|{\bf{y}}|^2-2k}{|\tau|},
\end{align}
 we can find constants $C=C(\varepsilon)<\infty$ and $\tau_{*}(\varepsilon)>-\infty$  such that 
\begin{equation}\label{w32_est}
    \|\mathcal{D}\|_{\mathcal{H}}\leq \frac{\kappa}{|\tau|}\quad \big\| {\mathcal{D}}(\cdot,\tau)\big\|_{W^{3,2}(B(0,\eps^{-1}))}\leq \frac{C}{|\tau|}
\end{equation}
holds for all $\tau\leq \tau_{0}$, provided $\tau_{0}\leq \tau_{*}(\varepsilon)$. Applying Agmon's inequality with  \eqref{w32_est}, we conclude that for all $\tau\leq\tau_0\leq\tau_\ast(\eps)$ we have
\begin{equation}
  \sup_{|{\bf y}|\leq \eps^{-1}}  |{\mathcal{D}}({\bf y},\tau)| \leq \frac{\eps}{|\tau|},
\end{equation}
provided $\kappa=\kappa(\eps)>0$ is sufficiently small. Remembering support of cut-off function $\chi(\cdot/\rho(\tau))$, we obtain asymptotics in parabolic region.

In the intermediate region, arguing as in \cite[Propostion 2.3]{CDDHS}, the sharp asympotics in the parabolic region can be promoted to the sharp lower bound in the intermediate region by the barrier argument in \cite[Theorem 5.1]{DZ_spectral_quantization}:
\begin{align}
    \inf\limits_{z\leq \sqrt{2}-\varepsilon} \left(\bar{v}(z,\omega,\tau) - \sqrt{(n-k)(2-z^2)}\right) \geq -\varepsilon.
\end{align}
Next, the evolution equation (\ref{equation_u}) together with the convexity gives
\begin{align}
    v_{\tau} \leq -\frac{n-k}{v} + \frac{1}{2}(v-yv_y).
\end{align}
Hence, given any $\omega \in S^{k-1}(1)$, the function
\begin{align}
    w^{\omega}(y,\tau) := v(y,\omega,\tau)^2 - 2(n-k)
\end{align}
satisfies 
\begin{align}\label{ODE_w}
    w_{\tau}^{\omega} \leq w^{\omega} - \frac{1}{2}y w_y^{\omega}.
\end{align}
On the other hand, by the uniform asymptotics in the parabolic region, given any $A<\infty$, there are $\kappa_{*}(A)>0$ and $\tau_{*}(A)>-\infty$, such that if the $k$-oval is strongly $\kappa$-quadratic from time $\tau_{0}\leq \tau_{*}$, where $0<\kappa\leq \kappa_{*}$, then 
\begin{equation}
     w^\omega(y, \tau)\leq |\tau|^{-1}(n-k)({2k-y^2})+A^{-1}|\tau|^{-1}
\end{equation}
holds for all $y\leq A$.
Then we apply the characteristic method  for \eqref{ODE_w} as in \cite[ Proposition 2.3]{CDDHS} and \cite[Theorem 5.1]{DZ_spectral_quantization}   to conclude the upper bound:
\begin{align}
    \sup\limits_{z\leq \sqrt{2}-\varepsilon} \left(\bar{v}(z,\omega,\tau) - \sqrt{(n-k)(2-z^2)}\right) \leq \varepsilon
\end{align}
holds for all $\omega\in S^{k-1}$ and all $\tau\leq\tau_0\leq \tau_\ast$, provided $\kappa>0$  is sufficiently small and $\tau_\ast$ is sufficiently negative.
This finishes the proof of the second part of the proposition.

Finally we discuss the tip region as in \cite[Proposition 2.4]{CDDHS} \cite[Theorem 5.1]{DZ_spectral_quantization}. Let us fix a normal direction $\Phi\in S^{k-1}(1)$, one can find $p_s$ that maximize $\langle p,(\Phi, 0)\rangle$ among all $p \in M_s$. We define the rescaled flow: $$\widehat{M}^{\Phi,s}_t = \lambda(s)(M_{s+\lambda^{-2}(s)}-p_s).$$
Then we have the tip region  uniform sharp asymptotics in terms of normal direction: 
\begin{claim}[c.f. {\cite[Proposition 2.4]{CDDHS} \cite[Theorem 5.1]{DZ_spectral_quantization}}]\label{Phi tip}
the flow $\widehat{M}^{\Phi,s}_t$ 
is $\eps$-close in $C^{\lfloor 1/\eps\rfloor}$ in $B_{\eps^{-1}}(0) \times (-\eps^{-2},\eps^{-2})$ to $N_t\times\mathbb{R}^{k-1}$, where $N_t$ is the round $\bowl$ in $\mathbb{R}^{n-k+2}$ with tip $0\in N_0$ that translates in negative $(\Phi, 0)$ direction with speed $1/\sqrt{2}$. 
\end{claim}
\begin{proof}[Proof of Claim \ref{Phi tip}]
    As in \cite[Proposition 2.4]{CDDHS}, the claim follows by a blowup contradiction argument and applying  intermediate region  asymptotics, convexity and 
Hamilton's Harnack inequality  \cite{Hamilton_Harnack} and inductively using \cite[Theorem 1.4]{DH_blowdown} as in the proof of \cite[Theorem 5.1]{DZ_spectral_quantization}.
\end{proof}

Now, we need to translate this Claim \ref{Phi tip} to $\omega$-spherical coordinate  version tip asymptotics to complete the proof of  the third part of the proposition. We  observe the following derivative bound for convex radial graph hypersurface in the $\mathbb{R}^k$: If $r=r(\omega)$ represents a closed convex radial graph in $\mathbb{R}^k$ and
\begin{equation}	
	\max_{\omega\in S^{k-1}} |r(\omega)-1| \leq \delta,
\end{equation}	 
then as the discussion in \cite[Corollary 2.5]{CDDHS} we have
\begin{equation}	
	\max_{\omega \in S^{k-1} }|\nabla r(\omega)| \leq \eps( \delta ),
\end{equation}	
where $\eps(\delta )   \to 0$ as $\delta \to 0$.

Now, in our setting thanks to asymptotics in the intermediate region and convexity, given any $\delta>0$, by choosing $\kappa>0$ small enough and $\tau_\ast$ negative enough, we can arrange that the radius $R$ in \eqref{R polar} satisfies
\begin{equation}
\left|\frac{R(\omega,s)}{\sqrt{2|s|\log|s|}}-1 \right|\leq \delta.
\end{equation}
This yields
\begin{equation}
\sup_{\omega \in S^{k-1}}|\nabla R (\omega,s)|\leq \eps(\delta) \sqrt{2 |s| \log| s|},
\end{equation}
and consequently $d_{S^{k-1}}(\omega, \Phi) < \varepsilon_2(\delta)$, where $\varepsilon_2 \rightarrow 0 $ as $\delta \rightarrow 0$.
In other words, $\omega$ and $\Phi$ differ by arbitrarily small amount. Hence, this and the uniform tip asymptotics in terms of normal direction conclude  the proof of the third part of the proposition on the tip region  uniform sharp asymptotics in terms of $\omega$-spherical coordinate.
\end{proof}

\subsection{\texorpdfstring{From $\kappa$-quadraticity to strong $\kappa$-quadraticity}{From kappa-quadraticity to strong kappa-quadraticity}}\label{sec2.2}

In this subsection, we will use a quantitative Merle-Zaag type argument to upgrade $\kappa$-quadraticity (see Definition \ref{k_tau00_intro}) to strong $\kappa$-quadraticity (see Definition \ref{strong}).  In our setting we have to analyze a more complicated system of spectral ODEs, c.f. \cite{DH_hearing_shape}. We aim to prove the following theorem in this subsection.
\begin{theorem}[strong $\kappa$-quadraticity]\label{point_strong}
    For every $\kappa>0$, there exist $\kappa'>0$ and $\tau_{*}>-\infty$, such that if a $k$-oval in $\mathbb{R}^{n+1}$ is $ \kappa'$-quadratic at some time $\tau_{0}\leq \tau_{*}$, then it is strongly $\kappa$-quadratic from time $\tau_{0}$.
\end{theorem}
To this end, we need the following  initial graphical radius estimate and quantitative Merle-Zaag type estimate.

\begin{lemma}[initial graphical radius, c.f {\cite[Lemma 2.6]{CDDHS}}]\label{poly_graph}
There exists some universal number $\gamma>0$ with the following significance. For every $\kappa>0$ sufficiently small, there exists a constant $\tau_{\ast}>-\infty$,  such that if a $k$-oval in $\mathbb{R}^{n+1}$ is $\kappa$-quadratic at time $\tau_0\leq \tau_{\ast}$, then  $\rho(\tau)=|\tau|^\gamma$ is an admissible graphical radius function for  $\tau \leq \tau_0$, namely \eqref{univ_fns} and \eqref{small_graph_admissible} hold for  $\tau \leq \tau_0$.
\end{lemma}
 \begin{lemma}[{quantitative Merle-Zaag type estimate. {c.f. \cite[Lemma 2.7]{CDDHS}}}]\label{quant_MZ}
For every $\kappa>0$ sufficiently small, there exists a constant $\tau_{\ast}>-\infty$,  such that if a bubble-sheet oval in $\mathbb{R}^4$ is $\kappa$-quadratic at time $\tau_0\leq \tau_{\ast}$, then for $\tau \leq\tau_0$ we have the estimate
\begin{equation}\label{zero_dom_mz}
U_{+}(\tau)+U_-(\tau) \leq \frac{C'}{|\tau|^\gamma}U_{0}(\tau),
\end{equation}
where $C'<\infty$ is a numerical constant.
\end{lemma}

Now, we work with the truncated graphical function
\begin{equation}
\hat{u}(\cdot,  \tau)=u(\cdot,  \tau)\chi\left(\frac{|\cdot|}{|\tau|^\gamma}\right),
\end{equation} 
and we consider the following expansion in $\mathcal{H}$-norm sense,
\begin{equation}\label{expansion_123}
 \hat{u} = \sum_{m=1}^{k}\alpha_{mm}(y^2_{m}-2)+\sum_{1\leq i<j\leq k} \alpha_{ij}(2y_{i}y_{j})+\hat{w}\, .
\end{equation}
Let 
\begin{equation}\label{psi_defi}
    \psi_{mm}=y^2_{m}-2\quad \psi_{ij}=2y_{i}y_{j},
\end{equation}
and $A=(\alpha_{ij})$ be the symmetric $k\times k$ spectral coefficients matrix with
\begin{equation}\label{def_exp_coeffs}
  \alpha_{ij} = \|\psi_{ij}\|_{\mathcal{H}}^{-2} \langle  \psi_{ij},\hat{u}\rangle_{\mathcal{H}}.
\end{equation}
Then, by \cite[Prop 3.1 and Lem 3.4]{DZ_spectral_quantization}, we have the following proposition about the spectral coefficients.
\begin{proposition}[spectral ODE system, {\cite[Prop 3.1 and Lem 3.4]{DZ_spectral_quantization}}]\label{odes-1}
There exists a universal constant $\eta > 0$ such that the spectral coefficients matrix $A(\tau)$ satisfies the following  ODE system: 
\begin{equation}\label{odes0}
 \dot A=- \beta_{n, k}^{-1} A^2+E,
\end{equation}
where 
\begin{equation}
    \beta_{n, k}=\frac{\sqrt{2(n-k)}}{4}
\end{equation}
and
\begin{equation}
    |E|\leq C|\tau|^{-\frac{\gamma}{2}}|A|^2 + Ce^{-\eta|\tau|^{\gamma}/10}, \quad |A|=O(|\tau|^{-1})
\end{equation}
holds with $\gamma > 0$ from Lemma \ref{poly_graph}.
\end{proposition}

With the above ingredients, we can now prove the main result in this subsection:

\begin{proof}[Proof of Theorem \ref{point_strong}]
    Let $\lambda_{1}, \dots, \lambda_{k}$ be the eigenvalues of $A$ and let $\lambda_{k+1}=0$, and $S_{j}=\sigma_{j}(A)$ be the symmetric polynomial of these eigenvalues and $ \lambda_{k+1}$, which depends on $A$ smoothly  and where $j=1,\dots,k, k+1$ and $S_{k+1}=0$. We have the following claim
\begin{claim}\label{claim Sequation}
For $1\leq j\leq k$  
\begin{equation}\label{Sequation}
\frac{d}{d\tau}S_{j}=-\beta^{-1}_{n, k}(S_{1}S_{j}-(j+1)S_{j+1})+E_{j},
\end{equation}
where $|E_{j}|\leq C|\tau|^{-\gamma/2}|A|^{j+1}=O(|\tau|^{-j-1-\gamma/2})$.
\end{claim}
\begin{proof}[Proof of Claim]
Let $I_{j}=\{(i_{1}, \dots i_{k}): \sum_{m=1}^{k}i_{m}=j,\, i_{m}\in\{0, 1\}\}$. We also recall that $\lambda_{1}, \dots, \lambda_{k}$ be the eigenvalues of $A$ and we set $\lambda_{k+1}=0$. By the formulas in \cite{magic}, we can write the symmetric polynomial of eigenvalues 
\begin{equation}
\begin{split}
    S_{j}&=\sum_{(i_{1}, \dots i_{k})\in I_{j}} \lambda^{i_{1}}_{1}\lambda^{i_{2}}_{2}\cdots\lambda^{i_{m}}_{m}\cdots\lambda^{i_{k}}_{k}\\
    &=\sum_{(i_{1}, \dots i_{k})\in I_{j}}\det(A^{i_{1}}_{1},A^{i_{2}}_{2},\dots,A^{i_{m}}_{m},\dots,A^{i_{k}}_{k}),
\end{split}
\end{equation}
where $A^{i_m}$  is the $i_{m}$th power of matrix $A$ and $A^{i_{m}}_{m}$ is the $m$th column of the matrix $A^{i_{m}}$. Notice that the eigenvalues of $A^{i_{m}}$ are $\lambda^{i_{m}}_{1}, \lambda^{i_{m}}_{2}, \dots, \lambda^{i_{m}}_{k}$.\\
Then, by taking derivative on determinant and using Proposition \ref{odes-1} we have
\begin{align}\label{dsj}
    &\quad(-\beta_{n,k})\frac{d}{d\tau}S_{j}\\
    &=(-\beta_{n,k})\sum_{m=1}^{k}\sum_{(i_{1}, \dots i_{k})\in I_{j}}\det(A^{i_{1}}_{1},A^{i_{2}}_{2},\dots,\delta^{i_{m}}_{1}(\frac{d}{d\tau}A)_{m},\dots,A^{i_{k}}_{k}) \\\nonumber
     &=\sum_{m=1}^{k}\sum_{(i_{1}, \dots i_{k})\in I_{j}}\det(A^{i_{1}}_{1},A^{i_{2}}_{2},\dots,\delta^{i_{m}}_{1}(A^2+O(\tau^{-2-\frac{\gamma}{2}}) )_{m},\dots,A^{i_{k}}_{k})\\\nonumber
    &=\sum_{m=1}^{k}\sum_{(i_{1}, \dots \hat{i}_{m} \dots,i_{k})\in I_{j-1}}\lambda^{i_{1}}_{1}\lambda^{i_{2}}_{2}\cdots\hat{\lambda}^{i_{m}}_{m}\lambda^{2}_{m}\cdots\lambda^{i_{k}}_{k}+O(\tau^{-j-1-\frac{\gamma}{2}}).\\\nonumber
\end{align}
On the other hand, by  counting the repeating terms in the right hand side of equation below, we have
\begin{equation}
    (j+1)S_{j+1}=\sum_{m=1}^{k}\sum_{(i_{1}, \dots {i}_{m}=1 \dots,i_{k})\in I_{j+1}}\lambda^{i_{1}}_{1}\lambda^{i_{2}}_{2}\cdots\hat{\lambda}^{i_{m}}_{m}\lambda_{m}\cdots\lambda^{i_{k}}_{k}.
\end{equation}
Using this we have
\begin{align}\label{s1sjsj+1}
    S_{1}S_{j}-(j+1)S_{j+1}
   &=\sum_{m=1}^{k}\sum_{(i_{1}, \dots \hat{i}_{m} \dots,i_{k})\in I_{j-1}}\lambda^{i_{1}}_{1}\lambda^{i_{2}}_{2}\cdots\hat{\lambda}^{i_{m}}_{m}\lambda^{2}_{m}\cdots\lambda^{i_{k}}_{k}.
\end{align}
By comparing the above two equations \eqref{dsj} \eqref{s1sjsj+1} and using proposition \ref{odes-1},
we complete the proof of claim \ref{Sequation}.
\end{proof}
Then,  we derive the evolution equation of  $\xi=(\xi_{1},\dots, \xi_{k})$ which is defined below in terms of normalization of $(S_{1},\dots, S_{j})$:  \begin{equation}.
    \xi_{j}=\frac{1}{C_{k}^{j}}\beta^{-j}_{n,k}\tau^{j}S_{j}-1\quad \tau=-e^{-\sigma}\quad 1\leq j\leq k.
\end{equation}
Then, by direct computation we have
\begin{equation}\label{ODE}
     \frac{d}{d\sigma}\xi=B\xi+N(\sigma, \xi(\sigma))=B\xi-k\xi_{1}I_{k}\xi+O(e^{\frac{\gamma}{2}\sigma}),
\end{equation}
where the constant matrix $B$ is given by
\begin{equation}\label{B}
   \left(\begin{array}{ccccccccc}
        -(2k-1) & (k-1) & 0 &\dots &0& 0& 0&\dots &0\\
         -k & -(k-2) & (k-2)&\dots &0& 0& 0&\dots &0\\
         \dots & \dots &  \dots&\dots &\dots &\dots &\dots &\dots &0\\
         -k & 0 & \dots &  \dots&-(k-j) & (k-j)& 0 &\dots &0\\
          \dots & \dots &  \dots&\dots &\dots &\dots &\dots &\dots &0\\
         -k & 0 & 0  & \dots &0&\dots&0& -1 & 1\\
         -k & 0 & 0 & \dots &0 &\dots &0 &0 & 0\\
    \end{array}
    \right)\!,
\end{equation}
and $B$ has eigenvalues $-1, -2, \dots, -k$.  For this ODE system, since the error term is sum of quadratic term and exponential decay term, we can follow the proof of the usual Lyapunov linear stability theorem  in \cite[Thm 3.4.1]{ODE-book} to show that in a local neighborhood $B(0, \varepsilon)$ of $\xi=0$, where $\varepsilon>0$ is a small fixed constant, 
$\xi=0$ is a stable critical point of the above ODE system and any flow line of the ODE system with initial data in $B(0, \varepsilon)$ will converge to $0$ as $\sigma\to -\infty$.

Now, by the definition of $\kappa'$-quadratic at time $\tau_{0}\leq \tau_{*}$, if we choose $\kappa'>0$ small enough and $\tau_{*}$ negative enough after doing variable change, we can make $\xi(\sigma_{0})\in B(0, \varepsilon)$. Hence,   $\xi(\sigma)\in B(0, \varepsilon)$ for all $\sigma\leq \sigma_{0}$. Translating back to our original variables and choosing $\kappa'>0$ small enough and $\tau_{*}$ negative enough, this shows that both eigenvalues of the matrix $A$
are $\frac{C\eps}{|\tau|}$-close to $\frac{\sqrt{2(n-k)}}{4|\tau|}$. Hence, choosing $\eps=\eps(\kappa)$ sufficiently small, we conclude that
\begin{align}\label{u_strong_kappa}
     \left\| \hat{u}({\bf{y}},  \tau)+\frac{\sqrt{2(n-k)}(|{\bf{y}}|^2-2k)}{4|\tau|}  \right\|_{\mathcal{H}}\leq \frac{\kappa/2}{|\tau|}
\end{align}
holds for all $\tau\leq \tau_{0}$.

Finally, having established \eqref{u_strong_kappa}, we can consider the quantity
\begin{equation}
    \beta(\tau):=\sup_{\tau'\leq \tau}\left(\int_{\mathbb{R}^{k}}\hat{u}({\bf{y}},  \tau')^2e^{-\frac{|{\bf{y}}|^2}{4}}dy\right)^{1/2},
\end{equation}
and argue similarly as in \cite[Section 2.3]{DZ_spectral_quantization} (or see \cite[Section 2.2]{DH_hearing_shape} and the proof of Theorem \ref{lem-quadraticity}) to upgrade the initial graphical radius $\rho(\tau)=|\tau|^\gamma$ to the improved graphical radius $\rho(\tau)=|\tau|^{1/10}$ for $\tau\leq \tau_{0}$. Observing also that with this new  graphical radius \eqref{u_strong_kappa} still holds with $\kappa/2$ replaced by $\kappa$, this concludes the proof of the theorem.

\end{proof}

As a corollary of the proof we also obtain:

\begin{corollary}[full rank]\label{cor_full_rank}
For every $\kappa>0$ small enough, there exists $\tau_\ast>-\infty$ with the following significance. Let $\mathcal{M}$ be an ancient noncollapsed flow in $\mathbb{R}^{n+1}$, whose tangent flow at $-\infty$ is given by \eqref{bubble-sheet_tangent_intro}, and suppose that $\mathcal{M}$ is $\mathrm{SO}(n+1-k)$-symmetric in the $x_{k+1} \dots x_{n+1}$-plane centered at the origin. If $\mathcal{M}$ is $\kappa$-quadratic  at some time $\tau_0\leq\tau_\ast$ (defined literally the same as in Definition \ref{k_tau00_intro}), then its fine cylindrical matrix $Q$ satisfies $\mathrm{rk}(Q)=k$.
\end{corollary}

\begin{proof}
Indeed, this follows by inspecting the above proof of Theorem \ref{point_strong} (strong $\kappa$-quadraticity).
\end{proof}

Together with the results from the previous subsection, we  get:

\begin{proof}[{Proof of Theorem \ref{strong_uniform0} (uniform sharp asymptotics)}]
This now follows from Proposition
\ref{strong_sharp_asymptotics}, which establishes the uniform asymptotics under the a priori of strong $\kappa$-quadraticity, together with Theorem \ref{point_strong} (strong $\kappa$-quadraticity), which justifies this a priori assumption.
\end{proof}

To conclude this section, we note that as a consequence of Theorem \ref{strong_uniform0} (uniform sharp asymptotics) we obtain the following standard cylindrical estimates, which will be used frequently throughout the paper:

\begin{corollary}[cylindrical estimate]
\label{lemma-cylindrical}
 For every $\varepsilon > 0$, there exist $\kappa > 0$  and $\tau_* > -\infty$
such that if $\mathcal{M}$ is $\kappa$-quadratic at time $\tau_0 \le \tau_*$, then for all $\tau\leq \tau_0$ we have
\begin{equation}\label{eqn-cyl0}  
\sup_{\{v(\cdot,\tau)\geq L/\sqrt{|\tau|}\}}\max_{1\leq m+\ell \leq 10} \left| v^{m+\ell-1} y ^ {-m} \nabla^m_{S^{k-1}} \partial ^\ell_y v\right | < \varepsilon.
\end{equation}
\end{corollary}
\begin{proof}
By convexity and Theorem \ref{strong_uniform0} (uniform sharp asymptotics), for every $\eps_1>0$ there exist $L_1<\infty$, such that for $\tau\leq\tau_0$ sufficiently negative we have
\begin{equation}\label{first_der_cyl_st}
\sup_{\{v(\cdot,\tau)\geq L_1/\sqrt{|\tau|}\}} |\partial_y v|+ \sup_{v(\cdot,\tau)\geq \eps_1}|y^{-1}\nabla_\omega v| \leq \eps_1,
\end{equation}
Then for given $v$ and $\tau$, we work with the convex hypersurface in $\mathbb{R}^{k}$ represented by $r(\omega):= {(2|\tau|)}^{-1/2} Y(v,\omega,\tau)$ where $Y$ inverse profile function  defined by  
\begin{equation}
    Y(v(y, \omega, \tau), \omega, \tau)=y,
\end{equation}
and argue similarly as in the proof of tip asymptotics in Proposition \ref{strong_sharp_asymptotics} to conclude
\begin{equation}\label{eqn-Yphi0}
\sup _{v\leq 2\theta}\sup_{\omega\in S^{k-1} } \, \big | \nabla_{\omega}Y (v,\omega,\tau)\big |<\varepsilon_1 \, \sqrt{|\tau|},
\end{equation}
provided  $\theta >0$ is chosen sufficiently small and $\tau \leq \tau_0$ is sufficiently negative. Hence, both suprema in \eqref{first_der_cyl_st} can be taken over $\{v(\cdot,\tau)\geq L_1/\sqrt{|\tau|}\}$. 

The higher order derivative estimates follows from a blowup contradiction argument and first order derivative estimates above as in the proof of \cite[Corollary 2.11]{CDDHS}.
\end{proof}

\section{Collar asymptotics and quadratic almost concavity}\label{Quadratic_almost_concavity_section}

The goal of this section is to prove the quadratic almost concavity estimate and its corollary. Throughout this section, it will be most convenient to work with the original flow $M_t$. Thanks to the tangent flow property \eqref{bubble-sheet_tangent_intro} and the $\mathrm{SO}(n+1-k)$-symmetry we can then parametrize our $k$-ovals via
\begin{equation}
(x^1,\ldots,x^k,\vartheta)\mapsto (x^1,\ldots,x^k,V(x^1,\ldots,x^k,t)\vartheta).
\end{equation}
Here $\vartheta \in{S}^{n-k}$ and $x=(x^1,\ldots,x^k)\in \mathbb{R}^k$. Let us denote local coordinates of $S^{n-k}$ by $(\vartheta^1,\ldots,\vartheta^{n-k})$ and the standard round metric by $ g^s_{\alpha\beta} d\vartheta^\alpha  d\vartheta^\beta  $.

In these $(x,\vartheta)$-coordinates the metric takes the form
\be g= (\delta_{ij}+V_iV_j)dx^idx^j + V^2 g^s_{\alpha\beta}d\vartheta^\alpha d\vartheta^\beta. \ee 
Hence, the inverse metric is
\be \label{eq-ginverse} g^{-1} = (\delta_{ij}-\eta^{-2}V_iV_j)\partial_{x^i} \partial _{x^j}  + V^{-2} (g^s)^{\alpha \beta } \partial_{\vartheta^\alpha}  \partial_{\vartheta^\beta}, \ee 
where 
\be
V_i=\partial_{x^i}V,\quad |DV|^2:=V_{1}^2 + V_{2}^2+\cdots+V_{k}^2 .
\ee
Furthermore, observing that the outward unit normal is
\begin{equation}
\nu=\frac{(-V_{1},\ldots,  -V_{k},\vartheta )}{\sqrt{1+|DV|^2}} \in \mathbb{R}^{n+1}\, ,
\end{equation}
we see that the second fundamental form is given by
\be \label{eq-h}h_{ij} = -\eta^{-1} \partial^2_{ij} V+\eta Vg^{s}_{ij}.
\ee 
In particular, convexity of our hypersurfaces $M_t$ is now captured by the analytic condition that $x \mapsto V(x,t)$ is concave and nonnegative.\footnote{Moreover, note that $-(1+|DV|^2)^{1/2}\textrm{tr}_g h=\Big(\delta_{ij}-\frac{V_{x_i}V_{x_j}}{1+|DV|^2}\Big)V_{x_ix_j}-\frac{1}{V}$, which is of course consistent with the evolution equation \eqref{equation_u} for the renormalized profile function.}

Finally, observing that
\be\label{obs_grad}
|\nabla V|^2 :=g^{ij}  V_{i} V_{j} = \frac{|DV|^2}{1+ |DV|^2},
\ee
throughout this section we will abbreviate
\be\label{abbr_gra}
\gra := (1+|DV|^2)^{1/2}= (1-|\nabla V|^2)^{-1/2}.
\ee

In this section, we aim to prove
\begin{theorem}[quadratic almost concavity]\label{prop-concavity}
There exist constants $\kappa>0$ and $\tau_*>-\infty$ with the following significance. If  $\mathcal{M}$ is $\kappa$-quadratic at time $\tau_0 \le \tau_*$, then for all $\tau\leq\tau_0$ we have
\begin{equation}\label{almost_quadratic concavity_estimates}
\left[{\nabla^2} Q- (\frac{\sqrt{|t|}}{\sqrt{Q\log |t|}})^{3}g-\frac{|\nabla Q|^2}{4Q}\left(\frac{\nabla Q\otimes\nabla Q}{2Q}-(\frac{\sqrt{|t|}}{\sqrt{Q\log |t|}})^{3}g\right)\right]\!\!(X, X)\leq 0
\end{equation}
for all vector field $X$ perpendicular to all spherical vector fields $\partial{\vartheta^{\alpha}}$ along $S^{n-k}$ factor, where $\alpha=1,\dots, n-k$.
\end{theorem}    
\subsection{Intrinsic quantities and their evolution}\label{sec_intrinsic}
 Throughout this subsection, we will view the unrescaled profile function as an intrinsic time-dependent quantity on $S^n$. Specifically, denoting the parametrized mean curvature flow by $F_t:S^n\to\mathbb{R}^{n+1}$, so that $M_t=F_t(S^n)$, we set
 \begin{equation}\label{intrinsic_prof}
 V(p,t):=F_t(p)\cdot {\vartheta}\left(F_t(p)\right),
 \end{equation}
 where $\vartheta$ denotes the radial vector field in $\mathbb{R}^{n+1}$ defined by
 \begin{equation}
 {\vartheta}(x,y):=\big(0,{y}/{|y|}\big)\quad \text{for } x\in\mathbb{R}^k \text{ and } y\in \mathbb{R}^{n+1-k}.
 \end{equation}

We will now compute the evolution equations of several intrinsic quantities. Throughout, we will briefly denote by $\Delta=\Delta_{g(t)}$ the Laplace-Beltrami operator with respect to the metric $g(t)$ induced by the embedding $F_t$.

  Given any $\delta>0$, we now consider the tensor
 \be\label{Tensor A} A_{ij} =\nabla_{ij}^2Q -\varepsilon g_{ij}-\frac{1}{2}|\nabla V|^{2}(Q^{-1}\nabla_i Q \nabla_j Q-2\varepsilon g_{ij}),\ee
where
\begin{equation}
 \label{eq-Q}
Q:=V^2 \quad \text{and}\quad \varepsilon=\gamma+\delta\quad \text{and} \quad \gamma:= \left(\frac{-t}{\log(-t)}\right)^{3/2} V^{-3}.
 \end{equation}
The tensor $A_{ij}$  will be used in  the proof of Theorem \ref{prop-concavity_restated} in the next subsection.  

As usual in tensor computations, we use the extended summation convention, where indices are raised using the metric and summed over, e.g. $h_{ip}h_{pm}=\sum_{j,\ell=1}^3 h_{ij}g^{j\ell} h_{\ell m}$.
We will now compute the evolution equations of several intrinsic quantities. Throughout, we will briefly denote by $\Delta=\Delta_{g(t)}$ the Laplace-Beltrami operator with respect to the metric $g(t)$ induced by the embedding $F_t$.

In the $(x,\vartheta)$-coordinate, $\nabla^2 V$ is a $n\times n$ matrix that has a block diagonal structure as follows:
\be \label{eq-hessian V} \nabla^2  V = (1+|DV|^2)^{-1} \begin{bmatrix}
 V_{x_ix_j} & 0 \\
 0 & |DV|^2 V {g^s_{\alpha\beta} } 
\end{bmatrix}.\ee 
Using this, we write the second fundamental form $h$ in terms of $\nabla^2 V$ as 
\be \label{eq-equivalent} h(X,Y) = -\eta \nabla^2 V (X,Y) + \eta V  { \nabla _X \vartheta^\alpha \nabla_Y \vartheta^\beta g^s_{\alpha\beta } }.\ee 

\begin{proposition}[evolution of profile function]
\label{lemma-ev-eq}
The profile function $V$, considered as an intrinsic quantity, satisfies the evolution equation
\be
 (\partial _t - \Delta ) V= - (n-k)V^{-1}.
 \ee
 \end{proposition}
 \begin{proof}
     The proof is similar as \cite[Prop 3.2]{CDDHS}.  Under the flow we have $\partial _t \vartheta(F_t)= 0$ due to the symmetry. Also, for any function $f$ that depends only on  $\vartheta$, we have $\Delta f =  V^{-2} \Delta_{S^{k-1}} f $. Hence, \be (\partial_t - \Delta) \vartheta(F_t)  = (n-k)\vartheta(F_t) V^{-2} .   \ee
Together with the mean curvature flow equation $\partial_t F_t  = \Delta F_t$, this yields
 \be
 (\partial_t - \Delta ) (F_t\cdot \vartheta(F_t) )
	=  F_t\cdot \vartheta(F_t)V^{-2} - 2 g^{\alpha\beta}  \partial_\alpha F_t\cdot \partial_\beta\vartheta(F_t)=-(n-k)V^{-1}.
\ee	
This proves the proposition.
 \end{proof}
 Using the above evolution equation, we obtain:
\begin{proposition}[evolution equation of tensor $A$]\label{evolution of new A main}
   Given any $\delta>0$, for $Q=V^2$ we now consider the tensor
\begin{align} \label{eq-tensorA}
     A_{ij} &=\nabla_{ij}^2Q-\varepsilon g_{ij} -2|\nabla V|^2\nabla_{i}V\nabla_{j}V+|\nabla V|^2\varepsilon g_{ij}.
\end{align}
Then $A$ satisfies
   \bea \label{eq-Aij-2} (\partial_t - \Delta + \nabla_Z)& {A}_{ij}=  N_{ij}\, ,\\
\eea 
where $Z$ is defined by  \be \label{eq-defZ} \langle Z,Y\rangle=Q^{-1}\nabla_Y Q+2\eta h(\nabla V, Y)\ee  for all $Y$, and 
\begin{align}\label{N=N1+N2+N3}
    N_{ij} = (N_1)_{ij} + (N_2)_{ij} + (N_3)_{ij}\, , 
\end{align}
where 
\bea  \label{eq-N1-2}
    {(N_1)}_{ij} &= -\frac{n-k}{2}Q^{-2}|\nabla Q|^2\left(Q^{-1}\nabla_i Q \nabla_j Q-\varepsilon g_{ij}\right) \\
    &-Q^{-3}|\nabla  Q|^2\nabla_i  Q \nabla_j Q - Q^{-1}\nabla^2_{ik}Q  \nabla^2_{jk} Q\\
&+
\tfrac{1}{2} Q^{-2} |\nabla  Q|^2 \nabla^2_{ij} Q
  + Q^{-2}(\nabla_i Q\nabla_k Q \nabla^2_{jk} Q  + \nabla_j Q\nabla_k Q  \nabla^2_{ik}Q )\\
    &+ \left(Q^{-1}\nabla_i Q \nabla_j Q-2\varepsilon g_{ij}\right)|\nabla^2 V|^2 + 2\eta^{-2}\varepsilon H h_{ij} + B_{ij}\, , 
\eea 
with 
\begin{align}
    B_{ij} &= (h_{ij} h_{kp}-h_{ip} h_{jk})(2\nabla^2_{kp}Q-Q^{-1}\nabla_k Q\nabla_p Q)\nonumber  \\
  & +(h_{ik}h_{kp}-h_{ip}H) (\nabla^2_{jp}Q-8^{-1}Q^{-2}|\nabla Q|^2\nabla_{p}Q\nabla_{j}Q)\nonumber\\
  &+(h_{jk}h_{kp}-h_{jp}H) (\nabla^2_{ip}Q-8^{-1}Q^{-2}|\nabla Q|^2\nabla_{i}Q\nabla_{p}Q) \, , \nonumber  
\end{align}
\bea \label{eq-N2}    {(N_2)}_{ij} &= 
    -\frac{1}{4}Q^{-1}h^2(\nabla Q,\nabla Q)\left(Q^{-1}\nabla_i Q \nabla_j Q-2\varepsilon g_{ij}\right)\\
    &+Q^{-1}|\nabla Q|^2\nabla^2_{im}V \nabla^2_{jm}V  +\Psi_{ij} +4\nabla_{\ell}(|\nabla V|^2)\nabla_{\ell}(\nabla_{i}V\nabla_{j}V)\\
    &+\nabla_{Z}(-2|\nabla V|^2\nabla_{i}V\nabla_{j}V)+\nabla_Z(|\nabla V|^2)\varepsilon g_{ij} \, ,
\eea 
 with 
\begin{align}\label{eq-Psi3}
   \Psi_{ij} =  &4\eta\nabla_p V h_{pk}(\nabla_k V \nabla^2_{ij}V+\nabla_i V\nabla^2_{jk}V + \nabla_j V \nabla^2_{ik}V)\\
    &-4\nabla_p V h_{pk}(\eta^3 V\nabla^2_{kq}V \nabla_q V \nabla^2_{ij}V-VG_{kij}),\notag
\end{align}
for some $(0,3)$-tensor $G$ satisfying $G_{kij}X^iX^j=0$ for any $X$ perpendicular to the spherical direction, and  
\begin{align}
    (N_3)_{ij}
    =& - 3\eta^{-2}V^{-2}\Big [(n-k)-6|\nabla V|^2 +\tfrac{V^2}{2t}\Big (1-\tfrac{1}{\log(-t)}\Big ) \Big ] \gamma g_{ij}\nonumber\\
    &-6  {\eta^{-1}}V^{-1}  h(\nabla V,\nabla V)  \gamma g_{ij} .\nonumber 
\end{align}

\end{proposition}
\begin{proof}
    See proof of Proposition \ref{evolution of new A} in Appendix \ref{evolution equation appendix} for details.
\end{proof}
\subsection{Proof of the quadratic almost concavity estimate} \label{sec-apriori}
 
 We will now prove the quadratic almost concavity estimate and its corollary. Similarly as in the previous subsection, we work with the tensor
  \be \label{eq-defAijrest}A_{ij} =\nabla_{ij}^2Q -\varepsilon g_{ij}-\frac{1}{2}|\nabla V|^{2}(Q^{-1}\nabla_i Q \nabla_j Q-2\varepsilon g_{ij}),\ee
 where $Q=V^2$ is the square of the profile function viewed as an intrinsic function. We have seen in Proposition \ref{evolution of new A main} (evolution of $A$-tensor) that $A$ satisifies an evolution equation of the form
\begin{equation}
\label{eq-right-form}
(\partial_t - \Delta+\nabla_Z) A =  N.
\end{equation}
where $N=N_1+N_2+N_3$. We will show that $A(X,X)\leq 0$ for all $X$ perpendicular to all spherical vector fields $\partial_{\vartheta^{\alpha}}$ (from now on we simply denote this condition by $X\perp \partial_{\vartheta^{\alpha}}$) is preserved along the flow. To this end, we will first estimate the reaction term when evaluated on null eigenvectors of $A$, away from a certain region, then check the sign in the region that has been excluded in the first step, and then conclude by adapting the proof of Hamilton's tensor maximum principle  \cite{Ham_pco} to our setting. 

\begin{lemma}[null eigenvector]\label{lemma_null_eigen}
If $X\perp \partial_\vartheta$ is a null eigenvector of $A$, then\begin{equation}\label{null_useful_inequ}
2\eps |X|^2 \leq Q^{-1}|\nabla_X Q|^2.
\end{equation}Moreover, there hold the following identities:\footnote{Recall that throughout this section, $\eta$ denotes the quantity defined in \eqref{abbr_gra}.}
\begin{align}\label{eq-h-ident}
h(X,X)= \frac{1}{4}\eta^{-1} V^{-1}\left(Q^{-1}|\nabla_X Q|^2-2\eps |X|^2\right),
\end{align}
\begin{align}
    h(X,\nabla V) = \frac{1}{4}\eta^{-1}V^{-1}\left(Q^{-1}|\nabla Q|^2-2\varepsilon\right)\nabla_X V ,
\end{align}
\be\label{eq-h2-ident}
h^2(X,\nabla V)= {\eta^{-1}}V^{-1}\left(h(\nabla V,\nabla V) - \frac{1}{8}\varepsilon 
{\eta^{-1}} V^{-1}\left(Q^{-1} |\nabla Q|^2- 2\eps\right) \right)\nabla_X V. 
\ee
\end{lemma}
\begin{proof}
By  the null eigenvector assumption and substituting $Q=V^2$, we thus get
\begin{equation}\label{eq_null_eigen_restated}
\nabla^2V(X,Y) = \eta^{-2}V^{-1}\left(\tfrac{1}{2}\eps \langle X,Y\rangle - \nabla_X V \nabla_Y V\right)\, .
\end{equation}
Since $\nabla^2 V(X,X)\leq 0$ thanks to the assumption $X\perp\partial_\vartheta$, this implies 
\begin{equation}
2\eps |X|^2\leq Q^{-1}|\nabla_X Q|^2,
\end{equation}
 which yields \eqref{null_useful_inequ}. Note that $X$ and $\nabla V$ are perpendicular to $\partial_\vartheta$ and hence $h(X,X)= - \eta \nabla^2 V(X,X)$ and $h(X,\nabla V)=-\eta \nabla^2 V(X,\nabla V)$ by \eqref{eq-equivalent} (second fundamental form). By plugging $X$ and $\nabla V$ into  $Y$ in \eqref{eq_null_eigen_restated}, we get the first two identities. 
Finally, arguing similarly we compute
\begin{align}
h^2(X,\nabla V) & =\tfrac{1}{2}\eps   V^{-1} \nabla^2V(X,\nabla V)+ V^{-1} \eta^{-1} h(\nabla V,\nabla V) \nabla_X V\nonumber\end{align} and \eqref{eq-h2-ident} follows by 
\begin{equation}
 \nabla^2V(X,\nabla V)=V^{-1}\eta^{-2} \left(\tfrac12 \eps -  |\nabla V|^2\right) \nabla_X V.
\end{equation}
\end{proof}
Then we have the following key estimate for the reaction term.
\begin{proposition}[reaction term] \label{prop-reactionterm}
\label{N(X, X)_estimates}
 there exist constants $\kappa>0$, $\tau_*>-\infty$ and $L<\infty$, which are all independent of the parameter $\delta$, with the following significance. If $\mathcal{M}$ is  $\kappa$-quadratic at time $\tau_0 \le \tau_*$, then for all times $t\le -e^{-\tau_0}$ in the region $\{ L\sqrt{|t|/\log|t|}\leq V\leq \sqrt{2|t|}-L\sqrt{|t|}/\log|t| \}$ we have the following: If $X\perp \textrm{span}\{\partial \theta^{\alpha}\}_{\alpha=1}^{n-k}$ is a null vector of $A$ then
\be      N(X,X)\leq -\frac{n-k+2}{4 } Q^{-2}|\nabla_X Q|^2(Q^{-1}|\nabla Q|^2-2\varepsilon) -\frac{1}{8}Q^{-3}|\nabla_X Q|^2|\nabla Q|^2. \ee 
\begin{proof}
We break this into digestible lemmas below. Lemma \ref{lemma-estN1}, Lemma \ref{lemma-estN2} and Lemma \ref{lemma-estN3} address estimates on $N_i(X,X)$, for $i=1,2,3$, respectively. The proposition directly follows by adding up three estimates.  
\end{proof}

\end{proposition}

 In these lemmas, we assume the condition in Proposition  \ref {N(X, X)_estimates}. In the proof, the factor $o(1)$ (and $O(1)$) means we can make the factor arbitrarily small (and bounded, respectivly) by choosing sufficiently small $\kappa>0$, $\tau_*>-\infty$, and large $L<\infty$.

\begin{lemma}[estimate on $N_1(X,X)$] \label{lemma-estN1}
  \begin{align}
    N_1(X,X) \leq &- \frac{n-k+2}{4}Q^{-2}|\nabla_X Q|^2(Q^{-1}|\nabla Q|^2-2\varepsilon) \\ &  -\left(\frac{1}{4}-o(1)\right)Q^{-3}|\nabla_X Q|^2 |\nabla Q|^2\nonumber.
\end{align}

\begin{proof}For convenience of readers, we first recall definition of $N_1$, \bea 
    {(N_1)}_{ij} &= -\frac{n-k}{2}Q^{-2}|\nabla Q|^2\left(Q^{-1}\nabla_i Q \nabla_j Q-\varepsilon g_{ij}\right) \\
    &-Q^{-3}|\nabla  Q|^2\nabla_i  Q \nabla_j Q - Q^{-1}\nabla^2_{ik}Q  \nabla^2_{jk} Q\\
&+
\tfrac{1}{2} Q^{-2} |\nabla  Q|^2 \nabla^2_{ij} Q
  + Q^{-2}(\nabla_i Q\nabla_k Q \nabla^2_{jk} Q  + \nabla_j Q\nabla_k Q  \nabla^2_{ik}Q )\\
    &+ \left(Q^{-1}\nabla_i Q \nabla_j Q-2\varepsilon g_{ij}\right)|\nabla^2 V|^2 + 2\eta^{-2}\varepsilon H h_{ij} + B_{ij}\, , 
\eea 
with 
\begin{align}
    B_{ij} &= (h_{ij} h_{kp}-h_{ip} h_{jk})(2\nabla^2_{kp}Q-Q^{-1}\nabla_k Q\nabla_p Q)\nonumber  \\
  & +(h_{ik}h_{kp}-h_{ip}H) (\nabla^2_{jp}Q-8^{-1}Q^{-2}|\nabla Q|^2\nabla_{p}Q\nabla_{j}Q)\nonumber\\
  &+(h_{jk}h_{kp}-h_{jp}H) (\nabla^2_{ip}Q-8^{-1}Q^{-2}|\nabla Q|^2\nabla_{i}Q\nabla_{p}Q) \, , \nonumber  
\end{align}

By Lemma \ref{lemma_null_eigen} (null eigenvector), {we first have
\begin{align}
     & -2(h_{ij}h_{kp}-h_{ip}h_{jk})\eta^{-2}\varepsilon g_{kp}X^{i}X^{j}\\
  &=(h_{jk}h_{kp}-h_{jp}H) (\nabla^2_{ip}Q-8^{-1}Q^{-2}|\nabla Q|^2\nabla_{i}Q\nabla_{p}Q) X^{i}X^{j}\nonumber\\
  &+(h_{ik}h_{kp}-h_{ip}H) (\nabla^2_{jp}Q-8^{-1}Q^{-2}|\nabla Q|^2\nabla_{p}Q\nabla_{j}Q)X^{i}X^{j}\nonumber.  
\end{align} 
This together with $\eta = 1+ o(1)$ implies}
\be \label{eq-Hh}2 \eta^{-2} \eps H h(X,X) = \frac{n-k}{2} \eps (1+o(1)) Q^{-1} (Q^{-1} |\nabla_{X} Q|^2 - 2\eps |X|^2).\ee 

Then, by definition of $B$ and above, we have
\begin{align}
    B_{ij}X^iX^j&= (h_{ij} h_{kp}-h_{ip} h_{jk})(2\nabla^2_{kp}Q-Q^{-1}\nabla_k Q\nabla_p Q) X^iX^j \\
  & +(h_{ik}h_{kp}-h_{ip}H) (\nabla^2_{jp}Q-8^{-1}Q^{-2}|\nabla Q|^2\nabla_{p}Q\nabla_{j}Q)X^iX^j\nonumber\\
  &+(h_{jk}h_{kp}-h_{jp}H) (\nabla^2_{ip}Q-8^{-1}Q^{-2}|\nabla Q|^2\nabla_{i}Q\nabla_{p}Q)X^iX^j \nonumber\\
  &=(h_{ij} h_{kp}-h_{ip} h_{jk})(2\nabla^2_{kp}Q-Q^{-1}\nabla_k Q\nabla_p Q-2\eta^{-2}\varepsilon g_{kp}) X^iX^j. \nonumber\\
\end{align}
Note $(2\nabla^2 Q-Q^{-1}\nabla Q\nabla  Q-2\eta^{-2}\varepsilon g) = 2
(2V \nabla ^2 V - \eta^{-2} \eps g )$.
Let $e_1,...,e_k$ be an orthonormal basis that are perpendicular to the spherical direction and such that $\nabla^2_{ij}V = \lambda_i \delta_{ij}$ (recall $\lambda_i<0$).
Then as in  \cite[(3.70)]{CDDHS}, 
\begin{align}\label{eq-BXX}
    B(X, X)& =\sum_{i=1}^{k}2(h(X, X)h(e_{i}, e_{i})- h(X, e_{i})^{2})(2V\lambda_{i}-\eta^{-2}\varepsilon)\nonumber\\
    &+{\frac{n-k}{4}}Q^{-1} \left(Q^{-1}|\nabla_X Q|^2-2\eps  |X|^2\right)(Q^{-1}|\nabla Q|^2-2\eta^{-2}\varepsilon).
\end{align}
Let us estimate the first line of \eqref{eq-BXX}.  Under the assumption on the existence of null eigenvector at $\tau\le \tau_0\le \tau_*$,  \eqref{null_useful_inequ} implies 
\be \label{eq-epsest} 0\le  \eps \le 2^{-1}  Q^{-1}|\nabla Q|^2 =2 |\nabla V| ^2=o(1)  .\ee 
By  \eqref{eq_null_eigen_restated},  \eqref{null_useful_inequ} and \eqref{eq-epsest}, 
\bea  \label{eqh(X,e_i)^2} \sum_{i=1}^k h(X,e_i)^2 = \eta ^{-2} V^{-2} \left| \frac{1}{2}\eps  X - \nabla_X V \nabla V\right|^2=O(1)Q^{-3}|\nabla Q|^2|\nabla_X Q|^2  .  \eea 
By \eqref{eq-h-ident}, $\sum_{i=1}^k h(e_i,e_i) \lambda_i =  - \eta |\nabla^2 V|^2$,    \eqref{eqh(X,e_i)^2},  \eqref{eq-epsest}, and  $V\lambda_i =o(1)$,   
\bea  \label{eq-lasttermcross}
    &\sum_{i=1}^{k} 2(h(X, X)h(e_{i}, e_{i})- h(X, e_{i})^{2})(2V\lambda_{i}-\eta^{-2}\varepsilon)\\
    =& 2h(X,X)\sum_{i=1}^{k}h(e_i,e_i)(2V\lambda_i -\eta^{-2}\varepsilon)-2\sum_{i=1}^{k}h(X,e_i)^2(2V\lambda_i-\varepsilon)\nonumber\\
    \leq & -4\eta V h(X,X)|\nabla^2 V|^2 - 2\sum_{i=1}^{k}h(X,e_i)^2(2V\lambda_i-\varepsilon) \nonumber\\
    \leq & -(Q^{-1}|\nabla_X Q|^2-2\varepsilon |X|^2)|\nabla^2 V|^2 \nonumber\\
     &+o(1)\eta^{-2}V^{-2}\left(\frac{1}{4}\varepsilon^2 |X|^2 + |\nabla_X V|^2|\nabla V|^2 - \varepsilon |\nabla_X V|^2\right) \nonumber\\
     \leq&-(Q^{-1}|\nabla_X Q|^2-2\varepsilon |X|^2)|\nabla^2 V|^2\nonumber\\
&+o(1)\sum_{i=1}^{k}\eta^{-2}V^{-2}\left(\frac{1}{4}\varepsilon^2 \langle X, e_i\rangle ^2 + |\nabla_X V|^2\langle\nabla V, e_i\rangle^2 + \varepsilon |\nabla_X V|^2\right)\nonumber\\
    \leq & -(Q^{-1}|\nabla_X Q|^2-2\varepsilon |X|^2)|\nabla^2 V|^2
    + o(1)Q^{-3}|\nabla Q|^2|\nabla_X Q|^2\nonumber.
\eea 
We  similarly consider remaining terms in $N_1(X,X)$. By the null eigenvector condition, we estimate
\begin{align}\label{eq-N11}
  & \quad Q^{-1}\nabla^2_{ik}Q \nabla^2_{jk}Q X^i X^j = \varepsilon^2 Q^{-1} |X|^2 + O(Q^{-4}|\nabla_X Q|^2|\nabla Q|^4),
\end{align}
\begin{align}\label{eq-N12}
    &\frac{1}{2}Q^{-2}|\nabla Q|^2 \nabla^2_{ij}Q X^i X^j\\
    =& \frac{1}{2}Q^{-2}|\nabla Q|^2(\varepsilon |X|^2 + 2|\nabla V|^2|\nabla_X V|^2 - \varepsilon |\nabla V|^2|X|^2) \\
    =& \frac{1}{2}\varepsilon Q^{-2}|\nabla Q|^2 |X|^2 + O(Q^{-4}|\nabla_X Q|^2|\nabla Q|^4) \nonumber,
\end{align}
\begin{align}\label{eq-N13}
    &Q^{-2}(\nabla_i Q\nabla_k Q \nabla^2_{jk} Q  + \nabla_j Q\nabla_k Q  \nabla^2_{ik}Q )X^iX^j\\ 
    &= 2Q^{-2}\nabla_X Q \langle \nabla Q, \varepsilon X + 2|\nabla V|^2
    \nabla_X V \nabla V -\varepsilon |\nabla V|^2 X\rangle \\
    &= 2Q^{-2}|\nabla_X Q|^2 \left(\varepsilon +2|\nabla V|^4-\varepsilon |\nabla V|^2\right) \nonumber\\
    &= 2\varepsilon Q^{-2}|\nabla_X Q|^2 + O(Q^{-4}|\nabla_X Q|^2|\nabla Q|^2) \nonumber.
\end{align}

Adding up the estimates \eqref{eq-Hh},\eqref{eq-BXX},\eqref{eq-lasttermcross},\eqref{eq-N11}, \eqref{eq-N12}, \eqref{eq-N13}  with the other (two) remaining terms of $N_1(X,X)$ in above or \eqref{eq-N1},  
\bea \label{eq-N1reaction}
    N_1(X,X) \le &-\frac{n-k}{2}Q^{-2}|\nabla Q|^{2} (Q^{-1}|\nabla_X Q|^2-\varepsilon|X|^2) \\
    &-Q^{-1}\left(Q^{-1}|\nabla Q|^2-2\varepsilon\right)\left(Q^{-1}|\nabla_X Q|^2- 2^{-1} \varepsilon |X|^2\right) \\
    &+\frac{n-k}{4}Q^{-2}|\nabla Q|^2\left(Q^{-1}|\nabla_X Q|^2-2\eps|X|^2\right)\\
&+ o(1)Q^{-3}|\nabla Q|^2|\nabla_X Q|^2. 
\eea 
Here we used  
$\eps  Q^{-1} (Q^{-1} |\nabla_{X} Q|^2 - 2\eps |X|^2)= O(1) Q^{-3}|\nabla_X Q|^2 |\nabla Q|^2  $ and $Q^{-4}|\nabla_X Q|^2|\nabla Q|^{4}= o(1)Q^{-3}|\nabla_X Q|^2|\nabla Q|^2$.

 Note, by \eqref{null_useful_inequ}, the terms in the second line of \eqref{eq-N1reaction} is bounded (from above) by $-\frac 34 Q^{-2} |\nabla_XQ|^2 (Q^{-1}|\nabla Q|^2- 2\eps )$.  Rearranging the terms on the right hand side of \eqref{eq-N1reaction}, we obtain the desired estimate. 
\end{proof}
\end{lemma}

\begin{lemma}[estimate on $N_2(X,X)$] \label{lemma-estN2}
    \begin{equation}
         N_2(X,X) \leq o(1)Q^{-3}|\nabla_X Q|^2|\nabla Q|^2.
    \end{equation}
\begin{proof} We start by estimating the last term in $N_2$ in \eqref{eq-N2}.  For convenience of reader, we recall that
\begin{align}
    {(N_2)}_{ij} &= 
    -\frac{1}{4}Q^{-1}h^2(\nabla Q,\nabla Q)\left(Q^{-1}\nabla_i Q \nabla_j Q-2\varepsilon g_{ij}\right)\nonumber\\
    &+Q^{-1}|\nabla Q|^2\nabla^2_{im}V \nabla^2_{jm}V  +\Psi_{ij} +4\nabla_{\ell}(|\nabla V|^2)\nabla_{\ell}(\nabla_{i}V\nabla_{j}V)\nonumber\\
    &+\nabla_{Z}(-2|\nabla V|^2\nabla_{i}V\nabla_{j}V+|\nabla V|^2\varepsilon g_{ij})\nonumber.
\end{align} 
Hence
\begin{align}
    N_2(X,X) &= -\frac{1}{4}Q^{-1}h^2(\nabla Q,\nabla Q)\left(Q^{-1}|\nabla_X Q|^2-2\varepsilon|X|^2\right)\nonumber\\
    &+Q^{-1}|\nabla Q|^2|\nabla^2 V(X)|^2  +\Psi(X,X)\\
    &+4\nabla_{\ell}(|\nabla V|^2)\nabla_{\ell}(\nabla_{i}V\nabla_{j}V)X^iX^j\nonumber\\
    &+\nabla_{Z}(-2|\nabla V|^2\nabla_{i}V\nabla_{j}V+|\nabla V|^2\varepsilon g_{ij})X^iX^j\nonumber.
\end{align}
Recalling the vector field $Z$ in \eqref{eq-defZ}, note that 
\bea  \nabla _Z (\nabla_i V \nabla_j V) &=-\eta^{-1}(h_{ki}\nabla_j V + h_{kj}\nabla_i V)(Q^{-1}\nabla_k Q + 2\eta \nabla_p V h_{pk}),\\ 
        \nabla_Z|\nabla V|^2 
        &=-4\eta^{-1}V^{-1}h(\nabla V, \nabla V)-4h^2(\nabla V, \nabla V)
    .\eea 
Using these and Lemma \ref{lemma_null_eigen} (null vector condition), we have
\begin{align}
    &\quad\nabla_{Z} (-2|\nabla V|^2\nabla_iV\nabla_j V)X^iX^j\\
    &= 8\left[\eta^{-1}V^{-1}h(\nabla V, \nabla V)+h^2(\nabla V, \nabla V)\right]|\nabla_X V|^2\nonumber\\
    &+ \frac{1}{8}\eta^{-2}Q^{-3}|\nabla Q|^2|\nabla_X Q|^2(Q^{-1}|\nabla Q|^2-2\varepsilon)\nonumber\\
    &+ \frac{1}{2}Q^{-3}|\nabla Q|^2|\nabla_X Q|^2\left( {\eta^{-1}}V h(\nabla V, \nabla V)-\frac{1}{8}\varepsilon\eta^{-2}(Q^{-1}|\nabla Q|^2-2\varepsilon)\right)\nonumber
\end{align}
and 
\begin{align}
    &\quad\nabla_{Z} (|\nabla V|^2)\varepsilon g_{ij}X^iX^j= -\left[4\eta^{-1}V^{-1}h(\nabla V, \nabla V)+4h^2(\nabla V, \nabla V)\right]\varepsilon |X|^2. \nonumber
\end{align}
Recall that $|Vh|$, $|\nabla V|=o(1)$.  Combining these with \eqref{null_useful_inequ} and \eqref{eq-epsest}, 
\[\nabla_{Z}(-2|\nabla V|^2\nabla_{i}V\nabla_{j}V)X^iX^j + \nabla_Z(|\nabla V|^2) \varepsilon |X|^2 =o(1)Q^{-3}|\nabla_XQ|^2|\nabla Q|^2. 
\]
Next, \bea 
        &4\nabla_{\ell}(|\nabla V|^2)\nabla_{\ell}(\nabla_i V \nabla_j V)X^iX^j = 16\eta^{-2}h^2(X,\nabla V)\nabla_X V , \eea 
and by Lemma \ref{lemma-nabla-h}(Codazzi term)
\bea 
    \Psi(X,X) = &-4 h(\nabla V,\nabla V)h(X,X) - 8h^2(X,\nabla V) \nabla_X V  \\
-&4\eta V h^2(\nabla V, \nabla V)h(X,X)\\
 =& -\eta^{-1}V^{-1}h(\nabla V, \nabla V)(Q^{-1}|\nabla_X Q|^2-2\varepsilon |X|^2) - 8h^2(X,\nabla V) \nabla_X V \nonumber\\
    -& h^2(\nabla V, \nabla V)(Q^{-1}|\nabla_X Q|^2-2\varepsilon |X|^2)\nonumber.\eea 
Replacing $h(X,X)$ and $h^2(X,\nabla V)$ via Lemma \ref{lemma_null_eigen},  we have
\begin{align}
    &\Psi(X,X) + 4\nabla_{\ell}(|\nabla V|^2)\nabla_{\ell}(\nabla_i V \nabla_j V)X^iX^j \\
    = &  -\eta^{-1}V^{-1}h(\nabla V, \nabla V)(Q^{-1}|\nabla_X Q|^2-2\varepsilon |X|^2) + (16\eta^{-2}\!-\! 8)h^2(X,\nabla V) \nabla_X V \nonumber\\
    &- h^2(\nabla V, \nabla V)(Q^{-1}|\nabla_X Q|^2-2\varepsilon |X|^2)\nonumber\\
    = &  -\eta^{-1}V^{-1}h(\nabla V, \nabla V)(Q^{-1}|\nabla_X Q|^2-2\varepsilon |X|^2)\\
   &- h^2(\nabla V, \nabla V)(Q^{-1}|\nabla_X Q|^2-2\varepsilon |X|^2)\nonumber\\
    &+ (4\eta^{-2}-2)Q^{-1}V^{-1}\left(h(\nabla V,\nabla V) \!-\! \frac{1}{8}\varepsilon \eta^{-2} V^{-1}\left(Q^{-1} |\nabla Q|^2- 2\eps\right) \right)|\nabla_X Q|^2\\
    =&o(1) Q^{-3} |\nabla Q|^2 |\nabla _X Q|^2 .
\end{align}
Next, we have $Q^{-1} |\nabla Q|^2 |\nabla^2 V(X)|^2 = o(1) Q^{-3}|\nabla _XQ|^2 |\nabla Q|^2 $ since 
  \bea 
        |\nabla^2 V(X)|^2 &= \eta^{-4}Q^{-1}\left(\frac{1}{4}\varepsilon^2|X|^2 -\varepsilon |\nabla_X V|^2 + |\nabla_X V|^2|\nabla V|^2\right) \\
&=\eta^{-4}Q^{-1}\left(|\nabla_X V|^2\left(|\nabla V|^2-\frac{1}{2}\varepsilon\right)-\frac{1}{2}\varepsilon \left(|\nabla_X V|^2-\frac{1}{2}\varepsilon|X|^2\right)\right)\\
        &\leq  \eta^{-4}Q^{-1}|\nabla_X V|^2\left(|\nabla V|^2-\frac{1}{2}\varepsilon\right)\nonumber\\
        &\leq  \frac{1}{16}Q^{-3}|\nabla_X Q|^2(Q^{-1}|\nabla Q|^2-2\varepsilon)\\
&= O(1) Q^{-3}|\nabla _XQ|^2 |\nabla Q|^2 .
    \eea
and 
\begin{align}
        Q^{-1}|\nabla Q|^2|\nabla^2 V(X)|^2 &\leq \frac{1}{16}Q^{-4}|\nabla_X Q|^2|\nabla Q|^2(Q^{-1}|\nabla Q|^2-2\varepsilon)\\
        & =  \frac{1}{4}|\nabla V|^2 Q^{-2}|\nabla_X Q|^2(Q^{-1}|\nabla Q|^2-2\varepsilon). \nonumber
    \end{align}
Finally, the first term of $N_2(X,X)$ in \eqref{eq-N2}   has negative sign. Combining things together, we finish the proof.  
\end{proof} 

\end{lemma}

\begin{lemma}[estimate on $N_3(X,X)$] \label{lemma-estN3}

 \begin{equation}
         N_3(X,X)\leq 0.
     \end{equation}
\begin{proof}
In view of terms in $(N_3)_{ij}$, it suffices to check 
\begin{align}
  &    3\eta^{-2}V^{-2}\Big[(n-k)-6|\nabla V|^2 +\tfrac{V^2}{2t}\Big (1-\tfrac{1}{\log(-t)}\Big ) \Big ] +6  \eta^{-1}V^{-1}  h(\nabla V,\nabla V)   > 0 \nonumber .
\end{align}
The second term is {positive} as $h$ is positive definite. Next, recall that on $ L\sqrt { \frac{-t}{\log (-t)}}\le V \le \sqrt {(2(n-k)(-t) } -L \frac{\sqrt{-t}}{\log (-t)}$, we have   \begin{align}
    \tfrac{V^2}{2t} \geq -(n-k)+\tfrac{L}{\log(-t)} \quad \text{and} \quad |\nabla V|^2 \leq \tfrac{C}{\log(-t)}.
\end{align}
We can take $L$ sufficiently large such that $L \gg C$ 
\begin{align}
    \Big [(n-k)-6|\nabla V|^2+\tfrac{V^2}{2t}  \Big (1-\tfrac{1}{\log(-t)}\Big ) \Big ] \geq \tfrac{L}{2\log (-t)}>0.
\end{align}
This finishes the proof.
\end{proof}
\end{lemma}

Next, we check the sign in the excluded region. Because the third term in $A$ is negative definite and the fourth term in $A$ is  negligible compared with the second term in $A$, as the argument in \cite{CDDHS} we have good sign in the excluded region:

\begin{proposition}[sign in excluded region]\label{lemma-soliton-region} For every $L<\infty$ there exist constants $\kappa>0$ and $\tau_*>-\infty$ with the following significance. If $\mathcal{M}$ is $\kappa$-quadratic at time $\tau_0 \le \tau_*$ then for all times $t\le -e^{-\tau_0}$
 in $\{ V \leq L\sqrt{|t|/\log|t|}\}$ and in $\{ V\geq \sqrt{2|t|}-L\sqrt{|t|}/\log|t| \}$ we have 
 \begin{equation}
\left[{\nabla^2} Q- \gamma g-\frac{|\nabla Q|^2}{4Q}\left(\frac{\nabla Q\otimes\nabla Q}{2Q}-\gamma g\right)\right]\!\!(X, X)\leq 0
\end{equation}
for all $X$ perpendicular to all spherical vector fields $\partial{\vartheta^{\alpha}}$.
\end{proposition}
\begin{proof}
    The proof is similar as proof of \cite[Prop 3.8]{CDDHS}. For convenience of reader and completeness, we include the proof here.  Recall from \cite[Lemma 5.4]{ADS2} that the profile function of the  round $k$ dimensional bowl is quadratically concave. Since $\nabla^2 Q$ is a scale invariant quantity, applying Theorem \ref{strong_uniform0} (uniform sharp asymptotics) we thus infer that
\be
\nabla^2 Q(X,X)\leq \lambda \, |X|^2
\ee
in $\{ V \leq L\sqrt{|t|/\log|t|}\}$, where $\lambda>0$ can be made arbitrarily small by choosing $\kappa>0$ small enough and $\tau_\ast<-\infty$ negative enough. On the other hand, inserting $V \leq L\sqrt{|t|/\log|t|}$ in equation \eqref{eq-Q} we see that
\begin{equation}
\gamma \, g(X,X) \geq L^{-3} |X|^2.
\end{equation}
Suppose now $V\geq\sqrt{2|t|}-L\sqrt{|t|}/\log|t|$. Then by Theorem \ref{strong_uniform0} (uniform sharp asymptotics) and convexity we have\footnote{One has to apply this twice, specifically first considering the tangent plane at any point with say renormalized radial coordinate $y=100$ to show that $\{V\geq\sqrt{2|t|}-L\sqrt{|t|}/\log|t|\}$ is contained in the parabolic region, and then again to get the gradient bound.} 
\begin{equation}
|\nabla V|\leq \frac{C}{\log|t|},
\end{equation}
hence
\be
\nabla^2 Q(X,X)\leq \frac{C}{(\log |t|)^2}|X|^2.
\ee
On the other hand, inserting the rough bound $V\geq \sqrt{|t|}$ in the definition of $\gamma$ we get
\begin{equation}
\gamma \, g(X,X)\geq \frac{1}{(\log|t|)^{3/2}}|X|^2.
\end{equation}
 Since the sum of last two tensors 
 \begin{equation}
    \frac{|\nabla Q|^2}{4Q} \left(\frac{\nabla_i Q\otimes\nabla_j Q}{2Q}-\gamma g_{ij}\right)= -2|\nabla V|^2\nabla_{i}V\nabla_{j}V+|\nabla V|^2\varepsilon g_{ij}
 \end{equation}
  in $A_{ij}$ is just small perturbation of $\nabla^2_{ij} Q-\gamma g_{ij}$, the assertion holds.
\end{proof}

 We are now ready to present the proof of Theorem \ref{prop-concavity} (quadratic almost concavity), which we restate here in terms of the unrescaled variables:
 
 \begin{theorem}[quadratic almost concavity]
      \label{prop-concavity_restated}
There exist constants $\kappa>0$ and $\tau_*>-\infty$ with the following significance. If  $\mathcal{M}$ is $\kappa$-quadratic at time $\tau_0 \le \tau_*$, then for all $t\le -e^{-\tau_0}$ we have
\begin{equation}
\left[{\nabla^2} Q- \gamma g-\frac{|\nabla Q|^2}{4Q}\left(\frac{\nabla Q\otimes\nabla Q}{2Q}-\gamma g\right)\right]\!\!(X, X)\leq 0
\end{equation}
for all $X$ perpendicular to all  vector fields $\partial_{\vartheta^{\alpha}}$ along $S^{n-k}$, where $\gamma$ denotes the function defined in \eqref{eq-Q}.
\end{theorem}
 
\begin{proof}Let $\kappa>0$, $\tau_*>-\infty$, and $L<\infty$ be the constants from Proposition \ref{prop-reactionterm} (reaction term). Adjusting $\kappa$ and $\tau_\ast$ we can arrange that the conclusion of Proposition \ref{lemma-soliton-region} (sign in excluded region) holds as well, and also that for all $t\leq - e^{-\tau_\ast}$ we have $\gamma\leq 1/100$.
Suppose now that  $\mathcal{M}$ is $\kappa$-quadratic at time $\tau_0 \le \tau_*$. Given any $\delta\in (0,1/100)$, we work with the tensor $A_{ij}$ in \eqref{Tensor A}

By Proposition \ref{lemma-soliton-region} (sign in excluded region) for all $t\leq -e^{-\tau_0}$ in the region 
\be
E_t:=\big\{ V \leq L\sqrt{|t|/\log|t|}\big\} \cup \big\{ V\geq \sqrt{2|t|}-L\sqrt{|t|}/\log|t| \big\},
\ee
we have $A(X,X)\leq 0 $ for all $X\perp \partial_{\vartheta^{\alpha}}$.

Moreover, by the derivative estimate from Corollary \ref{lemma-cylindrical} (cylindrical estimate)\footnote{One can apply this for some $\kappa'\ll \kappa$, since the inwards quadratic bending improves when one goes further back in time (as we have seen in the proof of Theorem \ref{point_strong}).}, and Proposition \ref{lemma-soliton-region}, there exists some $T_\delta < -e^{-\tau_0}$ such that for all $t\leq T_\delta$ the estimate
\be\label{a_leq_ineq}
A(X,X)<0  \quad \textrm{for all}\,\, \text{nonzero}\,\,  X\perp \partial_\vartheta
\ee
holds at all points (here, by convention we set $A(X,X):=-\infty$ at points with $V=0$). 

\smallskip 

Suppose towards a contradiction that there is some time $t\in (T_\delta, -e^{-\tau_0}]$ such that \eqref{a_leq_ineq}  fails, and let $\bar{t}$ be the first such time. Let $\bar{p}\in S^n$ be a point where this happens. By the above, we have $\bar{p}\notin E_{\bar{t}}$. We now choose a null eigenvector $X\in T_{\bar{p}}S^n\cap \textrm{span}\{\partial_{\vartheta^{\alpha}}\}_{\alpha=1, \dots, n-k}^\perp$ and extend it to a vector field around $\bar{p}$ as follows:

\begin{claim}[extension]\label{claim_extension}  There exists an extension of $X$ to a vector field $X(p)$ in an open neighborhood of $\bar p$, say $U$, with the following properties: 

\begin{enumerate}[(i)]
\item $X\perp \partial_{\vartheta^{\alpha}}$ in $U$,	
\item $\nabla _{\partial _{\vartheta^{\alpha}}} X =(\tfrac12 Q^{-1} \nabla_X Q)  \, \partial_{\vartheta^{\alpha}} $ in $U$,
\item$\nabla _{\dot \gamma(t)} X = 0$ for any geodesics $\gamma$ in $U$ with $\gamma(0)=\bar p$ and $\dot \gamma(0)  \perp \partial _\vartheta  $.   
\end{enumerate}
\end{claim}

\begin{proof}[Proof of the claim]Working in $(x,\vartheta)$-coordinates we can construct a vector field of the form
\be
X=\sum_{i=1}^{k}X^i(x_1, \dots, x_k)\partial_{x_i}.
\ee
by first parallel transporting in the $(x_1, \dots, x_k)$-plane along radial geodesics emanating from the point under consideration, and then declaring that in this local formula for $X$ is independent of $\vartheta$. Then, in a neighborhood of $\bar p$, the properties (i) and (iii) hold by construction, and moreover using  $\Gamma_{\ell j}^\ell =\nabla_jV /V$ for $j\in \{1,\dots, k\}, \ell \in \{k+1, \dots, n\}$  we get (ii). This proves the claim.
\end{proof}

Continuing the proof of the theorem, we consider the function
\begin{equation}
f(p,t):= A_{(p,t)}( X(p), X(p)).
\end{equation}
Then, at $(\bar{p},\bar{t})$, we have
\be\label{2nd_der_test}
\partial_t f \geq 0,\quad \nabla f = 0,\quad \Delta f \leq 0.
\ee
Recall the evolution of $A$ in Proposition \ref{evolution of new A}. At the point $(\bar{p},\bar{t})$, in terms of an orthonormal basis $e_1,\dots, e_{n}$, 
\begin{align}\label{tlaplacegradf}
    (\partial_t - \Delta-\nabla_{Z})f&=N(X, X) -2\sum_{i=1}^{n} A(\nabla_{e_i} X, \nabla_{e_i} X) - 4\sum_{i=1}^{n}(\nabla_{e_i} A)( \nabla_{e_i} X, X).
\end{align}
On the other hand,  using \eqref{obs_grad} and \eqref{eq-hessian V} we see that
\be
g(\partial_{\vartheta^{\alpha}},\partial_{\vartheta^{\alpha}})^{-1}A(\partial_{\vartheta^{\alpha}},\partial_{\vartheta^{\alpha}})=\tfrac12 (Q^{-1}|\nabla Q|^2-2\eps)\, .
\ee
Hence, applying Claim \ref{claim_extension} (extension) we infer that
\be
-2\sum_i A(\nabla_{e_i} X, \nabla_{e_i} X)=-{\tfrac{n-k}{4}}Q^{-2}|\nabla_X Q|^2(Q^{-1}|\nabla Q|^2-2\eps).
\ee
Moreover, working in a frame $\{\bar e_i\}$ that extends $\{e_i\}$ to a neighborhood of $\bar p$, the null eigenvector condition implies
\begin{multline}\label{eq-tensorial}
\sum_i\nabla_{\bar e_i} (A(\nabla_{\bar e_i} X,X))=\sum_i(\nabla_{\bar e_i} A)(\nabla_{\bar e_i} X,X)+ \sum_i A(\nabla_{\bar e_i} X,\nabla_{\bar e_i} X),
\end{multline}
and  we infer that the left-hand side of \eqref{eq-tensorial} evaluated at $\bar p$ is independent of the choice of $\{e_i\}$ and its extension as a frame. Hence, for any  curves $\{\gamma_i\}_{i=1}^{n}$ starting at $\bar p$, if $\{\dot \gamma_i(0)\}$ forms an orthonormal basis, then 
\be \sum_i\nabla_{\bar e_i} (A(\nabla_{\bar e_i} X,X)) = \sum_i \frac{d}{dt}\Big  \vert_{t=0}A_{\gamma_i(t)}(\nabla_{\dot \gamma_i(t)}X,X(\gamma_i(t))).\ee
    Choosing $\gamma_j$ for $1\leq j\leq k$  to be unit speed geodesics satisfying $\dot \gamma_i(0)\perp \partial_{\vartheta^{\alpha}}$ and $\gamma_\alpha$ to be the integral curve of $V^{-1} \partial_{\vartheta^{\alpha}}$, by Claim \ref{claim_extension} (extension), remembering also the block-diagonal structure of $A$, we infer that  for each $i\in\{1,\dots,n\}$ the function $t \mapsto A_{\gamma_i(t)}(\nabla_{\dot \gamma_i(t)}X,X(\gamma_i(t)))$ is identically zero. This yields
\begin{align}
- 4\sum_i(\nabla_{e_i} A)( \nabla_{e_k} X, X)&= 4\sum_i A(\nabla_{\bar e_i} X,\nabla_{\bar e_i} X)\\
&={\frac{n-k}{2}}Q^{-2}|\nabla_X Q|^2(Q^{-1}|\nabla Q|^2-2\eps).
\end{align}
Hence, these together with \eqref{tlaplacegradf}, we have
\begin{align}
    (\partial_t - \Delta-\nabla_{Z})f&=N(X, X) +2\sum_{i=1}^{n} A(\nabla_{e_i} X, \nabla_{e_i} X).
\end{align}

Next, let us further assume $e_{k+1},\dots, e_{n}$ is an orthonormal basis of the spherical part.   From the definition of $A$ in \eqref{eq-tensorA}, \be
2A(e_j ,e_j)=Q^{-1}|\nabla Q|^2-2\eps+2|\nabla V|^2\varepsilon, \quad \text{for }j\in\{k+1,\dots, n\}.
\ee
Hence, applying Claim \ref{claim_extension} (extension) we infer that
\bea \label{eq-Aspherical}
2\sum_i A(\nabla_{e_i} X, \nabla_{e_i} X)&=\tfrac{n-k}{4}Q^{-2}|\nabla_X Q|^2(Q^{-1}|\nabla Q|^2-2\eps)\\
&+\tfrac{n-k}{8}\varepsilon Q^{-3}|\nabla_X Q|^2|\nabla Q|^2.\nonumber    
\eea 
Finally, recall $Q^{-1}|\nabla Q|^2-2\eps>0$ by Lemma \ref{lemma_null_eigen} (null eigenvector) and $\eps = o(1)$ from \eqref{eq-epsest}. Combining \eqref{eq-Aspherical} with the estimate on $N(X,X)$ in Proposition \ref{prop-reactionterm} (reaction term),  \begin{equation}
    (\partial_t - \Delta-\nabla_{Z})f< 0
\end{equation}
at $(\bar{p},\bar{t})$, for some (but fixed) $\kappa>0$, $\tau_*>-\infty$, and $L<\infty$. This gives the desired contradiction with \eqref{2nd_der_test}. Since $\delta>0$ was arbitrary, this finishes the proof of the theorem.
\end{proof}

Finally, we can now prove the crucial Theorem \ref{prop-great} (almost Gaussian collar), which we restate here for convenience of the reader.

\begin{theorem}[almost Gaussian collar]\label{gaussian collar section 3}
For every $\varepsilon > 0$, there exist constants $\kappa>0$, $\tau_*>-\infty$, $L<\infty$, and $\theta >0$ with the following significance. If  $\mathcal{M}$ is a $k$-oval which is $\kappa$-quadratic at time $\tau_0 \le \tau_*$, then for all $\tau \le \tau_0$ and for the profile function $v$ in \eqref{v_radial_representation} and its inverse profile function $Y$ in \eqref{inverse profile}, the following radial derivative estimate
\begin{equation}
\left|y (v^2)_y + 4(n-k)\right| < \varepsilon,
\end{equation}
and equivalently the almost Gaussian collar estimates
\begin{equation}
    \left|1+\frac{vY}{2(n-k)Y_{v}}\right| < \varepsilon
\end{equation}
hold in the collar region $\mathcal{K}= \bigl\{  L/\sqrt{|\tau|} \le v \le 2  \theta \bigr\}$.
 \end{theorem} 
 
\begin{proof}
To begin with, given any $\delta>0$, using Theorem \ref{strong_uniform0} (uniform sharp asymptotics) and convexity, we see that for $\tau\leq \tau_0$ in the region $\{v\leq 2\theta\}$, provided $\theta=\theta(\delta)$ is small enough, we have
\be\label{cor_bas1}
2|\tau|(1-\delta)\leq y^2 \leq 2|\tau|(1+\delta)
\ee
for any $\mathcal{M}$ that is $\kappa=\kappa(\delta)$-quadratic from time $\tau_0\leq \tau_\ast(\delta)$. Moreover, Theorem \ref{strong_uniform0} (uniform sharp asymptotics) and convexity also yield
\begin{equation}\label{cor_bas2}
(vv_y)^2|_{v=2\theta} \geq \frac{2(n-k)}{|\tau|}(1-\delta), \, 
(vv_y)^2|_{v=  |\tau|^{-\frac12} L}\leq \frac{2(n-k)}{|\tau|}(1+CL^{-1}),
\end{equation}
for $L$ large enough, possibly after decreasing $\kappa$ and ${\tau}_{\ast}$ (see e.g. \cite[Proof of Lemma 5.7]{ADS2} for a similar computation, with more details).

Now, applying Theorem \ref{prop-concavity} (quadratic almost concavity) in direction of the radial vector $X=\partial_y$ and writing in terms of rescaled solution $v$, we get
\be\label{quadratic concavity ode inequality}
2\eta^{-2} (vv_{yy}+v_y^2)\leq |\tau|^{-3/2}v^{-3}\eta^{-2} (1+v_y^2),
\ee
where $\eta^2 =1+|Dv|^2$. 
Together with Corollary \ref{lemma-cylindrical} (cylindrical estimate) this implies
\begin{equation}
(v v_y)_y \le  20|\tau|^{-3/2}v^{-3}
\end{equation}
for all $\tau \le \tau_0$ in the region $\{v\geq L/\sqrt{|\tau|}\}$, provided $L$ is sufficiently large, $\kappa$ is sufficiently small, and $\tau_\ast$ is sufficiently negative.
For any fixed $\omega\in S^{k-1}$ and $\tau$, considering $v_y$ as a function of $v$, we can rewrite our inequality as
   \be   \label{eq-diff-ineq} \frac{d}{dv} (v^2 v_y^2)  \le  20  |\tau|^{-3/2}v^{-2}\, .
   \ee 
Integrating this from $v$ to $2\theta$ and from $L/\sqrt{\tau}$ to $v$ yields
\be
(vv_y)^2|_{v=2\theta} -20  L^{-1} |\tau|^{-1} \leq v^2 v_y^2 \leq (vv_y)^2|_{v=L/\sqrt{|\tau|}} + 20  L^{-1} |\tau|^{-1}.
\ee
Multiplying this by $y^2$, and then using \eqref{cor_bas1} and \eqref{cor_bas2}, the first assertion follows. The second assertion follows from inverse function theorem and the first assertion.
\end{proof}
\bigskip

\section{Spectral uniqueness theorem and reflection symmetry}\label{spectral_uniqueness}
In this section, we prove Theorem \ref{thm:uniqueness_eccentricity_intro} (spectral uniqueness).
Let $\mathcal{M}^{1}=\{M^1_t\}$ and $\mathcal{M}^{2}=\{M^2_t\}$ be two  $k$-ovals in $\mathbb{R}^{n+1}$ that are $\kappa$-quadratic at  time $\tau_{0}$, where $\tau_0\leq \tau_{*}$, and  their truncated renormalized profile functions $v^{1}_{\cC}$ and $v^{2}_{\cC}$ satisfy the spectral conditions
\begin{equation} \label{cent_0}
      \mathfrak{p}_{0}v^{1}_{\cC}(\tau_{0})=\mathfrak{p}_{0}v^{2}_{\cC}(\tau_{0}),
\end{equation}
\begin{equation}\label{p_plus_van}
     \mathfrak{p}_{+}v^{1}_{\cC}(\tau_{0})=\mathfrak{p}_{+}v^{2}_{\cC}(\tau_{0})
\end{equation}
Here, denoting by $v_i$ the profile function of the renormalized flow $e^{\frac{\tau}{2}} M^i_{ - e^{-\tau}}$, the truncated renormalized profile function is defined by
\be
v_\cC^i=\chi_\cC(v_i)v_i,
\ee
where $\chi_\mathcal{C} : [0, \infty) \to [0, 1]$ is a smooth cut-off function satisfying
\be \label{eq-cutoffC}
  (i)\,\, \chi_\mathcal{C} \equiv 0 \mbox{ on } \big[0,\tfrac58 \theta\big]   \qquad \mbox{and} \qquad  (ii) \,\, \chi_\mathcal{C} \equiv 1 \mbox{ on } \big[\tfrac78\theta, \infty\big). 
\ee
We also recall that the evolution of $v_\cC$ is governed by the Ornstein-Uhlenbeck operator, which in Euclidean coordinates is given by the formula in \eqref{OU_operator}
\be\label{2dOU}
\mathcal{L}=\Delta_{\mathbb{R}^{k}}-\tfrac{1}{2}\mathbf{y}\cdot \nabla +1,
\ee
and which is a self-adjoint operator on the Gaussian $L^2$-space
\begin{equation}
\mathcal{H}=L^2\big(\mathbb{R}^k,e^{-|{\bf y}|^2/4} \,  d{\bf y}\big)= \mathcal{H}_+\oplus \mathcal{H}_0\oplus \mathcal{H}_-.
\end{equation}
Here, the unstable and neutral eigenspace are explicitly given by
\begin{equation}
\mathcal{H}_+=\textrm{span}\big\{1, y_1, \dots, y_k\big\}, 
\mathcal{H}_0=\textrm{span}\big\{y_1^2-2, \dots, y_k^2-2, y_iy_j: 1\leq i<j\leq k\big\},
\end{equation}
and as before we denote the orthogonal projections by $\mathfrak{p}_{\pm}$ and $\mathfrak{p}_0$. 
As before, given any renormalized profile function $v$ (e.g. $v=v_1$) we consider the associated cylindrical region and tip region
\begin{equation}
\label{eq-cyl-region_rest}
\cC= \big\{ v \ge \theta/2  \big\}\qquad\mathrm{ and }\qquad \mathcal{T}= \big\{v \le 2   \theta \big\},
\end{equation}
where the latter can be subdivided into the collar region and soliton region:
\begin{equation}
\collar = \bigl\{  L/\sqrt{|\tau|} \le v \le 2  \theta \bigr\} \qquad\mathrm{ and }\qquad 
\mathcal{S} = \bigl\{v \le L/\sqrt{|\tau|} \bigr\}.
\end{equation}
Note that $\chi_\cC$ indeed localizes in the cylindrical region, more precisely\footnote{The wiggle room between $\tfrac12\theta$ and $\tfrac58 \theta$ will be used for estimating $v_1\chi_\cC(v_1)-v_2\chi_\cC(v_2)$.}
\be\label{eq_wiggle_room}
\textrm{spt}(\chi_\cC)\subseteq \{ v\geq \tfrac58\theta\} \subset\cC.
\ee
To localize in the tip region, we fix a smooth function $\chi_\mathcal{T} : [0, \infty) \to [0, 1]$ satisfying
\be\label{chi_cut-off}
  (i)\,\, \chi_\mathcal{T} \equiv 1 \mbox{ on } [0,\theta]   \qquad \mbox{and} \qquad  (ii) \,\, \chi_\mathcal{T} \equiv 0 \mbox{ on } [2\theta, \infty). 
\ee
In the tip region we work with the inverse profile function $Y$ defined by
\begin{equation}\label{def_inv_sec4}
Y(v(y,\omega,\tau),\omega,\tau)=y,
\end{equation}
and its zoomed in version $Z$ defined by
\begin{equation}\label{zoomed_in_5}
Z(\rho,\omega,\tau)= |\tau|^{1/2}\left(Y(|\tau|^{-1/2}\rho, \omega,\tau)-Y(0,\omega,\tau)\right).
\end{equation}

Throughout this section we use the convention that $\theta>0$ is a fixed small constant and $0<L<\infty$ is a fixed large constant. During the proofs one is allowed to decrease $\theta$ and increase $L$ at finitely many instances, as needed.

\subsection{Energy estimate in the cylindrical region}\label{energy-cylindrical}
In this subsection, we prove an energy estimate in the cylindrical region, by generalizing \cite[Section 6]{ADS2} and \cite[Section 5.4]{CHH} to the bubble-sheet setting. We first review the relevant norms and spaces for the energy estimate.
In addition to the Gaussian $L^2$-space $\mathcal{H}$, which is equipped with the norm
\begin{equation}\label{eqn-normp0}
  \|f\|_{\mathcal{H}} = \left(\int_{\R^2}  f(\bry)^2 e^{-|\bry|^2/4}\, d\bry\right)^{1/2},
\end{equation}
we also need the Gaussian $H^1$-space $\hD:= \{f\in\mathcal{H}  : Df  \in \mathcal{H}\}$ with the norm
\be
  \|f\|_\hD = \left(\int_{\R^2} \big( f(\bry)^2 +|Df(\bry)|^2 \big) \, e^{-|\bry|^2/4}d\bry\right)^{1/2},
\ee
and its dual space $\hD^\ast$ equipped with the dual norm
\be
\|f\|_{\hD^\ast}=\sup_{ \|g\|_\hD\leq 1 }\langle f,g\rangle,
\ee
where $\langle \,\,\, , \,\, \rangle: \hD^\ast \times \hD:\to\mathbb{R}$ denotes the canonical pairing.

For time-dependent functions the above induces the parabolic norms
\be
  \|f\|_{\mathcal{X} ,\infty} = \sup_{\tau \le \tau_0} \left(\int_{\tau -
    1}^{\tau} \|f(\cdot,\sigma)\|_{\mathcal{X} }^2\, d\sigma\right)^{\frac12}, 
\ee
where $\mathcal{X}=\mathcal{H},\hD$ or $\hD^\ast$.

Let us also recall a few basic facts that will be used frequently in the following proof. To begin with, by the weighted Poincar\'e inequality \eqref{easy_Poincare} multiplication by $1+|\bry|$ is a bounded operator from $\hD$ to $\mathcal{H}$, and hence by duality from $\mathcal{H}$ to $\hD^\ast$ as well, namely
  \be\label{dhstar norms}
    \|(1+|\bry|)f\|_\mathcal{H}  \leq C \|f\|_\hD\quad\textrm{and}\quad
    \|(1+|\bry|)f\|_{\hD^\ast}  \leq C \|f\|_\mathcal{H}.
  \ee
Consequently, $\pd_{y_i}$ and $\pd_{y_i}^*= -\pd_{y_i}+ \frac{1}{2} y_i$ are bounded operators from  $\hD$ to $\mathcal{H}$, and hence by duality from $\mathcal{H} $ to $\hD^*$ as well. In particular, this implies that the Ornstein-Uhlenbeck operator $\mathcal{L}:\hD\to \hD^\ast$ is well-defined.
Finally, for estimating the $\hD^\ast$-norm it is useful to observe that if $g\in \hD$ and $h\in W^{1,\infty}$, then by the product rule we have $\| hg\|_{\hD}\leq C \| h \|_{W^{1,\infty}}\| g\|_{\hD}$, hence by duality
\begin{equation}\label{eq_product_rule_norm}
\| h f \|_{\hD^\ast} \leq C \| h\|_{W^{1,\infty}} \| f \|_{\hD^\ast}\, .
\end{equation}

Now, given $\mathcal{M}^{1}$ and $\mathcal{M}^2$ we consider  the truncated difference of two renormalized profile functions
\be
w_\cC=v_1\chi_\cC(v_1)-v_2\chi_\cC(v_2),
\ee
whose evolution equation is given by  Proposition \ref{proposition-evolution-wC} (evolution equation of $w_{\cC}$). Then we have the following energy estimate in cylindrical region.
\begin{proposition}[energy estimate in cylindrical region]\label{prop-cyl-est} For every $\eps>0$, there exists $\kappa>0$ and $\tau_*>-\infty$ with the following significance. If $\mathcal{M}^1$ and $\mathcal{M}^2$  are $\kappa$-quadratic at time $\tau_0 \le \tau_*$ and satisfies \eqref{p_plus_van}, then
\begin{equation}
\label{eqn-cylindrical1}
\Vert w_\cC - \mathfrak{p}_0w_\cC \Vert_{\hD ,\infty } \le \eps  \Big( \Vert  w_\cC\Vert_{\hD ,\infty} + \Vert w 1_{\{ \theta/2\leq v_1\leq \theta\}} \Vert_{\mathcal{H},\infty}  \Big).
\end{equation}
\end{proposition}
\begin{proof}
 The proof is  the same as the energy estimates in \cite[Proposition 4.4]{CDDHS}, the only difference is that our current evolution equation of $w_{\cC}$ derived in Proposition \ref{proposition-evolution-wC} has slight difference, so the coefficient $b=\frac{2(n-k)-v_{1}v_{2}}{2v_{1}v_{2}}\chi'_{\cC}(v_1)v_2$ in \eqref{eq-diff-E} is different from the one in \cite[proof of Proposition 4.4]{CDDHS}. For the readers' convenience, we include the proof here.

Due to \eqref{p_plus_van} and \cite[Lemma 6.7]{ADS2}, we have the general estimate
\begin{equation}
\label{eq-gen-energy}
\|w_\cC - \mathfrak{p}_0 w_\cC \, \|_{\hD,\infty} \le C\|(\partial_{\tau} - \mathcal{L}) w_\cC\|_{\hD^*,\infty}.
\end{equation}
Then we write Proposition \ref{proposition-evolution-wC} (evolution of $w_\cC$)   in the form
\begin{equation}
\label{eq-ev-wC}
(\partial_\tau -\mathcal{L})\, w_\cC = \cE[w_\cC] + \bar{\cE}[w,\chi_\cC(v_1)]+J+K,
\end{equation}
where
\be\label{eqn-IJK}
\begin{split}
&J \, =  (v_{2,\tau} - v_{2,y_iy_i}  +\tfrac{1}{2}  y_i v_{2,y_i}  -\cE[v_2] ) w^{\chi_\cC} -  2 D v_2 \cdot Dw^{\chi_\cC},\\
&K\, = \cE[v_2] w^{\chi_\cC} - \cE[v_2w^{\chi_\cC}] + v_2\, (\partial_\tau -\mathcal{L})w^{\chi_\cC}.\end{split}
\ee
Following the same arguments as in \cite[Proof of Lemma 5.13, Lemma 5.14]{CHH_translator}, which in turn is similar to \cite[Proof of Lemma 6.8 and Lemma 6.9]{ADS2}, but now using the evolution of $w$, $(\partial_\tau -\mathcal{L})w=\mathcal{E}[w]$, the derivative estimates from Lemma \ref{lemma-cylindrical} (cylindrical estimate) and the uniform sharp asymptotics from Theorem \ref{strong_uniform0} (uniform sharp asymptotics), we have for any $\eps > 0$, there exist $\kappa > 0$ and $\tau_* > -\infty$, such that for all $\tau\leq\tau_0\leq \tau_*$ 
\begin{align}\label{eq-I1}
\quad&\big\| \cE[w_\cC(\tau)]   \big\|_{\hD^\ast}  +\big\| \bar{\cE}[w(\tau),\chi_\cC(v_1(\tau))]  \big\|_{\hD^\ast}+ \| J(\tau) \|_{\hD^*} \\
\leq&  \eps  \|w_\cC  (\tau) \|_{\hD} +\varepsilon \| w(\tau)  \, 1_{D_\tau } \|_{\mathcal{H}},
\end{align}
and
\be
1_{D_\tau} \leq 1_{ \{ \theta/2 \leq v_1(\cdot,\tau)\leq \theta\} },
\ee
where
\be
D_\tau:=\Big\{ y: \tfrac58 \theta \leq v_1(y,\tau)\leq \tfrac78 \theta\,\,\,\textrm{or}\,\,\,\tfrac58 \theta \leq v_2(y,\tau)\leq \tfrac78 \theta\Big\}.
\ee

Then we split $K(\tau)$ into following two lines of terms.
\begin{align}
K(\tau )&= v_2(\partial_\tau-\mathcal{L})w^{\chi_\cC}- \chi_\cC'(v_1)v_2 \cE[w]\\
&+\cE[v_2] w^{\chi_\cC} - \cE[v_2 w^{\chi_\cC}]+ \chi_\cC'(v_1)v_2 \cE[w] 
\end{align}
By a similar computation and estimates as in \cite[(4.54)-(4.55)]{CDDHS} we see that
\begin{align}\label{second line K}
\| v_2(\partial_\tau-\mathcal{L})w^{\chi_\cC}- \chi_\cC'(v_1)v_2 \cE[w]\|_{\hD^\ast} \leq \eps \left\| w 1_{D_\tau} \right\|_{\cH},
\end{align}
For the estimates of first line, we have:
\begin{claim}\label{abcd estimates}
    \be\label{abcd estimates eq}
\left\| \cE[v_2] w^{\chi_\cC} - \cE[v_2 w^{\chi_\cC}]+\chi_\cC'(v_1)v_2  \cE[w] \right\|_{\hD^\ast}\leq \eps \left\| w 1_{D_\tau} \right\|_{\cH}.
\ee
\end{claim}
Under assuming Claim \ref{abcd estimates} hold and combining the above estimates, we conclude the proof of the proposition.

\begin{proof}[Proof of Claim \ref{abcd estimates}]
By a similar computation as in \cite[(4.43)-(4.47)]{CDDHS} we have 
\begin{equation}
\label{eq-diff-E}
\quad\cE[v_2] w^{\chi_\cC}- \cE[v_2 w^{\chi_\cC}]+ \chi_\cC'(v_1)v_2 \cE[w]     = a\!\cdot\!Dw + b\, w + c\, w^{\chi_\cC'} + d \, w^{\chi_\cC''},
\end{equation}
where
\be\begin{split}
a& = \frac{v_2\chi_\cC''(v_1) Dv_1\!\cdot\! D(v_1+v_2)  + 2\chi_\cC(v_1) Dv_1\!\cdot\! Dv_2 }{1+|Dv_1|^2}\, Dv_1,  \\
b &= \frac{2(n-k)-v_1v_2}{2v_1v_2}\, \chi_\cC'(v_1) v_2,\\
c &= \frac{v_2D^2v_2\!:\!  ( Dv_1\otimes Dv_1+D(v_1+v_2)\otimes Dv_2)  +2(Dv_1\!\cdot\! D v_2)^2}{1+|Dv_1|^2} \\
& - \frac{v_2  Dv_2 \!\cdot\! D(v_1+v_2) \,  D^2v_2\!:\! Dv_2\otimes Dv_2  }{(1+|Dv_1|^2)(1+|Dv_2|^2)},\\
d &= \frac{v_2 (Dv_1\!\cdot\! D v_2)^2}{1+|Dv_1|^2}.
\end{split}
\ee
Now, using the basic facts reviewed in \eqref{dhstar norms}, \eqref{eq_product_rule_norm} above and  Theorem \ref{strong_uniform0} (uniform sharp asymptotics) and and using the additional $|{\bf y}|\geq |\tau|^{1/2}$ on the support of $1_{D_\tau}$ for term $b$, we can estimate
\be
\| a\cdot Dw \|_{\hD^\ast}+ \| c w^{\chi_\cC'} \|_{\hD^\ast}+\| d w^{\chi_\cC''}\|_{\hD^\ast}   \leq C (\| a\|_{W^{1,\infty}}+\| c\|_{W^{1,\infty}}+\| d\|_{W^{1,\infty}}) \| w1_{D_\tau}\|_{\mathcal{H}},
\ee
\be
\| b w \|_{\hD^\ast}\leq C \left\| \frac{1}{1+|{\bf y}|} bw  \right\|_{\mathcal{H}} \leq \frac{C}{|\tau|^{1/2}} \| w 1_{D_\tau} \|_{\cH}.
\ee
Furthermore, thanks to Corollary \ref{lemma-cylindrical} (cylindrical estimate) we can make $\| a\|_{W^{1,\infty}}+\| c\|_{W^{1,\infty}}+\| d\|_{W^{1,\infty}}$ arbitrarily small by choosing $\kappa>0$ small enough and $\tau_\ast>-\infty$ negative enough. Combining the above we finish the proof of claim.
\end{proof}


\end{proof}

\subsection{Energy estimate in the  tip region}\label{sec-tip}
In this subsection, we prove an energy estimate in the tip region, by generalizing some arguments from \cite[Section 7]{ADS2} and \cite[Section 5.5]{CHH} to the cylinder setting. We first set up the function spaces for energy estimates. Denote by $Y_B$ the rescaled profile function of the $\bowl$, namely
\begin{equation}
Y_B(v,\tau):=\frac{1}{|\tau|^{1/2}}Z_B(|\tau|^{1/2}v), 
\end{equation}
where $Z_B=Z_B(\rho)$ the profile function of  $\bowl $ with speed $1/\sqrt{2}$, namely the unique solution of
\be
\frac{Z_{B,\rho\rho}}{1+Z_{B,\rho}^2}+\frac{n-k}{\rho} Z_{B,\rho} + \frac{1}{\sqrt{2}}=0,\qquad Z_B(0)=Z_{B,\rho}(0)=0.\footnote{Note that in our sign convention $Z_B\leq 0$, which is consistent with \eqref{zoomed_in_5}.}
\ee

In the tip region, for a given inverse profile function $Y(v,\omega,\tau)$, we consider the weight function 
\begin{multline}\label{eqn-weight1}
 \mu(v, \omega , \tau) = - \frac{1}4 Y^2(\theta, \omega , \tau)\\
  + \int_v^\theta\left[  \zeta (v') \, \left(\frac{Y^2({v'}, \omega , \tau)}{4}\right)_{v'} - (1-\zeta(v'))(n-k) \, \frac{1+Y_{B,{v'}}^2(v',\tau)}{v'}\right] \, dv'\, ,
\end{multline}
where $\zeta: \mathbb{R} \to [0,1]$ is a smooth monotone function satisfying
\be\label{zeta_cutoff}
\zeta(v) = 0\,\, \textrm{for} \,\, v \leq \tfrac18 \theta\quad\textrm{ and }\quad \zeta(v) = 1\,\,\textrm{for} \,\, v \geq \tfrac14 \theta. 
\ee
In the tip region we work on energy estimates in function spaces equipped with the following norm
\begin{equation}\label{def_norm_tip_r}
\big\| F\big\|_{2,\infty}= \sup_{\tau\leq \tau_0} \frac{1}{|\tau|^{\frac14}} \left( \int_{\tau-1}^\tau \int_0^{2\theta} \int_{{S}^{k-1}} F(v,\omega ,\sigma)^2 e^{\mu(v,\omega,\sigma)} \, d\omega  dv d\sigma \right)^{\frac12}.
\end{equation}
Here, the integration on the sphere ${S}^{k-1}$ is taken with respect to the volume form induced by the standard round metric. Moreover, in what follows, the norm $|\cdot |$ and inner product are taken with respect to the round cylindrical metric $ dv^2 + g_{{S}^{k-1}}$ on $(v,\omega)\in \mathbb{R}\times {S}^{k-1}$, and $\nabla$ denotes the corresponding Levi-Civita connection. Moreover, $\nabla_{{S}^{k-1}}G$ and $\nabla^2_{{S}^{k-1}}G$ denotes the first and second covariant derivative in the spherical part, respectively.

Then we collect some derivative estimates in tip region as discussed in \cite[Sec 4.2]{CDDHS}.

\begin{proposition}[tip derivatives]\label{lemma-mainY}
For every $\eta>0$, there exist  $\kappa > 0$ small enough, $\theta > 0$ small enough $L>0$ large enough and $\tau_* > -\infty$ negative enough,
such that if $\mathcal{M}$ is $\kappa$-quadratic at time $\tau_0 \le \tau_*$ and satisfies \eqref{p_plus_van}, then
\be
\frac{1}{4}|\tau|^{1/2}\leq \left|\frac{\nabla_v Y}{v}\right|  \leq |\tau|^{1/2},\qquad |\nabla_{S^{k-1}} Y|\leq \eta |\tau|^{1/2}\quad \textrm{and} \quad |Y_\tau|\leq \eta \left|\frac{\nabla_v Y}{v}\right|
\ee
hold in the whole tip region $v \le 2\theta$ for $\tau\le \tau_0$,
\be\label{eqn-cylf}
\left|\frac{\nabla^2_{vv} Y}{1+|\nabla_v Y|^2}\right| + \left|\frac{\nabla_v Y  \nabla_v\nabla_{S^{k-1}}Y }{Y(1+|\nabla_v Y|^2)}\right| + \left|\frac{\nabla^2_{S^{k-1}} Y}{Y^2}\right|  \leq \eta \,  \left| \frac{\nabla_v Y}{v} \right|
\ee
hold in the collar region $L/|\tau|^{1/2}\leq v\leq 2\theta$  for $\tau\le \tau_0$,
\be\label{eqn-uphi}
\left|\nabla^2_{vv}Y\right|  \le C|\tau|^{1/2}, \quad 
\left| \nabla_v\nabla_{S^{k-1}}Y\right|  \leq \eta\, |\tau|, \quad  |\nabla^2_{S^{k-1}}Y|  \leq  \eta |\tau|^{3/2}
\ee
hold in the soliton region $v\leq L/|\tau|^{1/2}$ for $\tau\le \tau_0$, 
\begin{equation}\label{eq-soliton-improvement}
\left|\frac{1+Y_{B,v}^2}{1+Y_v^2} - 1 \right| \leq \eta \min\left\{ 1,\frac{|\tau|^{1/2} v}{L}\right\}
\end{equation}
hold in the whole tip region $v \le 2\theta$ for $\tau\le \tau_0$, and
\be\label{eqn-imp}
\Big |\frac{v \mu_v}{1+Y_{v}^2} - (n-k)\Big | \,  \leq \eta,\qquad
 |\mu_\vp  | \leq \eta \, |\tau |,\qquad
|\mu_\tau | \leq \eta \,  |\tau|\, 
\ee
hold in the whole tip region $v \le 2\theta$ for $\tau\le \tau_0$.
\begin{proof}
    The proof follows by the same proofs as in \cite[Prop 4.5, Cor 4.6, Cor 4.7, Prop 4.8]{CDDHS} by using tip region asymptotics and Theorem \ref{gaussian collar section 3} (almost Gaussian collar estimate).
\end{proof}

\end{proposition} 
Moreover, it is worthwhile to note that in the proof of Proposition \ref{lemma-mainY} in the soliton region relies on a consequence of Theorem \ref{strong_uniform0} (uniform sharp asymptotics). Indeed, the zoomed in profile function $Z=Z(\rho,\omega,\tau)$, as defined in \eqref{zoomed_in_5}, satisfies that for every $\eps>0$, there exist $\kappa > 0$  and $\tau_* > -\infty$,
such that if $\mathcal{M}$ is $\kappa$-quadratic at time $\tau_0 \le \tau_*$, then
\be\label{der_sol1}
\max_{0\leq i+j\leq 10}\sup_{\omega \in {S}^{k-1}}\sup_{\tau\leq \tau_0} \sup_{\rho\leq \eps^{-1}}|\partial_\tau^i\partial_\rho^j(Z-Z_B)|\leq \eps. \ee 

A Poincar\'e inequality, adapted to our weighted Sobolev space in tip region, follows from the derivative estimates.

\begin{corollary}[weighted Poincar\'e inequality]
\label{prop-Poincare}
There exist  $C_0<\infty$, $\theta>0$, $\kappa > 0$ and $\tau_* > -\infty$ with the following significance.
If $\mathcal{M}$ is $\kappa$-quadratic at time $\tau_0 \le \tau_*$ and satisfies \eqref{p_plus_van},  then for $\tau \le \tau_0$ and $\omega\in {S}^{k-1}$ we have 
\begin{equation}
\label{eq-Poincare}
\int_0^{2\theta} F^2(v) \, e^{\mu(v,\omega ,\tau )}\, dv  \le \frac{C_0}{|\tau|}\, \int_0^{2\theta} \frac{F_v^2(v)}{1+Y_{v}^2(v,\omega ,\tau)}\, e^{\mu(v,\omega ,\tau)}\, dv
\end{equation}
for all smooth functions $F$ satisfying $F'(0)=0$ and $\textrm{spt}(F)\subset [0,2\theta)$.
\end{corollary}
\begin{proof}
    The proof is the same as in \cite[Cor 4.9]{CDDHS}. 
    Recall the asymptotic expansions of bowl soliton 
\begin{equation}
  \label{eq-ZB-asymptotics}
  Z_B(\rho) =
  \begin{cases}
    -\frac{1}{2\sqrt2 (n-k)}\rho^2 + O(\log \rho) & \rho\to\infty \\[2pt]
    -\frac{1}{2\sqrt 2(n+1-k)}\rho^2 +O(\rho^4) & \rho\to0,
  \end{cases}
\end{equation}
and the fact that these expansions may be differentiated, see e.g. \cite[Proposition 2.1]{AV_degenerate_neckpinch}.

From this, there is a constant $\Lambda<\infty$ such that
\be\label{lambda_bound_bowl}
\Lambda^{-1} |\tau|^{1/2} v\leq |Y_{B,v}|\leq  \Lambda |\tau|^{1/2}v.
\ee
We fix $v_0=v_0(\tau)= 3\Lambda/ |\tau|^{1/2}$ and start with the integration by parts formula
\be
 - \int_{v_0}^{2\theta}  \frac{(F^2 )_v}v  \, e^{\mu}  dv \\
    =    \frac{F(v_0)^2}{v_0}e^{\mu(v_0,\omega, \tau)} +  \int_{v_0}^{2\theta}  (v\mu_v - 1) \frac{F^2}{v^2}  e^{\mu} dv.
\ee
Together with
\be
  -\frac {2\,F F_v}v  \leq \frac{4 F_v^2}{(n-k)(1+Y_{v}^2)}  +   (n-k)(1+Y_{v}^2)\, \frac{F^2}{4v^2},
 \ee
 this yields
  \bea   \label{eqn-Y13}
&  \frac{F(v_0)^2}{v_0}e^{\mu(v_0,\omega, \tau)}+  \int_{v_0}^{2\theta} \Big ( v\mu_v -  \frac {n-k}4  (1+Y_{v}^2) -1 \Big ) \, \frac{F^2}{v^2} \, e^\mu dv \\
   &\quad  \leq  \frac{4}{n-k}   \int_{v_0}^{2\theta}  \frac{F_v^2}{1+Y_{v}^2} \, e^{\mu}  dv .
  \eea 
Now, using \eqref{eq-soliton-improvement}, \eqref{eqn-imp},\eqref{lambda_bound_bowl},  we  estimate our integrand by
\be
v\mu_v  -  \frac {n-k}4  (1+Y_{v}^2) -1  \geq \frac{n-k}2 (1+Y_{B,v}^2)-1\geq \frac{n-k}4 (1+Y_{B,v}^2).
\ee
This implies
\begin{equation}\label{eqn-good100}
  |\tau| \int_{v_0}^{2\theta}  F^2\, e^{\mu}\, dv
  \leq  C   \int_{v_0}^{2\theta}  \frac{F_v^2}{1+Y_{v}^2} \, e^\mu \, dv,
\end{equation}
and
\be\label{eqn_good_bdryterm}
\frac{F(v_0)^2}{v_0}e^{\mu(v_0,\omega, \tau)}\leq C   \int_{v_0}^{2\theta}  \frac{F_v^2}{1+Y_{v}^2} \, e^\mu \, dv.
\ee
Finally, for $v\leq v_0$ by our choice of weight function we have
\be
\left|\mu(v,\omega,\tau)-\mu(v_0,\omega,\tau)-\log\left(\frac{v}{v_0}\right)^{n-k}\right| =\left| \int_{v}^{v_0}\frac{n-k}{v'}Y_{B,v'}^2 dv'\right| \leq C,
\ee
hence
\be
C^{-1} \left(\frac{v}{v_0}\right)^{n-k}\leq e^{\mu(v,\omega,\tau)-\mu(v_0,\omega,\tau)}\leq C \left(\frac{v}{v_0}\right)^{n-k}.
\ee
Together with the standard Poincar\'e inequality 
\be
\int_0^{v_0} (F(v)-F(v_0))^2\,  v ^{n-k} dv \leq C v_0^2 \int_0^{v_0} F_v^2\, v^{n-k} dv,
\ee
taking also into account \eqref{eqn_good_bdryterm}, this yields
\be
|\tau|\int_0^{v_0} F^2 e^{\mu} dv \leq C \int_0^{2\theta} \frac{F_v^2}{1+Y_v^2}e^\mu dv,
\ee
and thus concludes the proof of the corollary.
\end{proof}
We denote the inverse profile functions of $\mathcal{M}^1$ and $\mathcal{M}^2$ by $Y=Y_1$ and $\bY = Y_2$, respectively. Recall the weight $e^\mu$ is a function of $Y$. i.e. it is defined in terms of $\mathcal{M}^1$.   Consider the difference 
\be
W :=Y-\bY .
\ee
Utilizing the cut-off function  $\chi_{\mathcal{T}}$ from \eqref{chi_cut-off}, we define  
\be W_{\mathcal T}(v,\omega,\tau) := \chi_{\mathcal{T}}(v) \, W(v,\omega,\tau). \ee
 Using the evolution equation of $W$ in  Proposition \ref{lemma-ev-W-appendix}, we have the following energy inequality.
\begin{remark} \label{remark-cylindericalconnection}In what follows, we will express and prove identities using normal coordinates $(v,\varphi_1,\ldots, \varphi_{k-1})$ based at each given point. Here, $(\varphi_i)_{i=1}^{k-1}$ are normal coordinates of ${S}^{k-1}$. Within such a coordinate chart, we have followings basic identities: for given smooth function $G=G(v,\omega)$, at the base point of normal coordinates, 
\bea   \nabla _v G = G_v,&\quad  \nabla_{\varphi_m} G = G_{\varphi_m}, \\
 \nabla^2_{vv}G=G_{vv} ,\quad  \nabla^2_{v\varphi_m}G&= G_{\varphi_m v}, \quad \nabla^2_{\varphi_\ell \varphi_m} G= G_{\varphi_m \varphi_\ell}. \eea
As a consequence, $\nabla_{{S}^{k-1}}G= G_{\varphi_m} \partial_{\varphi_m}$, $|\nabla_{{S}^{k-1}} G|^2= G_{\varphi_m}^2$ and $\Delta_{{S}^{k-1}}G= G_{\varphi_m\varphi_m}$. Here, summation over repeated indices (from $1$ to $k-1$) is assumed.
 
\end{remark}

\begin{proposition}[energy inequality]\label{lemma-energy}
 There exist constants $\theta > 0$, $\kappa > 0$, $\tau_* > -\infty$
and $C = C(\theta) < \infty$ with the following significance. If $\mathcal{M}^1$ and $\mathcal{M}^2$ are  $\kappa$-quadratic at time $\tau_0 \le \tau_*$ and satisfies \eqref{p_plus_van},  then for $\tau \le \tau_0$ we have
\begin{align}
\frac12\frac{d}{d\tau}\, \int W_\cT^2 \mm  &\leq  - \frac{1}{20} \int \frac{| \partial_v W_\cT|^2}{1+ Y_v^2}\, \mm\\\nonumber
&+ \frac{C}{|\tau|}\, \int W^2 1_{\{ \theta\leq v\leq 2\theta\}}\, \mm
+ \int \,  \mathcal{I} W_\cT^2\, \mm ,
\end{align}
where
\be\label{abbrev_I}
\mathcal{I}=    \frac{ D}{Y^2}|\tilde{a}\partial_v|^2 + \frac{D}{1+Y_v^2}|\tilde{b}_m \partial_{\varphi_m} |^2+\tilde c ,
\ee
and the components of vector fields $\tilde a\partial_v $, $\tilde b_m\partial_{\varphi_m}$, and the coefficient $\tilde c $ are specified in equation \eqref{tildea}, \eqref{tildeb}, \eqref{tildec} below.
\end{proposition}
\begin{proof}
 Throughout this proof, we denote $\chi_\cT$ simply by $\chi$, which only depends on variable $v$, and denote $\int d\omega  dv$ simply by $\int$. To begin with, using Proposition \ref{lemma-ev-W-appendix} (evolution of $W$) and integration by parts we see that

\begin{align}\label{differential W inequality}
  &\quad \frac12 \frac{d}{d\tau}   \int  W_\cT^2 e^\mu \\
  &= -   \int \left( \frac{Y^2 + |\nabla_{S^{k-1}}Y|^2}{D} \, W_v^2  - \frac{2  Y_{\varphi_m}  Y_v}{D}   W_{\vp_m} W_v 
  \right) \chi^2 e^\mu\\
 &-\int\left(\frac{1 + |\nabla_v Y|^2}{D}   |\nabla_{S^{k-1}}W|^2  \right)\chi^2 e^\mu \nonumber\\
&+\int\left(\frac{Y^{-2}Y_{\vp_\ell}Y_{\vp _m}}{D}W_{\vp_\ell } W_{\vp_m}-\frac{Y^{-2}|\nabla_{S^{k-1}}Y|^2}{D}   |\nabla_{S^{k-1}}W|^2\right)\chi^2 e^\mu \nonumber\\
 &+ \int \left( \tilde aW_v  W+ \tilde b_m W_{\varphi_m}  W   +   \tilde c W^2\right) \chi^2e^\mu\nonumber \\
 &-\int\left(\frac{Y^2+|\nabla_{S^{k-1}}Y|^2}{D} W W_v -\frac{2Y_v Y_{\vp _m}}{D} W W_{\vp _m}   \right) (\chi^2)' e^{\mu} ,\nonumber
\end{align}
where 
\begin{align}\label{tildea}
    \tilde a &=  a  -  \left ( \frac{Y^2 + |\nabla_{S^{k-1}}Y|^2}{D} \right)_v  -  \left ( \frac{Y^2 + |\nabla_{S^{k-1}}Y|^2}{D} \right)\mu_v,
\end{align}
\begin{align}\label{tildeb}
    \tilde b_m &= b_m+ \left(\frac{2  Y_{\vp_m}  Y_v}{D} \right)_{v}  +\left(\frac{2  Y_{\vp_m} Y_v}{D} \right)  \mu_{v}-(Y^{-2})_{\vp_m} -Y^{-2}\mu_{\vp_m}\\\nonumber
&\, +\left(\frac{Y^{-2} Y_{\varphi_{\ell}}  Y_{\varphi_{m}} }{D}\right)_{\varphi_{\ell}}+\left(\frac{Y^{-2} Y_{\varphi_{\ell}} Y_{\varphi_m} }{D}\right)\mu_{\varphi_{\ell}} ,
\end{align}
\begin{align}\label{tildec}
    \tilde{c}=c+\frac{\mu_{\tau}}{2}.
\end{align}
and $a, b, c$ are given in \eqref{eq-anew}, \eqref{eq-bnew}, \eqref{eq-cnew}. 
With the aim of absorbing various mixed terms in \eqref{differential W inequality} we estimate
\begin{align}
\tilde a W_v W   \chi^2 &  \leq  \frac14  \frac{Y^2}{D}  W_v^2 \chi^2 +   \frac {D}{Y^2} |\tilde a|^2 W_\cT^2 ,\\
\tilde b_m W_{\vp_m} W    \chi^2  &\leq  \frac 14  \frac{1+Y_v^2}{D}  |\nabla_{{S}^{k-1}} W|^2  \chi^2+\frac {D}{1+  Y_v^2} |\tilde b_m \partial_{\vp_m}|^2  W_\cT^2 ,
\end{align}
and
\begin{align} \label{eq-YvYp}
 \frac{2  Y_v Y_{\vp_m}}{D} W_v   W_{\vp_m}\leq \frac14  \frac{Y^2}{D} W_v^2 +  \frac{1}{100} \frac{ Y_v^2}{D}|\nabla_{{S}^{k-1}} W|^2,
  \end{align}
where in the last step we used that $|\nabla_{S^{k-1}} Y|^2/Y^2\leq \eta$ thanks to Proposition \ref{lemma-mainY} (tip derivatives) and Theorem \ref{strong_uniform0} (uniform sharp asymptotics). 

Then, we notice that 
\begin{equation}
\nabla_{\varphi_\ell}Y\nabla_{\varphi_m}Y\nabla_{\varphi_\ell}W\nabla_{\varphi_m}W\leq |\nabla_{S^{k-1}}Y|^2|\nabla_{S^{k-1}}W|^2
\end{equation}
and use the above mixed terms estimates and discard some lower order good terms in \eqref{differential W inequality},  we obtain
Combining the above, and discarding some lower order good terms, gives
 \begin{align}\label{eqn-We2}
\frac12 \frac{d}{d\tau}  \int  W_\cT^2 e^\mu  \leq& -\frac12 \,  \int \left( \frac{Y^2 }{D} \, W_v^2 + 
 \frac{ Y_v^2}D  \, |\nabla_{S^{k-1}}W|^2 \right) \chi^2 e^\mu+\int \mathcal{I}\, W_\cT^2 e^\mu\nonumber\\
  &-2\int\left(\frac{Y^2+|\nabla_{S^{k-1}} Y|^2}{D} WW_v -\frac{2Y_{\varphi_{m}} Y_v}{D} W W_{\varphi_{m}} \right) \chi\chi' e^{\mu},
\end{align}
where $\mathcal{I}$ is as given in \eqref{abbrev_I}.
Now, using the fact that $W_v\, \chi = (W_\cT)_v-W\chi'$ we can estimate
\be
- W_v^2 \chi^2 \leq -\frac{1}{2} (W_\cT)_v^2 + 2  W^2\chi'^2.
\ee
Moreover, we also have 
\be
-W W_v \chi \chi' = -(W_\cT)_v  W \chi' + W^2 \chi'^2 \le \frac{1}{12} (W_\cT)_v^2 + 4\chi'^2 W^2  ,
\ee
and
\be
 \frac{Y_v Y_{\varphi_{m}}}{D} W W_{\varphi_{m}} \chi\chi' \le \frac{Y^2 }{D} W^2 \chi'^2 + \frac{1}{100} \frac{Y_v^2}{D} |\nabla_{S^{k-1}}W|^2\chi^2,
\ee
where in the last step we used again that $Y_\vp^2/Y^2\leq \eta$. This yields
 \begin{align}\label{eqn-We3}
\frac12 \frac{d}{d\tau}  \int  W_\cT^2 e^\mu  \leq&-\frac{1}{12} \,  \int \frac{Y^2 }{D} \, |\partial_v W_\cT|^2\,  e^\mu
+\int \mathcal{I}\, W_\cT^2\, e^\mu
+13\int \frac{Y^2}{D} W^2\chi'^2\, e^{\mu}.
\end{align}
where $\mathcal{I}$ is given by \eqref{abbrev_I}. From here, the desired energy inequality directly follows from the various derivative estimates in Proposition \ref{lemma-mainY} (tip derivatives), and this completes the proof of proposition.
\end{proof}
\begin{lemma}[estimate for coefficients]
\label{lemma-coefficients}
For every $\eta >0$, there exist constants $\kappa > 0$, $\tau_* > -\infty$ and $\theta > 0$ with the following significance.
If $\mathcal{M}^1$ and $\mathcal{M}^2$ are $\kappa$-quadratic at time $\tau_0 \le \tau_*$ and satisfies \eqref{p_plus_van},  then for $\tau \le \tau_0$ and $v\leq 2\theta$ we have
\be
    (1+Y_v^2 ) |\tilde{a}\partial_v |^2+     |\tau||\tilde{b}_m \partial_{\vp_m}  |^2  +|\tilde c | \leq \eta \, |\tau|.
\ee
In particular, the quantity $\mathcal{I}$ defined in \eqref{abbrev_I} satisfies $\mathcal{I}\leq 3\eta|\tau|$.
\end{lemma}

\begin{proof} The proof follows similarly as \cite[Lem 4.13]{CDDHS} except some new terms in the coefficients that have not appeared in the case $k=2$. We first point out those terms for those readers familiar with \cite[Lem 4.13]{CDDHS}. Note all $\tilde a$, $\tilde b_m$, $\tilde c$, have a common term\begin{equation}
  \frac{ |\nabla_{S^{k-1}}\bar Y|^2}{\bar Y} -\bar B,
\end{equation}
where
\begin{align}\label{barB}
    \bar{B}&=(\bar Y^2 + |\nabla_{S^{k-1}}\bar Y|^2)\nabla^2_{vv}\bar Y  - 2 \nabla^2_{\varphi_m v}\bar Y\nabla_{\varphi_m}\bar Y\nabla_v \bar Y \\
    &-\bar Y^{-2}\nabla^2_{\varphi_m\varphi_{\ell}}\bar Y\nabla_{\varphi_\ell}\bar Y\nabla_{\varphi_m}\bar Y.
\end{align}The last term in $\bar B$  has been changed from the previous last term of $\bar B$ in \cite{CDDHS}, namely $+(1+\bar Y_v^2)\bar Y_{\varphi\varphi}$.  This induces more natural expressions for remaining terms in $a$, $b$, $c$. As a consequence of this change and high dimensionality in $k$, first in $a$, the previously existing term $ \frac{\bar Y_{\varphi\varphi} (Y_v+\bar Y_v) }{D}$ is absorbed into $\bar B$. In $\tilde{b}_m$, we have new terms 
\begin{align} \label{eq-bnew1}
    -\frac{\nabla_{\varphi_\ell}(Y+\bar{Y})\nabla^2_{\varphi_m\varphi_{\ell}}\bar{Y}}{DY^2} ,
\end{align}
and
\begin{align} \label{eq-bnew2}
     &-\left(\frac{Y^{-2}|\nabla_{S^{k-1}}Y|^2}{D}\right)\mu_{\vp_m} +\left(\frac{Y^{-2} Y_{\vp_\ell }  Y_{\vp_m}}{D}\right)\mu_{\vp_\ell }
\\ \label{eq-bnew3} 
&-\left(\frac{Y^{-2}|\nabla_{S^{k-1}}Y|^2}{D}\right)_{\vp_m} +\left(\frac{Y^{-2} Y_{\vp_\ell }  Y_{\vp_m}}{D}\right)_{\vp_\ell }.\end{align}

In $\tilde c$, we have new terms 
\begin{align}\label{eq-cnew'}
   -\frac{(Y+\bar{Y})}{Y^{2}\bar{Y}^{2}}\Delta_{S^{k-1}}\bar{Y}+ \frac{Y+\bar{Y}}{DY^{2}\bar{Y}^{2}}\bar{Y}_{\vp_m\vp _\ell }\bar{Y}_{\vp_\ell }\bar{Y}_{\vp_m}.
\end{align}

To begin with the proof, we recall and list some derivative estimates used in \cite[(4.131)-(4.138)]{CDDHS} to estimate the above new additional terms. Let $o(1)$ denote a quantity that can be made arbitrarily small in the tip region $v\leq 2\theta$, by choosing $\kappa$ small enough, $\tau_*$ negative enough and $\theta$ small enough.  Moreover, we use
\be
f\sim g \qquad :\Leftrightarrow \qquad \left|\frac{f}{g}-1\right|=o(1).
\ee
In particular, by definition we have
\be
v=o(1),
\ee
and thanks to Theorem \ref{strong_uniform0} (uniform sharp asymptotics) we get
\be\label{Y_tip_estimates}
Y\sim \bar{Y}\sim \sqrt{2|\tau|}\, .
\ee
Recall that by Proposition \ref{lemma-mainY} (tip derivatives) we have the following estimates on the whole tip region $v \le 2\theta$:
\be
|\nabla_{S^{k-1}}Y|=o(1)|\tau|^{\frac12  },
\ee
\be\label{D_estimates}
D\sim Y^2(1+|\nabla_v Y|^2)\sim \bar{Y}^2(1+|\nabla_v\bar{Y}|^2)\sim\bar{D},
\ee
\be
\qquad \tfrac14 v|\tau|^{\frac12 }\leq |\nabla_v Y|\leq v|\tau|^{\frac12 }\, ,
\ee
\be\label{eqn-uphi_rest}
\left|\frac{\nabla^2_{vv}Y}{1+|\nabla_v Y|^2}\right|  \le C|\tau|^{\frac12 }, \,\,\, 
\left|\frac{\nabla_v Y \nabla_v\nabla_{S^{k-1}} Y }{1+|\nabla_v Y|^2}\right|  =o(1) |\tau|, \,\,\, |\nabla^2_{S^{k-1}}Y|  =o(1) |\tau|^{\frac32 },
\ee
\be\label{eqn-imp_rest}
v \nabla_v \mu \sim 1+|\nabla_v Y|^2,\qquad
 |\nabla_{S^{k-1}}\mu  | + |\mu_\tau | =o(1) |\tau|\, ,
\ee
and the same estimates hold for $\bar Y$ as well.


To estimate $\tilde a$, recall 
\begin{align}
    \tilde a &=  a  -  \left ( \frac{Y^2 + |\nabla_{S^{k-1}}Y|^2}{D} \right)_v-  \left ( \frac{Y^2 + |\nabla_{S^{k-1}}Y|^2}{D} \right)\mu_v,
\end{align}
where
\begin{align}
    a&=\frac{{n-k}} {v} - \frac v 2 -2 \frac{\bar Y_{{\varphi_m} v}Y_{\varphi_m}}{D}  + \frac{\bar Y^2 (Y_v +\bar Y_v)}{D\bar D} \left[ \frac{|\nabla_{S^{k-1}}\bar Y|^2}{\bar Y} -\bar B\right].
\end{align} 

Rearrange terms in $\tilde a$ so that $\tilde a = (\tilde a-\tilde a_1-\tilde a_2)+\tilde a_1+ \tilde a_2$, where 
\bea   \tilde a_1 &= \frac{n-k}{v} - \frac{Y^2+|\nabla _{S^{n-k}} Y|^2}{D} \mu _v \\
\tilde a_2& = \frac{2Y^2Y_vY_{vv} }{D^2} - \frac{\bar Y^4 (Y_v+\bar Y_v) \bar Y_{vv}}{D\bar D }.\eea

To estimate $(1+Y_v^2 )\tilde a_1$, note that, in the soliton region $\!\! v\!\le\! L /|\tau|^{1/2} $, $\mu_v= {(n-k)}({1+Y_{B,{v}}^2})/v$ and hence
 \bea |\tilde{a}_1| &= \frac{n-k}{v}\left \vert \frac{Y^2 Y_v^2 - Y^2 Y_{B,v}^2 -|\nabla_{S^{k-1}}Y|^2 Y_{B,v}^2 }{Y^2 (1 +Y_v^2)+|\nabla_{S^{k-1}}Y|^2} \right \vert \\
&\le \frac{n-k}{v}\left \vert  1-\frac{1+Y_{B,v}^2}{1+Y_v^2} \right \vert +(n-k)\frac{Y_{B,v}^2}{v} \frac{|\nabla_{S^{k-1}}Y|^2}{Y^2}  =o(1)|\tau|^{1/2},
\eea
On the other hand, in the collar region ${L}/{\sqrt{|\tau|}} \le v\le 2\theta$ using weight estimates \eqref{eqn-imp}, we get
\be
\left| {(n-k)} -v\mu_v \frac{Y^2+|\nabla_{S^{k-1}}Y|^2}{D}\right| = o(1).
\ee
Observing also that $v^{-1}\sqrt{1+Y_v^2}\leq C|\tau|^{1/2}$ in the collar region, this yields
\be
\sqrt{1+Y_v^2} |\tilde{a}_1|=o(1)|\tau|^{1/2}.
\ee

To estimate $\sqrt{1+Y_v^2}|\tilde{a}_2|$, motivated by the product rule for differences, we rewrite $\tilde{a}_2$ in the form
\begin{multline}
\tilde{a}_2=\frac{2 (Y^4-\bar Y^4)Y_vY_{vv}}{D\bar D}
+\frac{\bar Y^4  (Y_{v}- \bar Y_{v})\bar Y_{vv} }{D\bar D}\\
+\frac{2 \bar Y^4 Y_{v}(Y_{vv}-\bar Y_{vv})}{D\bar D} 
-\frac{2Y^4Y_vY_{vv}}{D}\frac{D-\bar D}{D\bar D}.
\end{multline}
The contribution from the first term can be readily estimated observing that $|Y^4-\bar{Y}^4|=o(1)Y^4$. To deal with the second term we use that  
$-Y_v \sim  \frac{vY}{2(n-k)} \sim -\bar Y_v$ in the collar region thanks to Corollary \ref{prop-great} (almost Gaussian collar) and that $|Y_v-\bar Y_v|=o(1)$ in the soliton region thanks to \eqref{der_sol1}. To deal with the third term we use that $|Y_{vv}|=\eps (1+Y_v^2)|\tau|^{1/2}$ in the collar region thanks to \eqref{eqn-cylf} (second tip derivatives), where $\eps$ can be made arbitrarily small by adjusting the parameters, in particular choosing $L$ large enough, and that $|Y_{vv}-\bar Y_{vv}|=o(1)|\tau|^{1/2}$ in the soliton region thanks to \eqref{der_sol1}. Finally, to deal with the last term we use that $D-\bar{D}=o(1)D$ thanks to \eqref{eq-soliton-improvement}. Summing up, this yields
\be
\sqrt{1+Y_v^2} |\tilde{a}_2|=o(1)|\tau|^{1/2}.
\ee
Finally regarding $\tilde a- \tilde a_1-\tilde a_2$,  a direct computation shows that
\begin{align}
&\tilde{a}-\tilde{a}_1-\tilde{a}_2 =-2 \frac{(Y_{\vp_m v}+\bar Y_{\varphi_m v})Y_{\varphi_m} }{D} -2 \frac{ YY_v}{D}- \frac v 2\nonumber  \\
 &+ \frac{ 2Y^2|\nabla _{{S}^{k-1}} Y|^2  Y_v Y_{vv}}{D^2} +\frac{2(Y^2 + |\nabla _{{S}^{k-1}} Y|^2 )\left(Y_{\vp_m} Y_{\vp_m v} +  Y Y_v (1+ Y_v^2)\right)}{D^2}\nonumber\\
 &- \frac{\bar Y^2 (Y_v +\bar Y_v)}{D\bar D} \!\left(\!|\nabla _{{S}^{k-1}} \bar Y|^2\bar Y_{vv} \!-\! 2\bar Y_{\varphi_m} \bar Y_v \bar Y_{\varphi_m v}  \!-\!  \frac{\bar Y_{\varphi_m}\bar Y_{\varphi_\ell }  \bar Y_{\varphi_m\varphi_{\ell}} }{\bar Y^{2}}\! -\!\frac{|\nabla _{{S}^{k-1}} \bar Y|^2}{\bar Y}\right)\!.
\end{align}

Most of these terms above, once they are multiplied by $\sqrt{1+Y_v^2}$, can be estimated using derivative estimates to give $o(|\tau|^{1/2})$. The only three terms for which one has to argue somewhat differently are $v$ and $(Y_{\vp_m v}+\bar{Y}_{\vp_m v}) Y_{\vp_m}/D$ and $(Y^2+|\nabla_{{S}^{k-1}}Y|^2 )Y_{\vp_m} Y_{\vp_m v}/D^2$. Regarding the first term, since $v=o(1)$ and $\sqrt{1+Y_v^2}=o(1)|\tau|^{1/2}$ we easily get
\be
\sqrt{1+Y_v^2}\,v=o(1)|\tau|^{1/2}.
\ee
To deal with the second term, in the collar region we estimate
\begin{multline}
\left|\sqrt{1+Y_v^2}\frac{(Y_{\vp_m v}+\bar{Y}_{\vp_m v})Y_{\vp_m}}{D}\right| \\
\sim  \left|\frac{1}{2|\tau|}\frac{(Y_{\vp_m v}+\bar{Y}_{\vp_m v})Y_v}{(1+Y_v^2)} \frac{\sqrt{1+Y_v^2}}{Y_v}Y_{\vp_m}\right| = o(1)|\tau|^{1/2}.
\end{multline}
 On the other hand, thanks to \eqref{eqn-uphi} of Proposition \ref{lemma-mainY} in the soliton region $v\leq L/|\tau|^{1/2}$ we have have the sharper estimate $|Y_{\vp v}|+|\bar Y_{\vp v}|=o(1)|\tau|$, so we also get the desired bound in the soliton region. Finally, since $(Y^2+|\nabla_{{S}^{k-1}}Y|^2)/D\leq 1$ the same argument applies for the third term as well. This yields
$
\sqrt{1+Y_v^2} \,|\tilde{a}-\tilde{a}_1-\tilde{a}_2|=o(1)|\tau|^{1/2}.$

Next, we estimate $\tilde b_m \partial_{\vp_m}$ and $\tilde c$. First, by \eqref{Y_tip_estimates}, \eqref{D_estimates}, \eqref{barB}, \begin{equation}\label{B/D}
    \frac{1}{D}\left[ \frac{ |\nabla_{S^{k-1}}\bar Y|^2}{\bar Y} -\bar B\right]\leq C|\tau|^{\frac12 }.
\end{equation}
The estimates \eqref{Y_tip_estimates}-\eqref{eqn-uphi_rest} and \eqref{B/D} can be used to show 
 each term in $c$ (see \eqref{eq-cnew}) is of order $o(1)$. (For readers' reference, see \cite[(4.141)]{CDDHS}.) In particular, new terms introduced in \eqref{eq-cnew'} are \bea
   \left|  -\frac{(Y+\bar{Y})}{Y^{2}\bar{Y}^{2}}\Delta_{S^{k-1}} \bar{Y} \right| \le C |\tau|^{-3/2}|\nabla ^2_{{S}^{k-1}} \bar Y |= o(1),
\eea 
\bea  \left| \frac{Y+\bar{Y}}{DY^{2}\bar{Y}^{2}}\nabla^2_{\varphi_m\varphi_{\ell}}\bar{Y}\nabla_{\varphi_\ell}\bar{Y}\nabla_{\varphi_m}\bar{Y} \right| \le C|\tau|^{-3/2}\frac{|\nabla_{{S}^{k-1}} \bar Y|^2 }{D} |\nabla^2_{{S}^{k-1}} \bar Y | =o(1).  \eea 
The only remaining term is $\tilde c -c = \mu _\tau/2  $, and this is of order $o(|\tau|)$ by \eqref{eqn-imp_rest}. This shows $\tilde c = o(|\tau|)$. 

In the estimate of $\tilde b_m \partial_{\vp_m}$, note that each new term in \eqref{eq-bnew1} and \eqref{eq-bnew2} is of order $o(1)$ by \eqref{Y_tip_estimates}-\eqref{eqn-imp_rest}. To estimate \eqref{eq-bnew3}, note by similar estimates as in \cite[(4.147)]{CDDHS} we have  $| {\nabla_{S^{k-1}}D}|/D=o(1)|\tau|$, and hence 
\bea \label{eqquotientderivative1}
&\left|\Big (\frac{Y^{-2}|\nabla_{S^{k-1}}Y|^2}{D}\Big )_{\vp_m} \right|+\left|\Big (\frac{Y^{-2} Y_{\vp_\ell}  Y_{\vp_m}}{D}\Big )_{\vp_{\ell}} \right|\\\nonumber
&\leq C\Big (   \frac{|\nabla_{S^{k-1}}Y|^3}{Y^3 D}+\frac{|\nabla_{S^{k-1}}Y||\nabla ^2 _{S^{k-1}}Y|}{Y^2D}+\frac{|\nabla_{S^{k-1}} Y|^2}{Y^2D} \left|\frac{\nabla_{S^{k-1}}D}{D}\right|\Big  ) =o(1).
\eea 
Each of remaining terms in $\tilde b_m$ can be estimated as $o(1)$ by  \eqref{B/D} and \eqref{Y_tip_estimates}-\eqref{eqn-imp_rest}. (For readers' reference, see \cite[(4.148)]{CDDHS}.)
We have $    |\tilde{b}|=o(1)$.
This completes the proof. 
\end{proof}
\begin{proposition}[energy estimate in tip region]\label{prop-tip}  For every $\eps>0$, there exist $\kappa > 0$ small enough, $\theta > 0$ small enough and $\tau_* > -\infty$ negative enough with the following significance. If $\mathcal{M}^1$ and $\mathcal{M}^2$ are $\kappa$-quadratic at time $\tau_0 \le \tau_*$ and satisfies \eqref{p_plus_van},  then for $\tau \le \tau_0$ we have
  \begin{equation}
    \label{eqn-tip}
   \big \| W_\cT \big\|_{2,\infty} \leq \eps \, \big\| W \, 1_{ \{ \theta \leq v \leq 2\theta\} } \big \|_{2,\infty}.
  \end{equation}
\end{proposition}
\begin{proof} 
The proof follows similarly as the proof of  \cite[Prop 4.14]{CDDHS} by applying the above Proposition \ref{lemma-energy} (energy inequality), Proposition \ref{prop-Poincare} (Poincare inequality) and  Lemma \ref{lemma-coefficients} (estimate for coefficients).

In view of Lemma \ref{lemma-coefficients} (estimate for coefficients) and Proposition \ref{prop-Poincare} (Poincare inequality), the energy inequality in Proposition \ref{lemma-energy} (energy inequality) becomes 
\begin{equation}
\frac{d}{d\tau}\, \int W_\cT^2 \, e^{\mu} \le - c |\tau|  \int W_\cT^2 \, e^{\mu} +\frac{C}{|\tau|}\int W^2 1_{\{ \theta\leq v\leq 2\theta\}}\, e^{\mu},
\end{equation}
for some constants $c>0$ and $C<\infty$ depending on $\theta$. 
Considering
\begin{align}
&A(\tau)=\int^{\tau}_{\tau-1} \int W_\cT^2 \, e^{\mu} , && B(\tau)=\int^{\tau}_{\tau-1}  \int W^2 1_{\{ \theta\leq v\leq 2\theta\}}\, e^{\mu},
\end{align}
it follows that
\begin{equation}
\frac{d}{d\tau} \left[ e^{-\frac{c\tau^2}{2}} A(\tau) \right] \leq C|\tau| e^{-\frac{c\tau^2}{2}}\frac{B(\tau)}{|\tau|^2} .
\end{equation}
Integrating this from $-\infty$ to $\tau$ yields
\begin{align}
 A(\tau) &\leq C \sup_{\tau'\leq \tau}|\tau'|^{-2} B(\tau') \, ,
\end{align}
and hence in particular
\begin{equation}
|\tau|^{-\frac{1}{2}}A(\tau)\leq C|\tau|^{-2}\sup_{\tau'\leq \tau}|\tau'|^{-\frac{1}{2}}B(\tau')\, .
\end{equation}
This shows that
  \begin{equation}
    \label{eq-nice}
   \big \|W_\cT\big\|_{2,\infty} \le \frac{C}{|\tau_0|}\big \|W \, 1_{ \{ \theta \leq v \leq 2\theta\} } \big\|_{2,\infty}\, ,
  \end{equation}
and thus concludes the proof of the proposition.
\end{proof}

\subsection{Proof of spectral uniqueness theorem and reflection symmetry}
\label{sec-conclusion}
In this subsection, we prove our spectral uniqueness theorem by generalizing the arguments from \cite[Section 8]{ADS2} and \cite[Section 5.6]{CHH} to the cylinder setting. We first recall the following coercivity estimate theorem and lemma in \cite[Sec 4]{CDDHS} used for proving spectral uniqueness theorem.
\begin{theorem}[{coercivity estimate, c.f. \cite[Prop 4.15]{CDDHS}}]
  \label{prop-cor-main}
  For every $\varepsilon > 0$ there exist $\kappa > 0$ small enough, $\theta > 0$  small enough in definitions of cylindrical region $\mathcal{C}$ and tip region $ \mathcal{T}$ and $\tau_* > -\infty$ negative enough such that if $\mathcal{M}^{1}$ and $\mathcal{M}^{2}$ are $k$-ovals which are $\kappa$-quadratic at  time $\tau_0\leq \tau_{*}$ and satisfy \eqref{p_plus_van}, then for $\tau \le \tau_0$ we have
\be
\|w_\cC-\mathfrak{p}_0 w_\cC \|_{\hD,\infty}+\|W_\cT\|_{2,\infty} \leq  \eps \|\mathfrak{p}_0 w_\cC \|_{{\hD},\infty}\, .
\ee
\end{theorem}

\begin{proof}
As in the proof of \cite[Prop 4.15]{CDDHS}, the proposition follows from the Proposition \ref{prop-cyl-est} (energy estimate in cylindrical region), Proposition \ref{prop-tip}(energy estimate in tip region) and the ratio estimates of $w$ and $W$ in transition region $\theta\leq v\leq 2\theta$. 

First of all, to compare our different norms in the transition region, let us abbreviate
\be
f\sim_\theta g\quad :\Leftrightarrow \quad \exists C=C(\theta)<\infty \,: \,C^{-1}f \leq g\leq Cf.
\ee
Then, by Theorem \ref{strong_uniform0} (uniform sharp asymptotics) and convexity we have
\be 
 \big|\partial_y v_{i}(y,\omega ,\tau)\big|\sim_\theta \frac{1}{|\tau|^{1/2}} \qquad \textrm{for}\,\, \theta \leq v_{i}       (y,\omega,\tau) \leq 2\theta,
 \ee
for sufficiently small $\kappa>0$ and $\tau_*>-\infty$.
By the mean value theorem this implies
\be 
\left|\frac{w(Y_1(v,\omega,\tau),\omega,\tau)}{W(v,\omega,\tau)}\right |\sim_\theta \frac{1}{|\tau|^{1/2}} \qquad \textrm{for}\,\, \theta \leq v \leq 2\theta.
 \ee
Hence, by the change of variables formula we infer that
\be
\frac{1}{|\tau|^{1/2}}\int_{\theta}^{2\theta} (W^2 e^{\mu})(v,\omega,\tau) \, dv \sim_\theta \int_{Y_1(\theta,\omega,\tau)}^{Y_1(2\theta,\omega,\tau)} w^2(y,\omega,\tau) e^{-y^2/4} dy .
\ee
Remembering \eqref{def_norm_tip_r}, this shows that 
\begin{equation}\label{eq_equiv_of_norms}
C^{-1}\big\| W 1_{ \{ \theta \leq v \leq 2\theta\} } \big\|_{2,\infty}\leq \big\| w\, 1_{\{\theta\leq v_1(\tau)\leq 2\theta\}} \big\|_{\mathcal{H},\infty}\leq  C\big\| W 1_{ \{ \theta \leq v \leq 2\theta\} }\big\|_{2,\infty}\, .
\end{equation} Proposition \ref{prop-tip} (energy estimate in tip region) combined with \eqref{eq_equiv_of_norms} yields
  \begin{equation}
    \label{eqn-tip_rr}
    \big\| W_\cT \big\|_{2,\infty} \leq  C\eps \big\| w \, 1_{ \{ \theta \leq v_1 \leq 2\theta\} }  \big\|_{\mathcal{H},\infty} \leq C\eps  \big\| w_{\cC}  \big\|_{\mathcal{H},\infty},
  \end{equation}
where in the last step we used that $w(y,\omega,\tau)=w_{\cC}(y,\omega,\tau)$ for $v_1(y,\omega,\tau)\geq \theta$, provided that $\kappa$ is small enough and $\tau_\ast$ is negative enough.

Similarly, by \eqref{eq_equiv_of_norms} applied with $\theta/2$ instead of $\theta$, we have
\be\label{w_trans_est}
 \big\| w \, 1_{ \{ \theta/2 \leq v_1 \leq \theta\} } \big\|_{\mathcal{H},\infty}  \leq C\big\| W_\cT \big\|_{2,\infty}\ ,
\ee
where we also used that $W(v,\omega,\tau)=W_{\cT}(v,\omega,\tau)$ for $v\leq \theta$, and so Proposition \ref{prop-cyl-est} (energy estimate in cylindrical region) gives
\begin{equation}
\label{eqn-cylindrical1_r}
 \big\Vert w_\cC - \mathfrak{p}_0w_\cC \big \Vert_{\hD ,\infty } \le \eps    \big\Vert  w_\cC \big\Vert_{\hD ,\infty} +C\eps  \big\Vert W_{\cT}  \big\Vert_{2,\infty}  .
\end{equation}
Finally, by the triangle inequality we clearly have
\begin{equation}
 \big\|w_\cC  \big\|_{\hD,\infty}\leq  \big\|w_\cC-\mathfrak{p}_0 w_\cC  \big\|_{\hD,\infty}+ \big\|\mathfrak{p}_0 w_\cC  \big\|_{\hD,\infty}\, .
\end{equation}
Combining the above inequalities, and adjusting $\varepsilon$, the assertion follows.
\end{proof}

We also need the following standard derivative estimates:
\begin{lemma}[{derivative estimates, \cite[Lem 4.16]{CDDHS}}]\label{std_pde_der_est}
For any $L<\infty$ and $\theta>0$, there exist $\kappa>0$, $\tau_\ast>-\infty$ and $C<\infty$ with the following significance. If  $\mathcal{M}$ is $\kappa$-quadratic at time $\tau_0 \le \tau_*$ and satisfies \eqref{p_plus_van}, then for all $\tau\leq\tau_0$ we have
\be\label{very_good_grad}
\sup_{y\leq L} \left(|Dv(\tau)|+|D^2v(\tau)|\right)\leq \frac{C}{|\tau|},
\ee
and
\be\label{good_grad}
\sup_{v(\tau)\geq \theta/2} \left(|Dv(\tau)|+|D^2v(\tau)|\right)\leq \frac{C}{|\tau|^{1/2}}.
\ee
\end{lemma}
\begin{proof}
The lemma from standard interior estimates for evolution equation of profile function $V(x, t)=\sqrt{|t|}v(\sqrt{|t|}y, -\log|t|)$ of the unrescaled mean curvature flow: \begin{equation}
\partial_t V = \left(\delta_{ij}-\frac{V_i V_j}{1+|DV|^2}\right)V_{ij} - \frac{n-k}{V}.
\end{equation}
which in the parabolic ball $P({\bf x}_{0}, t_{0},\sqrt{-t_0})$ satisfies
\be
C^{-1}\leq\frac{V}{\sqrt{-t_0}}\leq C,\qquad |DV| \leq \frac{C}{\sqrt{\log(-t_0)}}.
\ee
due to  Theorem \ref{strong_uniform0} (uniform sharp asymptotics) and convexity. 
\end{proof}
Now, we prove the spectral uniqueness theorem, which follows from similar argument as in \cite[Proof of Theorem 1.12]{CDDHS}. For completeness, we include its proof here.

\begin{proof}[Proof of Theorem \ref{thm:uniqueness_eccentricity_intro} (spectral uniqueness)]
Let us denote orthogonal neutral eigenfunctions by
\begin{equation}\label{psi_defi_2}
    \psi_{ii}=y^2_{i}-2\quad\text{ and }\quad  \psi_{ij}=2 y_{i}y_{j}\,  \text{ if }i\neq j ,
\end{equation}
so that 
\be
\mathfrak{p}_0w_\cC(\tau)= \sum_{ i\le j } a_{ij}(\tau) \psi_{ij}\, .
\ee
Let $\bar{A}=(a_{ij})$ be the symmetric $k\times k$ spectral coefficients matrix with
\begin{equation}\label{def_exp_coeffs_2}
 a_{ij} = \|\psi_{ij}\|_{\mathcal{H}}^{-2} \langle  \psi_{ij},w_{\cC}\rangle_{\mathcal{H}}.
\end{equation}
recall from \eqref{evolution-wC} (evolution of $w_\cC$) that we can rewrite the evolution of  $w_\cC$ in the following form  
\bea 
(\partial_\tau -\mathcal{L}) w_\cC=\mathcal{E}[w_\cC]+\overline{\mathcal{E}}[w,\chi_\cC(v_1)]+J+K,
\eea 
where $J$ and $K$ are defined by
\bea \label{eqn-IJK-2} 
&J \, =  (v_{2,\tau} - v_{2,y_iy_i}  +\tfrac{1}{2}  y_i v_{2,y_i}  -\cE[v_2] ) w^{\chi_\cC} -  2 D v_2 \cdot Dw^{\chi_\cC},\\
&K\, = \cE[v_2] w^{\chi_\cC} - \cE[v_2w^{\chi_\cC}] + v_2\, (\partial_\tau -\mathcal{L})w^{\chi_\cC}.
\eea 
 Since $\mathcal{L}\psi_i=0$, this implies
\begin{equation}\label{ODE_a_i}
\frac{d}{d\tau}a_{ij}(\tau)=\left\langle \mathcal{E}[w_\cC]+\overline{\mathcal{E}}[w,\chi_\cC(v_1)]+J+K,\frac{\psi_{ij}}{\|\psi_{ij}\|^2}\right\rangle\, .
\end{equation}
 An elementary computation on the inner product of these neutral mode eigenfunctions shows that for the indices $i, j,p,q, \ell, m \in \{ 1, \ldots , k\}$, the only nonvanishing terms of $\langle \psi_{ij}\psi_{pq}, \psi_{\ell m} \rangle $  have the following relations: for pairwise disjoint indices $i$, $j$, $\ell$, $m$, 
\bea 
&   \langle \psi_{ii}\psi_{ii}, \psi_{ii} \rangle = 8\|\psi_{ii}\|_\mathcal{H}^2, \quad     \langle \psi_{ij}\psi_{ij}, \psi_{ii} \rangle = 8\|\psi_{ii}\|_\mathcal{H}^2  , 
\quad     \langle \psi_{im}\psi_{i\ell}, \psi_{\ell m} \rangle = 8\|\psi_{ii}\|_\mathcal{H}^2  .
\eea Combining above with $ \|{\psi_{ii}}\|_\mathcal{H}^2 = \frac{1}{2}\|{\psi_{\ell m }}\|_\mathcal{H}^2$, we can thus rewrite \eqref{ODE_a_i} in the form
\be\label{aij_ODE_F}
  \frac{d}{d\tau}a_{ij}(\tau) = \frac{2a_{ij}(\tau)}{|\tau|} + F_{ij}(\tau),
\ee
where
  \begin{align}\label{eq-F-tau}
      F_{ij} &=  \left\langle \cE[w_\mathcal{C}] -  \frac{|\mathbf{y}|^2-2k}{4|\tau |}\mathfrak{p}_0w_\cC,  \frac{\psi_{ij}}{\|\psi_{ij}\|^2}\right\rangle\!+\!\left\langle\bar{\cE}[w,\varphi_\mathcal{C}(v_1)] + J + K, \frac{\psi_{ij}}{\|\psi_{ij}\|^2}\right \rangle\!.
  \end{align}
Solving these ODEs, taking into account the fact that $a_{ij}(\tau_0)=0$ thanks to the spectral assumption \eqref{cent_0}, we get
\begin{equation}\label{eq_solution_a}
a_{ij}(\tau)=-\frac{1}{\tau^2}\int_{\tau}^{\tau_0} F_{ij}(\sigma)\sigma^2\, d\sigma.
\end{equation}
In the following, we use the notation 
\begin{equation} \label{eq-tildeA}
 \bar {A}(\tau):= \sup_{\tau'\leq \tau}\left( \int_{\tau'-1}^{\tau'}  \sum_{i\le j} | a_{ij} (\sigma )|^2 \, d\sigma \right)^{1/2}\, .
\end{equation}

\begin{claim} \label{claim411}
    For $\varepsilon>0$, there exist $\kappa>0$ and ${\tau}_{\ast}>-\infty$, such that assuming $\kappa$-quadraticity at $\tau_0\leq \tau_\ast$, for $\tau\leq\tau_0$ we have
\begin{equation}
\int_{\tau-1}^\tau |F_{ij}(\sigma)|\, d\sigma \leq \frac{\varepsilon}{|\tau|} \bar  A(\tau_0).
\end{equation}
\end{claim}
\begin{proof}[Proof of Claim]

     By Theorem \ref{prop-cor-main} (coercivity estimate) we have
\be\label{cons_from_coerc}
\|w_\cC-\mathfrak{p}_0 w_\cC \|_{\hD,\infty}+\|W_\cT\|_{2,\infty} \leq  \eps C{\bar  A}(\tau_0).
\ee

Now, fix a smooth cut-off function $\bar{\chi}:\mathbb{R}^+\to [0,1]$ such that $\bar{\chi}(v)=1$ if $v\in [\tfrac{9}{16}\theta,\tfrac{15}{16}\theta]$ and $\bar{\chi}(v)=0$ if $v\notin [\tfrac{\theta}{2},\theta]$. Then, since $\textrm{spt}(\chi_\cC')\subseteq[\tfrac{5}{8}\theta,\tfrac{7}{8}\theta]$ using the estimates     
\eqref{eq-I1} \eqref{second line K} and \eqref{abcd estimates eq} we infer that
    \begin{align}\label{eq-barE-est}
        \big|\langle \bar{\cE}[w(\tau),\chi_\mathcal{C}&(v_1(\tau))] + J(\tau) + K(\tau) , \psi_i\rangle\big| \nonumber\\
&\leq     \big\|\bar{\cE}[w(\tau),\chi_\mathcal{C}(v_1(\tau))] + J(\tau)  + K(\tau) \big\|_{\hD^*} \big\|\psi_i \, \bar {\chi}(v_1) \big\|_{\hD} \nonumber\\
        &\leq \eps \, \big\| w(\tau)  1_{\{\theta/2\leq v_1(\tau)\leq \theta\} }\big\|_{\mathcal{H}} \, e^{-|\tau|/4}.
\end{align}
Together with \eqref{w_trans_est} and \eqref{cons_from_coerc} this yields
\be
\int_{\tau-1}^\tau\left| \langle\bar{\cE}[w,\varphi_\mathcal{C}(v_1)] + J + K, \psi_i \rangle\right |\, \leq \frac{\eps}{|\tau|}{\bar  A}(\tau_0).
\ee

\medskip
Next, by \eqref{eqn-E15} we have
\be
 \cE[w_\cC]-  \frac{\psi_0}{4|\tau |}\mathfrak{p}_0w_\cC =  \mathcal{D}[w_\cC] + \frac{2(n-k)-v_1v_2}{2v_1v_2}w_\cC-  \frac{\psi_0}{4|\tau |}\mathfrak{p}_0w_\cC,
\ee
where
\begin{multline}
\mathcal{D}[w_\cC] =  - \frac{Dv_1\otimes Dv_1\! :\! D^2 w_\cC}{1+|Dv_1|^2 } - \frac{D^2v_2\! :\! D(v_1+v_2)\!\otimes\! Dw_\cC}{(1+|Dv_1|^2)} \\
+\frac{D^2v_2\! :\! Dv_2\otimes Dv_2\, D(v_1+v_2)\!\cdot\! Dw_\cC}{(1+|Dv_1|^2)(1+|Dv_2|^2)}.
 \end{multline}
 
Now, using Lemma \ref{std_pde_der_est} (derivative estimates) we see that
\be
\big| \langle \mathcal{D}[w_\cC(\tau)] ,\psi_i\rangle\big| \leq \frac{\eps}{|\tau|} \| w_\cC(\tau) \|_{\mathcal{D}}.
\ee
Indeed, when computing  $\langle \mathcal{D}[w_\cC(\tau)], \psi_i \rangle$, the contribution to the integrals from the region $y\leq L$ decays quadratically in $|\tau|$ thanks to \eqref{very_good_grad}, and the contribution from the region $y\geq L$ can be bounded by $\eps(L)/|\tau|$ thanks to \eqref{good_grad} and the Gaussian weight. Hence,
\be
\int_{\tau-1}^\tau \big| \langle \mathcal{D}[w_\cC], \psi_i \rangle \big| \leq \frac{C\eps}{|\tau|}{\bar  A}(\tau_0).
\ee

\medskip

Next, we estimate
\begin{align}\label{est_third_term2_F}
&  \left|\left\langle \frac{2(n-k)-v_1v_2}{2v_1v_2}w_\cC -\frac{\psi_0}{4|\tau|}\mathfrak{p}_0w_\cC , \psi_i \right\rangle \right|\\
&\leq 
 \left|\left\langle \frac{2(n-k)-v_1v_2}{2v_1v_2}\chi_{\cC}(v_1) \Big(w_\cC-\sum_{j=0}^2a_j\psi_j\Big),\psi_i\right\rangle \right| \\
&+\sum_{j}|a_j|\left|\left\langle\frac{2(n-k)-v_1v_2}{2v_1v_2} \chi_{\cC}(v_1)\psi_j-\frac{\psi_0}{4|\tau|}\psi_j,\psi_i\right\rangle \right|\\
&+\left|\left\langle \frac{2(n-k)-v_1v_2}{2v_1v_2}(1-\chi_{\cC}(v_1))w_\cC,\psi_i\right\rangle \right|\, . \nonumber
\end{align}

{To proceed, we need the following estimate 
\begin{equation} \label{eq-1678}
\big\|(\sqrt{2(n-k)}-v_i)\chi_{\cC}(v_1)\big\|_{\cD} \leq \frac{C}{|\tau|}.
\end{equation}
Note first that due to $\kappa$-quadraticity at of $v_i$ at time $\tau_0$ and the Gaussian weight, we have
\begin{equation}\label{diff_norm_small}
\big\|\sqrt{2(n-k)}-v_i1_{\{v_1 \ge 0 \}} \big\|_{\mathcal{H}}\leq \frac{C}{|\tau|}.
\end{equation}
An estimate on $\Vert D(v_i \chi_{\cC}(v_1))\Vert_{\mathcal{H}} $ follows by an energy estimate similar to  \cite[Lemma 3.2]{DH_hearing_shape}. Here we present this estimate. As the proof will be identical for $v_i$, $i=1,2$, we simply use $v$ below. The evolution equation of $v$  
\[v_\tau  =\Big (\delta_{ij} -\frac{v_iv_j}{1+|Dv|^2}\Big ) v_{ij} - \frac12 y_i v_i + \frac v{2} -\frac{n-k}{v}  \]
can be rewritten as 
\[(\partial_{\tau} - \Delta + 2^{-1} y\cdot  D )v =-\frac{v_{y_i}v_{y_j}}{1+|Dv|^2} v_{y_iy_j} +\frac{v^2 - 2(n-k)}{2v}.\]
Let us consider $\bar v_\cC:= v \bar \chi(v_1)$ for some fixed smooth cut-off function $\bar\chi$ supported on $[\theta/4,\infty)$ and $\bar \chi =1$ on $[\theta/2,\infty)$.   
\bea (\partial_{\tau} - \Delta + 2^{-1} y\cdot  D) \bar v_\cC  &= \chi (\partial_{\tau} - \Delta + 2^{-1} y\cdot  \nabla )v +v\chi' (\partial_{\tau} - \Delta + 2^{-1} y\cdot  \nabla )v_1 \\&  -v \chi ''|Dv_1|^2 -2 \chi' Dv\cdot Dv_1 \eea 
and let us denote the right side by $g$. i.e. $ (\partial_{\tau} - \Delta + 2^{-1} y\cdot  \nabla ) \bar v_\cC =: g$. For a given time $\tau_1\le \tau_0-1$,
\bea &\partial_\tau \int ( \tau-\tau_1) |D \bar v_\cC |^2 e^{-\frac14 |y|^2} dy  \\ &= \int|D\bar v_\cC |^2 e^{-\frac14 |y|^2} dy + \int (\tau-\tau_1 ) 2D\bar v_\cC  D(\partial_\tau\bar v_\cC )e^{-\frac14 |y|^2} dy \\
& =\int|D \bar v_\cC |^2 e^{-\frac14 |y|^2} \!-\! 2(\tau-\tau_1) ( \Delta - \frac12 y \cdot D)\bar v_\cC ) (g\!+\!( \Delta \!-\!\frac12 y \cdot D)\bar v_\cC )  )   e^{-\frac14 |y|^2}dy \\
&\le \int|D\bar v_\cC|^2 e^{-\frac14 |y|^2} + 2(\tau-\tau_1) \int  \frac 14 g^2   e^{-\frac14 |y|^2} dy,
\eea and also
\bea &\partial_\tau \int  e^{\tau_1 - \tau} |  \bar v_\cC - (2(n-k))^{1/2} |^2 e^{-\frac{1}{4}|y|^2}  dy \\ 
&= e^{\tau_1-\tau} \!\int\! - |  \bar v_\cC - (2(n-k))^{1/2} |^2 e^{-\frac14 |y|^2} \!+\!2 (\bar v_\cC\!-\! (2(n-k))^{1/2} ) \partial_\tau v_c e^{-\frac 14 |y|^2}dy \\
&\le  e^{\tau_1-\tau} \int(- 2 |D \bar v_\cC|^2  +  g^2) e^{-\frac 14 |y|^2} dy .\eea 
On $\tau \in [\tau_1,\tau_1+1]$, 
\bea\label{eq-16778} \partial_\tau \int  \big ( (\tau-\tau_1)|D\bar v_\cC|^2 \!+\! \frac{e }{2}e^{\tau_1-\tau}|\bar v_\cC\!-\! \sqrt {2(n-k)} |^2  \big) e^{-\frac 14 |y|^2} dy \le 2 \int g^2 e^{-\frac 14 |y|^2 } dy  \eea  
To estimate $g$, observe $ |Dv|\le C /|\tau|^{1/2}$ on $v_1\ge \theta/4$ due to convexity and the asymptotics in the intermediate region  (c.f. Lemma \ref{std_pde_der_est}). Combining this with \eqref{diff_norm_small} and  $|D^2v|\le C$ (by Lemma \ref{lemma-cylindrical})
\bea\int \left[ -\frac{v_{y_i}v_{y_j}}{1+|Dv|^2} v_{y_iy_ j} +\frac{v^2 - 2(n-k)}{2v} \right]^2 1_{\{v_1\ge \theta/4\}}e^{-\frac 14 |y|^2} dy \le  {C}/{|\tau|^2}.\eea 
Thus we have $\int g^2 e^{-\frac 14 |y|^2 } dy \le C/|\tau|^2$. 
Integrating \eqref{eq-16778} from $\tau=\tau_1$ to $\tau_1+1$,  we obtain the assertion that for all $\tau\le \tau_0$, 
\begin{equation}
\big\| D (v\bar \chi (v_1 )) \big\|_{\mathcal{H}}\leq \frac{C}{|\tau|}.\, \end{equation}
Remembering also \eqref{eq_wiggle_room} this yields \eqref{eq-1678}.
}

Hence, arguing similarly as in \cite[Proof of Claim 5.26]{CHH_translator} we get
\be
\int_{\tau-1}^\tau \left|\left\langle \frac{2(n-k)-v_1v_2}{2v_1v_2}\chi_{\cC}(v_1) \Big(w_\cC-\sum_{j=0}^2a_j\psi_j\Big),\psi_i\right\rangle \right| \leq \frac{\eps}{|\tau|} {\bar  A}(\tau_0).
\ee
and arguing similarly as in \cite[Proof of Claim 8.3]{ADS2} we get
\be
\int_{\tau-1}^\tau \sum_{j=0}^2 |a_j|\left|\left\langle\frac{2(n-k)-v_1v_2}{2v_1v_2} \chi_{\cC}(v_1)-\frac{\psi_0}{4|\tau|},\psi_i\psi_j\right\rangle \right| \leq \frac{\eps}{|\tau|} {\bar  A}(\tau_0).
\ee
Finally, since $(1-\chi_{\cC}(v_1))w_{\cC}$ is supported in the region where the Gaussian weight is exponentially small, we easily get
\begin{align}
\int_{\tau-1}^\tau \left|\left\langle \frac{2(n-k)-v_1v_2}{2v_1v_2}(1-\chi_{\cC}(v_1))w_\cC,\psi_i\right\rangle \right| 
\leq \frac{\eps}{|\tau|} {\bar  A}(\tau_0).
\end{align}
This concludes the proof of the claim.
\end{proof}

Chopping $[\tau, \tau_0]$ into unit intervals and applying the claim repeatedly and using \eqref{eq_solution_a} as in \cite[Section 8]{ADS2}, we infer
that for all $\tau\leq \tau_0$ we have
\begin{equation}\label{a_small_always}
\bigg( \sum_{i\le j} |a_{ij}(\tau)|^2 \bigg)^{1/2}  \leq \eps \bar {A}(\tau_0).
\end{equation}
Integrating and choosing $\varepsilon>0$ small enough, this implies
$
a_{ij}(\tau_0)=0$ for all $i,j$. 
Together with Theorem \ref{prop-cor-main} (coercivity estimate) we conclude that
\begin{equation}
\|w_\cC\|_{\cD,\infty} +\| W_\cT \|_{2,\infty} =0\, ,
\end{equation}
hence
 $
\mathcal{M}^1=\mathcal{M}^2.
$
This finishes the proof of the spectral uniqueness theorem.
\end{proof}

To prove the Corollary \ref{reflection symmetry}, we need to transform the $k$-ovals to the correct position.
Given any $k$-oval $\mathcal{M}=\{M_t\}$ in $\mathbb{R}^{n+1}$ (with coordinates chosen as usual) and  parameters $\alpha\in \mathbb{R}^{k}, \beta\in \mathbb{R}, \gamma\in \mathbb{R}$ and $R\in \mathrm{SO}(k)$, we set
\begin{equation}
\mathcal{M}^{\alpha, \beta, \gamma, R}:=\{e^{\gamma/2} R(M_{e^{-\gamma}(t-\beta)}-\alpha)\},
\end{equation}
We first have the following proposition on the modes orthogonality.

\begin{proposition}[orthogonality]\label{prop_orthogonality}
For  any $\kappa>0$, there exist   constant $\tau_{\ast}={\tau}_\ast(\kappa)>-\infty$ and $\kappa'=\kappa'(\kappa)>0$ with the following significance. If a $k$-oval $\mathcal{M}$ is $\kappa'$-quadratic at $\tau_0\leq \tau_*$, for every $\eta \in [-\frac{1}{1000}, \frac{1}{1000}]$,  there exist  $\alpha\in \mathbb{R}^{k}, \beta\in \mathbb{R}, \gamma\in \mathbb{R}$ and $R\in \mathrm{SO}(k)$ depending on $\mathcal{M}, \kappa, \tau_0, \eta$  such that the truncated renormalized profile function $v^{\alpha, \beta, \gamma, R}_{\cC}$ of the transformed flow $\mathcal{M}^{\alpha, \beta, \gamma, R}$  is $\kappa$-quadratic at time $\tau_0$ and satisfies
\begin{equation}\label{positive_mode}
\fp_+ \big(v^{\alpha, \beta, \gamma, R}_\cC(\tau_0)-\sqrt{2(n-k)}\big)=0,
\end{equation}
\begin{equation}\label{cross}
\qquad \Big\langle v^{\alpha, \beta, \gamma, R}_{\cC}(\tau_0), y_{i}y_{j}\Big\rangle_\cH=0\quad i\not=j.
\end{equation}
\begin{equation}\label{quadractic_mode}
 \Big\langle v^{\alpha, \beta, \gamma, R}_{\cC}(\tau_0) +\frac{\sqrt{2(n-k)}(|\mathbf{y}|^2-2k)}{4|\tau_0|},|\mathbf{y}|^2-2k\Big\rangle_\cH=\frac{\eta\kappa}{|\tau_0|},
\end{equation}
\end{proposition}
\begin{remark}[partial orthogonality]\label{remark_orthogonality}
The proof below actually shows that (i) translations in space and time (i.e., $\alpha \in \mathbb{R}^k$ and $\beta \in \mathbb{R}$) are responsible for \eqref{positive_mode}, (ii) parabolic dilations ($\gamma\in\mathbb{R}$) are responsible for \eqref{quadractic_mode}, and (iii) rotations in $\mathbb{R}^k$ ($R\in \mathrm{SO}(k)$) are responsible for \eqref{cross}. Consequently, after a small modification of the proof below, we obtain that if we only allows partial transformations while fixing other transformations as $0$ or identity  $\mathrm{Id}_k$, then we can make the renormalized profile function satisfies only one or two equations among \eqref{positive_mode}-\eqref{quadractic_mode}. 
\end{remark}
\begin{proof}
We will solve equations in \eqref{quadractic_mode} and \eqref{positive_mode} using a degree argument similarly as in \cite[Section 7]{ADS2}, and solve the remaining  equations in \eqref{cross} using basic linear algebra. For convenience, we set
\begin{equation}\label{new_vars_scale}
a=e^{\tau/2}\alpha,\quad  b= \sqrt{1+\beta e^{\tau}}-1,\quad \Gamma=\frac{\gamma-\ln (1+\beta e^{\tau})}{\tau}.
\end{equation}
Then
\begin{align}\label{transofrm_scale}
     v^{ab\Gamma R}(y,\tau)=(1+b)v\left(\frac{R^{-1}y-a}{1+b},(1+\Gamma)\tau\right).
\end{align}
Our first goal is, given any $R\in \mathrm{SO}(k)$, to find  a suitable zero of the map
\begin{equation}\label{Psi=0}
\Psi^{R}(a, b, \Gamma)=\left(\begin{array}{c}
   
      \Big\langle y_{1}, v^{ab\Gamma R}_\cC-\sqrt{2(n-k)} \Big\rangle_\cH\\
      \cdots\\
       \Big\langle y_{k}, v^{ab\Gamma R}_\cC-\sqrt{2(n-k)} \Big\rangle_\cH\\
         \Big\langle  1,v^{ab\Gamma R}_\cC-\sqrt{2(n-k)} \Big\rangle_\cH \\
        \Big\langle  |\mathbf{y}|^2-2k, v^{ab\Gamma R}_\cC +\frac{\sqrt{2(n-k)}(|\mathbf{y}|^2-2k)}{4|\tau|}\Big\rangle_\cH-\frac{\eta\kappa}{|\tau|}
            \end{array}
    \right).
\end{equation}
To this end, observe that the vector space spanned by the eigenfunctions
\begin{equation}
 \psi_{i}= y_{i} \quad (1\leq i\leq k),\quad  \psi_{k+1}=1, \quad \psi_{k+2}=|\mathbf{y}|^2-2k
\end{equation}
is $\mathrm{SO}(k)$-invariant. Hence, if $(a, b, \Gamma)$ is a solution of $\Psi^{\mathrm{Id}}(a, b, \Gamma)=0$, then it actually solves $\Psi^R(a, b, \Gamma)=0$ for all $R\in \mathrm{SO}(k)$. To proceed, we need:
\begin{claim}[transformation estimate]\label{tranform_est}
For every $\kappa>0$ there exist  constant $\tau_{\ast}={\tau}_\ast(\kappa)>-\infty$ and  $\kappa'=\kappa'(\kappa)>0$ with the following significance. If a $k$-oval $\mathcal{M}$ is $\kappa'$-quadratic at $\tau_0\leq \tau_*$, for every $\eta \in [-\frac{1}{1000}, \frac{1}{1000}]$ and all $(a,b, \Gamma)\in[-1,1]^{k}\times  [-1/|\tau|,1/|\tau|]\times [-1/2,1/2]$, we have 
\begin{align}
&\max_{i=1, \dots, k}\left| \Big\langle \frac{\psi_i}{ \|\psi_i\|^{2}}, v^{ab\Gamma \mathrm{Id}}_\cC-\sqrt{2(n-k)} \Big\rangle_\cH-\frac{\sqrt{2(n-k)}a_{i}}{2|\tau|(1+\Gamma)}\right|\leq \frac{\kappa}{500|\tau|},\\
&\left| \Big\langle\frac{\psi_{k+1} }{ \|\psi_{k+1}\|^{2}}, v^{ab\Gamma \mathrm{Id}}_\cC\!-\!\sqrt{2(n-k)} \Big\rangle_\cH\!\!-\!\left(\sqrt{2(n-k)}b\!-\!\frac{\sqrt{2(n-k)}|a|^2}{4|\tau|(1+\Gamma)}\right)\right| \leq \frac{\kappa}{500|\tau|},\\
&\left| \Big\langle \frac{\psi_{k+2}}{ \|\psi_{k+2}\|^{2}}, v^{ab\Gamma \mathrm{Id}}_\cC +\frac{\sqrt{2(n-k)}(|\mathbf{y}|^2-2k)}{4|\tau|}\Big\rangle_\cH-\frac{\eta\kappa}{|\tau|}-\frac{\sqrt{2(n-k)}\Gamma}{4|\tau|(1+\Gamma)}\right|\leq \frac{\kappa}{500|\tau|}.
\end{align}
\end{claim}
\begin{proof} By Theorem \ref{point_strong} (strong $\kappa$-quadraticity),   for every $\kappa''>0$ there exist $\kappa'>0$ small enough and $\tau_{*}>-\infty$ negative enough, such that if a $k$-oval 
 $\mathcal{M}$ in $\mathbb{R}^{n+1}$ is $ \kappa'$-quadratic at some time $\tau_{0}\leq \tau_{*}$, then it is strongly $\kappa''$-quadratic from time $\tau_{0}$. This together with standard Gaussian tail estimates for replacing cut-off functions as in \cite[proof of Lemma 4.2]{ADS2} implies
\begin{equation}
    \left\| v(\mathbf{y},  \tau)-\sqrt{2(n-k)}+\frac{\sqrt{2(n-k)}(|\mathbf{y}|^2-2k)}{4|\tau|}  \right\|_{\mathcal{H}}\leq \frac{\kappa''}{|\tau|} .
\end{equation}
holds for all $\tau\leq\tau_0\leq \tau_\ast$.

Since $(a,b, \Gamma)\in[-1,1]^{k}\times  [-1/|\tau|,1/|\tau|]\times [-1/2,1/2]$, using \eqref{transofrm_scale} this implies
\begin{multline}
 v^{ab\Gamma \mathrm{Id}}=\sqrt{2(n-k)} + \sqrt{2(n-k)}b-\frac{\sqrt{2(n-k)}|a|^2}{4|\tau|(1+\Gamma)} 
-\frac{\sqrt{2(n-k)}}{4|\tau|(1+\Gamma)}(|\mathbf{y}|^2-2k)\\
 +\sum_{i=1}^{k}\frac{\sqrt{2(n-k)}}{4}\frac{2a_{i}y_{i}}{|\tau|(1+\Gamma)}+O\left(\frac{\kappa''}{|\tau|}\right)
\end{multline}
in $\mathcal{H}$-norm. For every $\kappa>0$, fixing above $\kappa''\ll \frac{\kappa}{1000}$ small enough, which then determines above $\kappa'>0$ and $\tau_*>-\infty$, and noting $\eta \in [-\frac{1}{1000}, \frac{1}{1000}]$, together with standard Gaussian tail estimates to include the effect of cut-off, the claim follows.
\end{proof}

Now, by Claim \ref{tranform_est} (transformation estimate), for $\kappa>0$ small enough, for every $\tau_0\leq \tau_\ast$ and $|\eta|\leq \frac{1}{1000}$, the map $\Psi^{\mathrm{Id}}$ restricted to the boundary of
\begin{equation}\label{Dk}
D_{\kappa}:=\left\{(a, b, \Gamma)\in\mathbb{R}^{k+2}\;:\; |a|^2+|\tau_0|^2b^2+\Gamma^2\leq 100 \kappa^2  \right\}
\end{equation}
is homotopic to the injective map 
\begin{equation}
(a,b,\Gamma)\mapsto \sqrt{2(n-k)}\left(
     \,
     \frac{a\|\psi_{1}\|^2}{2|\tau_0|(1+\Gamma)},  b{||{\psi_{k+1}}||^2}-\frac{|a|^2\|\psi_{k+1}\|^2}{4|\tau_0|(1+\Gamma)},
     \frac{\Gamma\|\psi_{k+2}\|^2}{4|\tau_0|(1+\Gamma)}
    \right),
\end{equation}
through maps from $\partial D_\kappa$ to $\mathbb{R}^{k+2}\setminus\{0\}$. The map $\Psi^{\mathrm{Id}}$ from the full ball to $\mathbb{R}^{k+2}$ has therefore degree one. Hence, there exists $(a,b,\Gamma)\in  D_\kappa$ solving $\Psi^{\mathrm{Id}}(a,b,\Gamma)=0$. Remembering the above discussion, this actually solves
\begin{equation}
\Psi^R(a, b, \Gamma)=0\,\,\, \textrm{ for all  } R\in \mathrm{SO}(k).
\end{equation}
Finally, we want to find $R\in \mathrm{SO}(k)$ such that the remaining equations
\begin{equation}\label{R_cross_vanishing}
0=\big\langle y_{i}y_{j}, v^{ab\Gamma R}_\cC(\tau_0) \big\rangle_\cH =\big\langle (R^{T}y)_{i}(R^{T}y)_{j}, v^{ab\Gamma \mathrm{Id}}_\cC(\tau_0) \big\rangle_\cH\quad i\not=j
\end{equation}
hold. If we let
matrix $S=(\big\langle y_{s}y_{\ell}, v^{ab\Gamma \mathrm{Id}}_\cC(\tau_0) \big\rangle_\cH)_{k\times k}$,  \eqref{R_cross_vanishing} is equivalent to finding $R\in \mathrm{SO}(k)$ to diagonalize the matrix $S$. By the symmetric property of $S$, we can always find a $R\in \mathrm{SO}(k)$  solving the remaining equations \eqref{R_cross_vanishing}. In addition, noting  that by the above estimates the transformed flow $\mathcal{M}^{\alpha, \beta, \gamma, R}$ is $\kappa$-quadratic at time $\tau_0$, this finishes the proof of the proposition.
\end{proof}

Then we prove the $\mathbb{Z}_2^k\times\mathrm{O}(n+1-k)$ symmetry of $k$-ovals.

\begin{proof}[Proof of Corollary \ref{reflection symmetry}]There is no harm in applying parabolic rescaling and time-shift to $k$-ovals, since the desired symmetry remains invariant under parabolic rescaling and and time-shift. Then, for any $k$-oval $\mathcal{M}$, by Lemma \ref{lem-quadraticity} we have that given any $\kappa'>0$ there is $\tau_0>-\infty$ negative enough such that $\mathcal{M}$ is $\kappa'$-quadratic at $\tau_0$. Choose small $\kappa>0 $ from Theorem \ref{thm:uniqueness_eccentricity_intro} (spectral uniqueness) and $\kappa'=\kappa'(\kappa)>0$ from Proposition \ref{prop_orthogonality} (orthogonality). By Proposition \ref{prop_orthogonality}  with $\eta=0$, after suitable space-time rigid motion and parabolic rescaling, $\mathcal{M}$ is $\kappa$-quadratic at $\tau_0$  and satisfies orthogonality conditions \eqref{positive_mode}, \eqref{cross}, \eqref{quadractic_mode} with $\eta=0$. 

For $j=1,\dots, k$, let $\mathcal{M}^{-, j}$ be the flow obtained by reflecting $\mathcal{M}$ about the hyperplane through the origin that is orthogonal to $j$-th axis of $\mathbb{R}^k$. Then we have $\mathcal{M}^{-, j}$  preserves $\kappa$-quadraticity at $\tau_0$  and orthogonality conditions \eqref{positive_mode}, \eqref{cross}, \eqref{quadractic_mode}. Moreover, the above reflection  preserves  the 
following quadratic mode spectral values of $\mathcal{M}$ at $\tau_0$
\begin{equation}
{\Big\langle v_{\cC}(\cdot, \tau_{0}), y^2_{j}-2 \Big\rangle_\cH}\quad j=1,\dots, k.
\end{equation}
Therefore, $\mathcal{M}$ and $\mathcal{M}^{-, j}$ are both $\kappa$-quadratic at $\tau_0$  and have the same neutral mode and positive mode projections. By Theorem \ref{thm:uniqueness_eccentricity_intro} (spectral uniqueness), we have 
\begin{equation}
    \mathcal{M}=\mathcal{M}^{-, j} \quad j=1,\dots, k
\end{equation}
which verifies the $\mathbb{Z}_2^k$-symmetry. Together with \cite[Theorem 1.8]{DZ_spectral_quantization}, we know any $k$-oval is $\mathbb{Z}_2^k\times\mathrm{O}(n+1-k)$-symmetric and we completes the proof of Corollary    \ref{reflection symmetry}.
\end{proof}

\bigskip

\section{Spectral stability theorem and Lipschitz continuity}\label{stability section}
In this section,  we prove Theorem \ref{spectral_stability_intro} (the spectral stability), which is a quantitative version of spectral uniqueness theorem for $k$-ovals, via showing a Lipschitz continuity result. We first recall the following definition. Given $\kappa>0$ and $\tau_0>-\infty$, $\bar{\mathcal{A}}_{\kappa}(\tau_0)$ is the set of all $k$-ovals which are $\kappa$-quadratic at time $\tau_0$, and the corresponding truncated profile function  $v_{\cC}$  satisfies orthogonality conditions (this is always possible for $k$-ovals due to Proposition \ref{prop_orthogonality} (orthogonality)).
\begin{equation}
    \mathfrak{p}_{+}\big(v_{\cC}({\bf{y}}, \tau_{0})-\sqrt{2(n-k)}\big)=0,
\end{equation}
\begin{equation}
    \Big\langle v^{\mathcal{M}}_{\cC}({\bf{y}}, \tau_0), y_{i}y_{j}\Big\rangle_\cH=0,\quad 1\leq i<j\leq k
\end{equation}
and
\begin{equation}
     \Big\langle v^{\mathcal{M}}_{\cC}({\bf{y}}, \tau_0) +\frac{\sqrt{2(n-k)}(|\mathbf{y}|^2-2k)}{4|\tau_0|}, |\mathbf{y}|^2-2k\Big\rangle_\cH=0.
\end{equation}

Then we define the spectral ratio map $\mathcal{E}=\mathcal{E}(\tau_{0}): \bar{\mathcal{A}}_{\kappa}(\tau_0)\rightarrow \Delta^{k-1}$ at $\tau_0$:
\begin{equation}
   \mathcal{E}(\mathcal{M})=\mathcal{E}(\tau_{0})(\mathcal{M})=\left(\frac{\langle v^{\mathcal{M}}_{\cC}(\tau_{0}), y^2_{1}-2 \rangle_\cH}{\langle v^{\mathcal{M}}_{\cC}(\tau_{0}), |{\bf{y}}|^2-2k \rangle_\cH}, \dots,\frac{\langle v^{\mathcal{M}}_{\cC}(\tau_{0}),  y^2_{k}-2 \rangle_\cH}{\langle v^{\mathcal{M}}_{\cC}(\tau_{0}), |{\bf{y}}|^2-2k \rangle_\cH}\,\right).
\end{equation}
We also recall $
w_\cC=v_1\chi_\cC(v_1)-v_2\chi_\cC(v_2)$ and
$W_{\mathcal{T}}:=\chi_{\mathcal{T}} W=\chi_{\mathcal{T}} (Y_1-Y_2)$, where for $i=1, 2$, $Y_i(\cdot,\omega,\tau)$ is defined as inverse function of $v_i(\cdot,\omega,\tau)$, as in last subsection. For convenience, we state the detailed version of spectral stability theorem below.
\begin{theorem}[spectral stability]\label{spectral_stability}
    There exist constant $\kappa>0$ small enough, $\tau_{*}>-\infty$ negative enough, $C<\infty$  with the following significance: If $\mathcal{M}^1, \mathcal{M}^2\in \bar{\mathcal{A}}_{\kappa}(\tau_0)$, namely they are $k$-ovals which are $\kappa$-quadratic at $\tau_0\leq \tau_{*}$ and whose profile function $v$ satisfy orthogonality conditions \eqref{condition_positive}, \eqref{OC11} and \eqref{OC21}, then 
    \begin{align}\label{part1-1}
        \|w_\cC\|_{\hD,\infty} +\| W_\cT \|_{2,\infty} \leq C\|\mathfrak{p}_0 w_{\mathcal{C}}(\tau_0)\|_{\mathcal{H}}.
    \end{align}
  Moreover,  for each fixed $m\in \mathbb{N}$, fixed $0<l<+\infty$ and $\tau_0\le \tau_*$, there are constants $c(\tau_0, l, m)>0$ and $\delta(\tau_0, l)>0$, such that if the $\tau_0$-time spectral closeness  condition  
    \begin{equation}
        |\mathcal{E}(\tau_0)(\mathcal{M}^1)-\mathcal{E}(\tau_0)(\mathcal{M}^2)|\leq \delta(\tau_0, l)
    \end{equation}
    holds,  then for  the renormalized flows  $\bar{M}^i_{\tau}$ of $\mathcal{M}^i$, $i=1, 2$, one can write $\bar{M}^1_{\tau}$ as a graph over $\bar{M}^2_{\tau}$ 
    of a function $\zeta(\tau)$ for $\tau\in [\tau_0-l, \tau_0]$ such that
   \begin{align}
        &\quad c(\tau_0, l, m)\sup_{\tau\in [\tau_0-l, \tau_0]}\|\zeta(\tau)\|_{C^{m}}\\
        &\leq |\mathcal{E}(\tau_0)(\mathcal{M}^1)-\mathcal{E}(\tau_0)(\mathcal{M}^2)|\\
        &\leq \!C\|\zeta(\tau_0)\|_{C^{m}}.
    \end{align}
\end{theorem}

To prove Theorem \ref{spectral_stability_intro} (the spectral stability), we need the following  proposition about local Lipschitz continuity estimates of $\mathcal{E}^{-1}$.
\begin{proposition}[Lipschitz continuity of $\mathcal{E}^{-1}$]\label{Lipshcitz continuity}
    There exist  $\kappa>0$ small enough, $\tau_{*}>-\infty$ negative enough and a constant $C$ with the following significance: if two $k$-ovals $\mathcal{M}^1, \mathcal{M}^2 \in \bar{\mathcal{A}}_{\kappa}(\tau_0)$, where $\tau_0\leq \tau_{*}$, then 
    \begin{align}\label{part1}
        \|w_\cC\|_{\hD,\infty} +\| W_\cT \|_{2,\infty} \leq C\|\mathfrak{p}_0 w_{\mathcal{C}}(\tau_0)\|_{\mathcal{H}}.
    \end{align}
  Moreover, for each fixed $m\in \mathbb{N}$, fixed $0<l<+\infty$ and $\tau_0\le \tau_*$ , there are $C(\tau_0, l, m)<\infty$ and $\delta(\tau_0, l)>0$ such that if  the condition $|\mathcal{E}(\tau_0)(\mathcal{M}^1)-\mathcal{E}(\tau_0)(\mathcal{M}^2)|\leq \delta(\tau_0, l)$ holds, then for $\tau\in [\tau_0-l, \tau_0]$ one can write renormalized flow $\bar{M}^1_{\tau}$ as a graph over $\bar{M}^2_{\tau}$ 
    of a function $\zeta(\tau)$ satisfying
    \begin{align}\label{Lip spectral graph}{
        \sup_{\tau\in [\tau_0-l, \tau_0]}\|\zeta(\tau)\|_{C^{m}}\leq C(\tau_0, l, m)|\mathcal{E}(\tau_0)(\mathcal{M}^1)-\mathcal{E}(\tau_0)(\mathcal{M}^2)|, \quad m\geq 1.
    }\end{align}
\end{proposition}
\begin{proof}
For proving the desired inequality, by Theorem \ref{prop-cor-main} using $\bar {A}$ in \eqref{eq-tildeA},
\begin{align}
    \|w_\cC\|_{\hD,\infty} +\| W_\cT \|_{2,\infty} \leq C\|\mathfrak{p}_0 w_\cC \|_{\hD,\infty}.
\end{align}
Since $\mathfrak{p}_0 w_\cC $ lies in a finite dimensional neutral pace $\mathcal{H}_0$ (see \eqref{H0+} for definition), the $\hD$-norm and $\mathcal{H}$-norm are equivalent and  we have 
\begin{align}
    \|\mathfrak{p}_0 w_\cC \|_{\hD,\infty} \leq C\|\mathfrak{p}_0 w_\cC \|_{\mathcal{H},\infty} =  C\bar {A}(\tau_0),
\end{align}
where $C>0$ is a constant depending on $k$. Thus
\begin{align}
    \|w_\cC\|_{\hD,\infty} +\| W_\cT \|_{2,\infty} \leq C\bar {A}(\tau_0).
\end{align}
Since 
\begin{align}
    \|\mathfrak{p}_0 w_{\mathcal{C}}(\tau_0)\|_{\mathcal{H}} = \big (\sum_{i\le j} a_{ij}^2 (\tau_0)\big )^{\frac12},
\end{align}
for the first part of the theorem, it suffices to prove 
\begin{align}
    \bar {A}(\tau_0)\leq C\big (\sum_{i\le j} a_{ij}^2 (\tau_0)\big )^{\frac12}  .
\end{align}
Recall $a_{ij}$ satisfies the ODE in \eqref{aij_ODE_F}:
\be
  \frac{d}{d\tau}a_{ij}(\tau) = \frac{2a_{ij}(\tau)}{|\tau|} + F_{ij}(\tau),
\ee
where $F_{ij}$ is from \eqref{eq-F-tau}. Applying the integrating factor $\tau^2$ to the equation we get
\be
  \frac{d}{d\tau}(\tau^2a_{ij}(\tau)) = \tau^2 F_{ij}(\tau),
\ee
Integrating the ODE from $\tau$ to $\tau_0$ gives: 
\begin{align}
    a_{ij}(\tau) = \frac{\tau_{0}^2}{\tau^2}a_{ij}(\tau_0) -\frac{1}{\tau^2}\int_{\tau}^{\tau_0} F_{ij}(\sigma)\sigma^2\, d\sigma.
\end{align}
Chopping $[\tau, \tau_0]$ into unit intervals and applying this repeatedly and using the estimate in Claim \ref{claim411}, for $\varepsilon>0$ small enough, there are $\kappa>0$ small enough and $\tau_*$ negative enough such that, for $\tau\leq \tau_0\leq \tau_*$,
\begin{equation}
\left|\int_{\tau}^{\tau_0} F_{ij}(\sigma)\sigma^2\, d\sigma \right|\leq \varepsilon|\tau|^2\bar {A}(\tau_0).
\end{equation}
We find that for $\tau\leq \tau_0\leq \tau_*$ 
\begin{align}
    |a_{ij}(\tau)| \leq  |a_{ij}(\tau_0)| + \varepsilon \bar {A}(\tau_0).
\end{align}
Integrating the square of both sides of above inequality from $\tau-1$ to $\tau$ and taking the supremum over $\tau \leq \tau_0$, we infer that
\begin{align}
    \bar {A}(\tau_0)\leq C \left[\big (\sum_{i\le j} a_{ij}^2 (\tau_0)\big )^{\frac12} + \varepsilon\bar {A}(\tau_0)\right].
\end{align}
where $C>0$ is a constant depending on $k$.  Choosing $\varepsilon=\frac{1}{2}$, we conclude that
\begin{align}
    \bar {A}(\tau_0)\leq C\big (\sum_{i\le j} a_{ij}^2 (\tau_0)\big )^{\frac12},
\end{align}
and finish the proof of \eqref{part1} in the first part of the proposition.

Now we prove \eqref{Lip spectral graph} in the second part of the proposition.  Note that Proposition \ref{lemma-mainY} and Lemma \ref{lemma-coefficients} imply that the first derivative of $Y$ at time $\tau$ is bounded by a constant $C(\tau_0,l)$ depending on $\tau_0$ and $l$, when $\tau \in [\tau_0-l, \tau_0]$. 
Therefore the coefficients of the evolution equation of $W$ (see Appendix \ref{lemma-ev-W-appendix}) are bounded by $C(\tau_0, l)$. By parabolic interior estimate in \cite[Theorem 6.17]{Lieberman}, we have
\begin{align}\label{W_c20}
	&\quad\sup_{\tau\in [\tau_0-l, \tau_0]}||W_{\cT}(\cdot,\tau)||_{C^{m}(B_{17\theta/18})}\\
    &\leq C(\tau_0, l, m, \theta) \Big ( \sup_{\tau\leq \tau_0}\int_{\tau-1}^{\tau}||W_{\cT}(\cdot, s)||_{L^2(B_{\theta })}^2 ds \Big)^{1/2}\\
    &\leq C(\tau_0, l, m, \theta)\| W_\cT \|_{2,\infty}
\end{align}
We shall estimate $w_{\mathcal{C}}$ similarly. To this end, we define the following balls:
\begin{align*}
	B^1 = B(0,\sqrt{|\tau|(2(n-k)-(19\theta/20)^2)}) \subset\mathbb{R}^k
\end{align*}
and
\begin{align*}
    \ B^2 = B(0,\sqrt{|\tau|(2(n-k)-(9\theta/10)^2)}) \subset\mathbb{R}^k.
\end{align*}
By the uniform sharp asymptotics in Proposition \ref{strong_sharp_asymptotics}, there exists $\cT(\theta) < -1$ negative enough such that when $\tau \leq \cT(\theta)$ 
\begin{align*}
	\{v\geq \theta\} \subset B^1 \subset  B^2 \subset \{v\geq 7\theta/8\}.  
\end{align*} 

By definition,  $\chi_{\cC} \equiv 1$ on $B^2$. In particular, all derivatives of $\chi_{\cC}$ vanish on $B^2$. Then we can apply standard parabolic estimates (c.f \cite[Theorem 6.17]{Lieberman}) to the evolution equation of  $w_{\cC}$  in Proposition \ref{evolution-wC} and get
\begin{align}\label{wc20}
	&\quad\sup_{\tau\in [\tau_0-l, \tau_0]}||w_{\cC}(\cdot,\tau)||_{C^{m}(B^1)}\\ &\leq C(\tau_0, l, m, \theta) \sup_{\tau\leq \tau_0}\left(\int_{\tau-1}^{\tau}||w_{\cC}(\cdot, s)||^2_{L^2(B^2)}ds\right)^{\frac{1}{2}} \\
	&\leq C(\tau_0, l, m, \theta)  \|w_\cC\|_{\hD,\infty}.
    \end{align}

Combining \eqref{W_c20}, \eqref{wc20} with \eqref{part1}, we conclude that for $\tau \in [\tau_0-l, \tau_0]$
\begin{align}\label{C20-L2}
	&\quad||w_{\cC}(\cdot,\tau)||_{C^{m}(B^1)}  +  ||W_{\cT}(\cdot,\tau)||_{C^{m}(B_{17\theta/18})}\\
    &\leq
	C(\tau_0, l, m, \theta) \|\mathfrak{p}_0 w_{\mathcal{C}}(\tau_0)\|_{\mathcal{H}}.
\end{align}
Then for the second part of proof of Proposition \ref{Lipshcitz continuity} 
 ({Lipshcitz continuity} of $\mathcal{E}^{-1}$), we
combine the above inequalities and \eqref{graph-w-W} in Lemma \ref{graphical_estimates claim} with fixed small enough $\theta>0$ under assuming $|\mathcal{E}(\tau_0)(\mathcal{M}^1)-\mathcal{E}(\tau_0)(\mathcal{M}^2)|\leq \delta(\tau_0, l, \theta)$ to obtain
\begin{align}
    ||\zeta(\cdot,\tau)||_{C^{m}} &\leq C(\tau_0, l, m, \theta)(||w_{\cC}(\cdot,\tau)||_{C^{m}(B^1)}  +  ||W_{\cT}(\cdot,\tau)||_{C^{m}(B_{17\theta/18})}) \\
    &\leq C(\tau_0, l, m, \theta)\|\mathfrak{p}_0 w_{\mathcal{C}}(\tau_0)\|_{\mathcal{H}}\\
    &\leq C(\tau_0, l, m, \theta)|\mathcal{E}(\tau_0)(\mathcal{M}^1)-\mathcal{E}(\tau_0)(\mathcal{M}^2)|,
\end{align}
where $\theta>0$ is fixed. This completes the proof of the proposition. 
\end{proof}

\bigskip

Then, we prove Theorem \ref{spectral_stability_intro} (the spectral stability). 
\begin{proof}[Proof of Theorem \ref{spectral_stability_intro}]
First, by definition of the two truncated profile functions $ v^{i}_{\mathcal{C}}$ of  the two profile functions $v^{i}$ for $i=1, 2$, we have
\begin{align}
    &\quad v^{1}_{\mathcal{C}}-v^{2}_{\mathcal{C}}\\
    &=(v^{1}-v^{2})\chi(v^{1})+v^{2}(\chi(v^{1})-\chi(v^{2}))\\
    &=(v^{1}-v^{2})\chi(v^{1})+v^{2}\chi'(\tilde{v})(v^{1}-v^{2}),
\end{align}
where we used intermediate value theorem in the last line. Together with the  definition of $\mathcal{E}(\tau_0)(\mathcal{M}^{i})$ at $\tau_0$ time for $i=1, 2$, normalization condition \eqref{OC21} and  we obtain
\begin{equation}
    |\mathcal{E}(\tau_0)(\mathcal{M}^1)\!-\!\mathcal{E}(\tau_0)(\mathcal{M}^2)|\!\leq \!C\|\zeta(\tau_0)\|_{C^{m}}.
\end{equation}

Then,  let $c(\tau_0, l, m)=C(\tau_0, l, m)^{-1}$ where $C(\tau_0, l, m)^{-1}$ is from Proposition \ref{Lipshcitz continuity}. Together with the above, the estimates \eqref{bilip stability} in Theorem \ref{spectral_stability_intro} (the spectral stability) directly follow from
  Proposition \ref{Lipshcitz continuity} (Lipschitz continuity of $\mathcal{E}^{-1}$), the definition of $\mathcal{E}(\tau_0)(\mathcal{M}^{i})$ at $\tau_0$ time for $i=1, 2$ and 
  that  by \eqref{OC21} for $\tau_0$ negative enough
 \begin{align}
       |\mathfrak{p}_0 w_{\mathcal{C}}(\tau_0)\|_{\mathcal{H}} &\leq \frac{\sqrt{2(n-k)}}{4|\tau_0|}\||{\bf y}|^2-2k\|_{\mathcal{H}}|\mathcal{E}(\tau_0)(\mathcal{M}^1)-\mathcal{E}(\tau_0)(\mathcal{M}^2)|\\
       &\leq |\mathcal{E}(\tau_0)(\mathcal{M}^1)-\mathcal{E}(\tau_0)(\mathcal{M}^2)|.
    \end{align}

\end{proof}

\bigskip

\begin{appendix}
    \section{Derivation of some  evolution equations}\label{evolution equation appendix}

We first recall some definitions from Section \ref{Quadratic_almost_concavity_section}. For constant $\delta>0$, the $(0,2)$-tensor $A$ is defined by
\be  A_{ij} =\nabla_{ij}^2Q-\varepsilon g_{ij} -2|\nabla V|^2\nabla_{i}V\nabla_{j}V+|\nabla V|^2\varepsilon g_{ij},\ee 
where $\eps = \delta+ \gamma $ is a scalar function with 
\be\gamma:= \left(\frac{-t}{\log(-t)}\right)^{3/2} V^{-3}.\ee

\begin{proposition}[evolution of Hessian]\label{prop_evol_hess} The Hessian of $Q=V^2$, viewed as an intrinsic function, evolves by
\begin{align}\label{evolution_Qij}
&\quad (\partial_t - \Delta) \nabla^2_{ij}Q\\
&=-Q^{-1}\nabla_k  Q\,  \nabla_k \nabla^2_{ij}Q - Q^{-1}\nabla^2_{ik}Q  \nabla^2_{jk} Q\\
 &+
\tfrac{1}{2} Q^{-2} |\nabla  Q|^2 \nabla^2_{ij} Q    - Q^{-3}|\nabla  Q|^2\nabla_i  Q \nabla_j Q \nonumber\\
&+Q^{-2}(\nabla_i Q\nabla_k Q \nabla^2_{jk} Q  + \nabla_j Q\nabla_k Q  \nabla^2_{ik} Q)\\
 & - (H h_{ik}-h_{ip} h_{pk}  ) \nabla^2_{jk} Q- (H h_{jk}-h_{jp} h_{pk}  )\nabla^2_{ik} Q\nonumber\\
  & - (h_{ij}h_{kp}-h_{ip}h_{jk})Q^{-1}\nabla_k Q \nabla_p Q\nonumber\\
  &+2(h_{ij} h_{pq} - h_{iq}h_{jp}) \nabla^2_{pq} Q +2 h_{kp}\nabla_k h_{ij}\nabla_p Q. 
\end{align}
\end{proposition}

 \begin{proof}
Applying Proposition \ref{lemma-ev-eq} (evolution of profile function) yields
 \begin{equation}
 \label{eq-f}
 (\partial_t-\Delta) Q = -\tfrac{1}{2}Q^{-1}|\nabla Q|^2 - 2(n-k).
 \end{equation}
Differentiating we get
\begin{equation}
\nabla_{j} (\partial_t-\Delta) Q =- Q^{-1}\nabla_k Q\nabla^2_{jk}Q+\tfrac{1}{2}Q^{-2}|\nabla Q|^2 \nabla_j Q ,
\end{equation}
and differentiating again we obtain
\begin{align} \label{comp_begin_with}
&\quad \nabla^2_{ij} (\partial_t-\Delta) Q \\&= - Q^{-1}\nabla_k Q {\nabla_i\nabla^2_{jk} Q}- Q^{-1}\nabla^2_{ik}Q  \nabla^2_{jk} Q \nonumber\\
&+ \tfrac{1}{2}Q^{-2}|\nabla  Q|^2 \nabla^2_{ij} Q - Q^{-3}|\nabla  Q|^2\nabla_i  Q \nabla_j Q\nonumber\\
&+Q^{-2}(\nabla_i Q\nabla_k Q \nabla^2_{jk} Q  + \nabla_j Q\nabla_k Q  \nabla^2_{ik} Q)
\end{align}
Hence, our main task is to compute the commutator of the heat operator and the Hessian.

In general, the time derivative of the Hessian of a function $f$ equals
\begin{equation}
\partial_t \big(\nabla^2_{ij} f\big) 
 = \nabla^2_{ij}  \big(\partial_t f\big) - \big(\partial_t \Gamma_{ij}^k\big) \,\partial_k f\, ,
\end{equation}
where the variation of the Christoffel symbols is given by the formula
\begin{equation}
\partial_t \Gamma_{ij}^k= \frac{1}{2}g^{k\ell}\big(\nabla_i (\partial_t g_{j\ell}) +\nabla_j (\partial_t g_{i\ell}) -\nabla_\ell (\partial_t g_{ij})  \big).
\end{equation}
Since under mean curvature flow we have $\partial_t g_{ij} = -2H h_{ij}$ together with the Codazzi equation $\nabla_i h_{jk}=\nabla_j h_{ik}$ this yields
 \begin{align}
 \label{eq-comm-tder}
&\quad\, \partial_t \big(\nabla^2_{ij} Q\big) -\nabla^2_{ij}\big(\partial_t Q\big)\\
& =    \big(H \nabla_i h_{jk}  - h_{ij} \nabla_k H +  h_{ik} \nabla_j H  +  h_{jk} \nabla_i H \big)\, \nabla_k Q.
 \end{align}
On the other hand, thanks to the second Bianchi identity we have the general commutator formula
 \begin{align}
 \label{comm-lap-hess}
 &\quad \nabla^2_{ij} \Delta f - \Delta  \nabla^2_{ij} f\\
 &=2R_{i p j q}\nabla^2_{pq} f-R_{ik} \nabla^2_{jk} f-R_{jk} \nabla^2_{ik} f\\
 &-(\nabla_i R_{jk}+\nabla_j R_{ik}-\nabla_k R_{ij})\nabla_k f,
 \end{align}
where $R_{ijk\ell}$ is the Riemann tensor and $R_{ik}=R_{i p k p}$ is the Ricci tensor. In our setting of hypersurfaces, we in addition have the Gauss equation
\begin{equation}
R_{ijk\ell} = h_{ik} h_{j\ell} - h_{i\ell}h_{jk},
\end{equation}
and its trace
\begin{equation}
R_{ik} = H h_{ik}-h_{ij}h_{jk}.
\end{equation}
In particular, together with the Codazzi equation this yields
\begin{align}
&\quad\nabla_i R_{jk}+\nabla_j R_{ik}-\nabla_k R_{ij}\\
&=\big(H \nabla_i h_{jk}  - h_{ij} \nabla_k H +  h_{ik} \nabla_j H  +  h_{jk} \nabla_i H \big)-2h_{kp}\nabla_p h_{ij} .
\end{align}
Combining the above formulas we infer that
\begin{align} \label{eq-comm-heat-op}
&\quad (\partial_t -\Delta) \nabla^2_{ij} Q - \nabla^2_{ij} (\partial_t -\Delta)Q\\
&= 2 h_{kp}\nabla_p h_{ij}\nabla_k Q+2(h_{ij} h_{pq} - h_{iq}h_{jp})\, \nabla^2_{pq} Q  \\
&- (  H h_{ik}-h_{ip} h_{pk} )\, \nabla^2_{jk} Q - (H h_{jk} -h_{jp} h_{pk})\, \nabla^2_{ik} Q .
\end{align}
Together with \eqref{comp_begin_with}, where we rewrite the third derivative term using
\begin{equation}
{\nabla_i\nabla^2_{jk} Q}=\nabla_{k}\nabla^2_{ij} Q+(h_{ij}h_{kp}-h_{ip}h_{jk})\nabla_p Q,
\end{equation}
this yields the assertion.
 \end{proof}

We rewrite the Codazzi term from the last line of \eqref{evolution_Qij} in terms of $V$:

 \begin{lemma}[Codazzi term] \label{lemma-nabla-h}
 We have
  \bea 2  h_{kp}\nabla_k h_{ij}\nabla_p Q = -2\gra  \nabla_p V h_{pk} \nabla_k \nabla^2_{ij}Q + \Psi_{ij},   \eea
  where
   \begin{align}   \label{eq-psi2}
  \Psi_{ij}
  &= 4\gra  \nabla_pV h_{pk}( \nabla_k V \nabla^2_{ij}V+\nabla_j V \nabla^2_{ik}V +\nabla_i V \nabla^2_{jk}V ) \\
   &- 4 \nabla_p V h_{pk}(\gra ^3 V \nabla^2_{kq}V \nabla_q V \nabla^2_{ij}V-V G_{kij} ),
 \end{align}
and $G_{kij}$ is a 3-tensor that satisfies $G_{kij}=G_{kji}$ and  $G_{kij}X^iX^j=0$ for all vector fields such that $X\perp \partial_\vartheta$.
 \end{lemma}
 
 \begin{proof} Using \eqref{eq-equivalent} (second fundamental form) we see that
  \bea 
 \label{eq-nabla-h-U}
 \nabla_k h_{ij}  = -\gra  ^3 \nabla_q V \nabla^2_{kq}V\nabla^2_{ij}V - \gra  \nabla_k\nabla^2_{ij}V + G_{kij} \, , \eea 
 where
  \begin{equation}
 \label{eq-Gijk}
G_{kij}=\nabla_k(\eta V)\nabla_i \vartheta\nabla_j\vartheta+\eta V \nabla_{i}\vartheta \nabla^2_{jk}\vartheta  +\eta V \nabla^2_{ik}\vartheta \nabla_{j}\vartheta\, .
 \end{equation}
Note $G_{kij}=G_{kji}$ and $G_{kij}X^iX^j=0$ whenever $X\perp \partial_\vartheta$. Moreover, substituting $Q=V^2$ we see that
\begin{equation}
\nabla _k \nabla^2_{ij}Q = 2 ( V\nabla_k \nabla^2_{ij}V+\nabla_k V \nabla^2_{ij}V   +\nabla_i V \nabla^2_{kj}V +\nabla_j V \nabla^2_{ki}V ).
\end{equation}
Combining the above facts yields the assertion.
 \end{proof}

   \begin{proposition}[evolution of $A_{ij}$]\label{evolution of new A} The tensor $A$ evolves by   \bea \label{eq-Aij} (\partial_t - \Delta + \nabla_Z)& {A}_{ij}=  N_{ij}\, .\\
\eea 
Here  $Z$ is the vector field defined by $\langle Z,Y\rangle=Q^{-1}\nabla_Y Q+2\eta h(\nabla V, Y)$ for all $Y$, and
$N_{ij}$ is $(0,2)$-tensor that is       \begin{align}\label{N=N1+N2+N3 2}
    N_{ij} = (N_1)_{ij} + (N_2)_{ij} + (N_3)_{ij}\, , 
\end{align}
where 
\bea  \label{eq-N1}
    {(N_1)}_{ij} &= -\frac{n-k}{2}Q^{-2}|\nabla Q|^2\left(Q^{-1}\nabla_i Q \nabla_j Q-\varepsilon g_{ij}\right) \\
    &-Q^{-3}|\nabla  Q|^2\nabla_i  Q \nabla_j Q - Q^{-1}\nabla^2_{ik}Q  \nabla^2_{jk} Q\\
&+
\tfrac{1}{2} Q^{-2} |\nabla  Q|^2 \nabla^2_{ij} Q
  + Q^{-2}(\nabla_i Q\nabla_k Q \nabla^2_{jk} Q  + \nabla_j Q\nabla_k Q  \nabla^2_{ik}Q )\\
    &+ \left(Q^{-1}\nabla_i Q \nabla_j Q-2\varepsilon g_{ij}\right)|\nabla^2 V|^2 + 2\eta^{-2}\varepsilon H h_{ij} + B_{ij}\, , 
\eea 
with 
\begin{align}
    B_{ij} &= (h_{ij} h_{kp}-h_{ip} h_{jk})(2\nabla^2_{kp}Q-Q^{-1}\nabla_k Q\nabla_p Q)\nonumber  \\
  & +(h_{ik}h_{kp}-h_{ip}H) (\nabla^2_{jp}Q-8^{-1}Q^{-2}|\nabla Q|^2\nabla_{p}Q\nabla_{j}Q)\nonumber\\
  &+(h_{jk}h_{kp}-h_{jp}H) (\nabla^2_{ip}Q-8^{-1}Q^{-2}|\nabla Q|^2\nabla_{i}Q\nabla_{p}Q) \, , \nonumber  
\end{align}
\bea \label{eq-N2-2}    {(N_2)}_{ij} &= 
    -\frac{1}{4}Q^{-1}h^2(\nabla Q,\nabla Q)\left(Q^{-1}\nabla_i Q \nabla_j Q-2\varepsilon g_{ij}\right)\\
    &+Q^{-1}|\nabla Q|^2\nabla^2_{im}V \nabla^2_{jm}V  +\Psi_{ij} +4\nabla_{\ell}(|\nabla V|^2)\nabla_{\ell}(\nabla_{i}V\nabla_{j}V)\\
    &+\nabla_{Z}(-2|\nabla V|^2\nabla_{i}V\nabla_{j}V)+\nabla_Z(|\nabla V|^2)\varepsilon g_{ij} \, ,
\eea 
 with 
\begin{align}\label{eq-Psi1}
   \Psi_{ij} &=  4\eta\nabla_p V h_{pk}(\nabla_k V \nabla^2_{ij}V+\nabla_i V\nabla^2_{jk}V + \nabla_j V \nabla^2_{ik}V)\\
    &-4\nabla_p V h_{pk}(\eta^3 V\nabla^2_{kq}V \nabla_q V \nabla^2_{ij}V-VG_{kij}),\notag
\end{align}
for some $(0,3)$-tensor $G$ satisfying $G_{kij}X^iX^j=0$ for any $X$ perpendicular to the spherical direction, and  
\begin{align}
    (N_3)_{ij}
    =& - 3\eta^{-2}V^{-2}\Big [(n-k)-6|\nabla V|^2 +\tfrac{V^2}{2t}\Big (1-\tfrac{1}{\log(-t)}\Big ) \Big ] \gamma g_{ij}\nonumber\\
    &-6  {\eta^{-1}}V^{-1}  h(\nabla V,\nabla V)  \gamma g_{ij} .\nonumber 
\end{align}
\end{proposition}
\begin{proof} 
{Combining  Proposition \ref{prop_evol_hess} (evolution of Hessian) and Lemma \ref{lemma-nabla-h} (Codazzi term),
\bea\label{eq-1111} 
&\quad (\partial_t - \Delta+\nabla _Z) \nabla ^2 _{ij}Q\\ 
&\,=+ \tfrac{1}{2} Q^{-2} |\nabla  Q|^2 \nabla^2_{ij} Q
 + Q^{-2}(\nabla_i Q\nabla_k Q \nabla^2_{jk} Q  + \nabla_j Q\nabla_k Q  \nabla^2_{ik} Q)\\
 &\,\, {- Q^{-1} \nabla^2_{ik} Q \nabla^2 _{jk} Q}
 - Q^{-3}|\nabla  Q|^2\nabla_i  Q \nabla_j Q {- (h_{ij} h_{kp}-h_{ip} h_{jk})Q^{-1}\nabla_k Q\nabla_p Q} \\
 &\,   + 2(h_{ij}h_{kp} - h_{ip}h_{jk}) \nabla^2_{kp}Q+ (h_{ik}h_{kp}-h_{ip}H) \nabla^2_{jp}Q \!+\!(h_{jk}h_{kp}-h_{jp}H) \nabla^2_{ip}Q\\
 & + \Psi_{ij}.
   \eea 
 By the commutator identity and Proposition \ref{lemma-ev-eq} (evolution of profile function), \ \bea  \label{eq-evolDVDV'}
     &\quad(\partial _t - \Delta ) (\nabla_{i}V\nabla_{j}V) \\
     &=\tfrac{n-k}{2}Q^{-2} \nabla_{i}Q\nabla_{j}Q -2\nabla_{im}^2V \nabla^2_{jm}V \\ 
     &+4^{-1}Q^{-1}(h_{im}h_{m\ell}-h_{i\ell}H)\nabla_{\ell}Q\nabla_{j}Q+4^{-1}Q^{-1}(h_{jm}h_{m\ell}-h_{j\ell}H)\nabla_{\ell}Q\nabla_{i}Q.
\eea 
Combining \eqref{eq-evolDVDV'}, and $\partial_t g_{ij} = - 2H h_{ij}$, 
\bea
     &(\partial _t - \Delta ) |\nabla V|^2  =\tfrac{n-k}{2}Q^{-2}|\nabla Q|^2 
      +(2Q)^{-1}h^2(\nabla Q, \nabla Q)-2|\nabla^2 V|^2. \eea 
\bea \label{eq-2222}
    &\quad(\partial_t-\Delta+\nabla_Z)(-2|\nabla V|^2\nabla_iV\nabla_jV)\\
    &=-\tfrac{n-k}{2}Q^{-3}|\nabla Q|^2\nabla_{i}Q\nabla_{j}Q +Q^{-1}|\nabla Q|^2\nabla_{im}^2V \nabla^2_{jm}V \\
     &\, -8^{-1}Q^{-2}|\nabla Q|^2 \big [ (h_{im}h_{m\ell}-h_{i\ell}H)\nabla_{\ell}Q\nabla_{j}Q+(h_{jm}h_{m\ell}-h_{j\ell}H)\nabla_{\ell}Q\nabla_{i}Q \big ]  \\
     &\, -4^{-1}Q^{-2} h^2(\nabla Q, \nabla Q)\nabla_{i}Q\nabla_{j}Q+Q^{-1}\nabla_{i}Q\nabla_{j}Q|\nabla^2 V|^2\\
     &\, +4\nabla_{\ell}(|\nabla V|^2)\nabla_{\ell}(\nabla_{i}V\nabla_{j}V)+\nabla_Z(-2|\nabla V|^2\nabla_iV\nabla_jV) .  
\eea 
Next, $ \nabla_Z \eps =  \langle Z, \nabla   { \gamma}\rangle = -6(V^{-2}|\nabla V|^2 + \eta V^{-1}h(\nabla V,\nabla V)) \gamma  $ implies  
\bea & \quad (\partial_t -\Delta+\nabla_Z) \eps\\
&=   \Big  [  \Big (\tfrac{-t}{\log(-t)}\Big )^{3/2} (\partial_t - \Delta) V^{-3} +\tfrac{3}{2t}\Big (1-\tfrac{1}{\log(-t)}\Big )\gamma \Big  ] g_{ij}+\nabla _Z \eps \\
& = 3V^{-2}  \Big [(n-k)-6|\nabla V|^2 +\tfrac{V^2}{2t}\Big (1-\tfrac{1}{\log(-t)}\Big ) \Big ] \gamma  -6 \eta V^{-1}h(\nabla V,\nabla V) \gamma .\eea 
Since $- \eps g_{ij} + \eps |\nabla V|^2 g_{ij} = - \eta^{-2}\eps g_{ij}  $,
\bea \label{eq-3333}& \quad (\partial_t-\Delta+\nabla_Z) ( -\eta^{-2}\eps   g_{ij}) \\
&=2 \eta^{-2}\eps Hh_{ij}  \\
&\, +\Big[\frac{n-k}{2}Q^{-2}|\nabla Q|^2 
      +(2Q)^{-1}h^2(\nabla Q, \nabla Q)-2|\nabla^2 V|^2\Big]\varepsilon g_{ij}\\
& \, -3\eta^{-2} V^{-2}  \Big [(n-k)-6|\nabla V|^2 +\tfrac{V^2}{2t}\Big (1-\tfrac{1}{\log(-t)}\Big ) \Big ] \gamma g_{ij}\\
& \, +6\eta^{-1}  V^{-1}h(\nabla V, \nabla V)\gamma g_{ij}+  \nabla_Z( |\nabla V|^2 ) \eps g_{ij}  {-12\eta^{-1}  V^{-1}h(\nabla V, \nabla V)\gamma g_{ij}} .\eea 
Here the last term comes from $-2 \langle \nabla \eps, \nabla |\nabla V|^2 \rangle g_{ij}$.
Adding \eqref{eq-1111}, \eqref{eq-2222}, and \eqref{eq-3333}, we obtain the result.
}

\end{proof}

Next, we derive equations appeared in Section \ref{spectral_uniqueness}. In the following, we denote $
w_\cC=v_1\chi_\cC(v_1)-v_2\chi_\cC(v_2)$, $D=(\partial_{y_i})$, $D^2 = (\partial^2_{y_iy_j})$ and $A\! :\! B=a_{ij}b_{ij}$, where $\chi_{\cC}$ is defined in the beginning of Section \ref{spectral_uniqueness}. We simplify the notation in the computations that follow,  for any cut-off scalar function  (e.g. for $\chi_\cC$, $\chi_\cC'$ or $\chi_\cC''$) we denote by $w^\chi$  the function 
\begin{equation}
\label{eq-diff-cut}
w^\chi (\bry, \tau) = \chi  (v_1(\bry,\tau)) -\chi (v_2(\bry,\tau)). 
\end{equation}
In particular, note that $w^{\textrm{id}}=w$ and $\left|w^\chi\right| \leq  |w| \sup|\chi'|$.

\begin{proposition}[{evolution of $w_\cC$}]\label{proposition-evolution-wC} The function $w_\cC$ satisfies
\begin{align}\label{evolution-wC}
(\partial_\tau -\mathcal{L})w_\cC =&\cE[w_\cC] + \bar{\cE}[w,\chi_\cC(v_1)]-\mathcal{E}[v_2w^{\chi_\cC}]- 2 D  v_2 \cdot D w^{\chi_\cC}\nonumber \\
&+ v_2 (\partial_\tau -\mathcal{L}) w^{\chi_\cC} + w^{\chi_\cC} (\partial_\tau -\mathcal{L}+1)v_2,
 \end{align}
where $\mathcal{L}$ is the Ornstein-Ulenbeck operator in \eqref{OU_operator}, $\mathcal{E}$ and  $\bar{\mathcal{E}}$  is given by 
\begin{multline}\label{eqn-E15}
 \cE[w] =  - \frac{Dv_1\otimes Dv_1\! :\! D^2 w}{1+|Dv_1|^2 } - \frac{D^2v_2\! :\! D(v_1+v_2)\!\otimes\! Dw}{(1+|Dv_1|^2)} \\
+\frac{D^2v_2\! :\! Dv_2\otimes Dv_2\, D(v_1+v_2)\!\cdot\! Dw}{(1+|Dv_1|^2)(1+|Dv_2|^2)} + \frac{2{(n-k)}-v_1v_2}{2v_1v_2}w,
 \end{multline}and
\begin{equation}
\label{eq-line-E}
\bar{\cE}[w,\chi_\cC(v_1)]= (\partial_\tau -\mathcal{L})(w\chi_\cC(v_1))- \cE[ w \chi_\cC(v_1)].
\end{equation}
\end{proposition}
\begin{proof}
The proof is based on inspecting the computation in \cite[Lem 4.2, Lem 4.3]{CDDHS} and using the evolution equation of $v$:
 \be\label{eqn-v}
 v_\tau= \left(\delta_{ij}-\frac{v_{y_i}v_{y_j}}{1+|Dv|^2}\right)\, v_{y_iy_j}-\frac{1}{2}  \, y_i v_{y_i} + \frac{v}2 -\frac{n-k}{v}. 
\ee

Subtracting the evolution equations of $v_1$ and $v_2$ yields
\bea w_\tau 
 &= \mathcal{L} w - \frac{Dv_1\otimes Dv_1\!:\! D^2v_1 }{1+|Dv_1|^2} + \frac{Dv_2\otimes Dv_2\!:\! D^2v_2 }{1+|Dv_2|^2} +\frac{2(n-k)-v_1v_2}{2v_1v_2} w  .\eea 
Now, using the product rule for differences we compute
\bea   & \quad  \frac{Dv_{1}\otimes  Dv_{1}\! :\!  D^2 v_{1} }{1+|Dv_1|^2} -\frac{Dv_{2}\otimes  Dv_{2}\! :\!  D^2 v_{2} }{1+|Dv_2|^2} \\  & =   \frac{Dv_{1}\otimes  Dv_{1}\! :\!  D^2 w }{1+|Dv_1|^2}  + \frac{Dv_{1}\otimes  Dw\! :\!  D^2 v_2 }{1+|Dv_1|^2} + \frac{Dw\otimes  Dv_2\! :\!  D^2 v_2 }{1+|Dv_1|^2} \\ & + \left( \frac{1}{1+|Dv_1|^2}- \frac{1}{1+|Dv_2|^2}\right) {Dv_2\otimes  Dv_2\! :\!  D^2 v_2 } .\eea 
Moreover, we have
\bea Dv_{1}\otimes  Dw\! :\!  D^2 v_2 +Dw\otimes  Dv_2\! :\!  D^2 v_2 = Dw\otimes D ( v_1+v_2)\! :\!  D^2 v_2,\eea
and
\bea \frac{1}{1+|Dv_1|^2}- \frac{1}{1+|Dv_2|^2}= - \frac{ Dw \cdot D(v_1+v_2)}{(1+|Dv_1|^2)(1+|Dv_2|^2)}.\eea 
The rest of derivation is straightforward.\end{proof}

Let us consider the polar coordinates $(y_1,\cdots,y_{k}) = y \omega $ for  $y>0$ and $\omega\in{S}^{k-1}$, and view $v$ as a function on the polar coordinates $v( y, \omega,\tau)$. Let us define the inverse profile function $Y=Y(y,\omega,\tau)$ by 
\[y= Y(v(y,\omega, \tau),\omega ,\tau). \]   We impose the round cylindrical metric $ dv^2 + g_{{S}^{k-1}}$ on $(v,w)\in \mathbb{R}\times {S}^{k-1}$. We derive an evolution equation of $Y$.    \begin{proposition}[evolution equation of inverse profile function  $Y$]\label{evolution equation of inverse profile function}
   \begin{align} \label{eq-Ytensor}
    Y_{\tau}
    &=\frac{(Y^2 \!+\! |\nabla_{S^{k-1}}Y|^2)\nabla^2_{vv}Y  \!-\! 2 \nabla^2 Y ( \nabla_{{S}^{k-1} }Y ,  \nabla_{\mathbb{R}} Y ) + (1 +|\nabla_\mathbb{R} Y|^2)\Delta_{S^{k-1}}Y}{ Y^2\, (1 + |\nabla_\mathbb{R} Y|^2)+|\nabla_{S^{k-1}}Y|^2} \\\nonumber
    &+\Big (\frac{n-k}{v}-\frac{v}{2}\Big )\nabla_v Y-\frac{|\nabla_{S^{k-1}}Y|^2}{Y[Y^2(1+|\nabla_\mathbb{R} Y|^2)+|\nabla_{S^{k-1}}Y|^2]}+\frac{Y}{2}-\frac{k-1}{Y}\\\nonumber
 &+ \frac{1}{Y^2} {\frac{\Delta_{S^{k-1}}Y|\nabla_{S^{k-1}}Y|^2-\nabla^2 Y( \nabla_{{S}^{k-1}}Y, \nabla_{{S}^{k-1}}Y)}{Y^{2}(1+|\nabla_\mathbb{R} Y|^2)+|\nabla_{S^{k-1}}Y|^2}}.
    \end{align}
   Here, $|\cdot|$ and $\nabla$ are the norm and connection induced by the cylindrical metric $dv^2 +g_{{S}^{k-1}}$. $\nabla _{{S}^{k-1}} V$ and $\nabla_{\mathbb{R}}V$ are orthogonal projections of $\nabla V$ onto the spherical part and $\mathbb{R}$-part, respectively, and $\Delta_{{S}^{k-1}}$ is the Laplacian in the spherical part. 
   \end{proposition}

In what follows, we will express and prove identities using normal coordinates $(v,\varphi_1,\ldots, \varphi_{k-1})$ based at each given point. We refer readers to Remark \ref{remark-cylindericalconnection} for basic useful identities. Summation over repeated indices (from $1$ to $k-1$) is assumed.

\begin{proof}Recall 
    \begin{equation}\label{eq-recallv}
    v_{\tau}=\Big (\delta_{ij}-\frac{v_{y_i}v_{y_j}}{1+|Dv|^2}\Big )v_{y_{i}y_{j}}-\frac{1}{2}y_{i}v_{y_{i}}+\frac{v}{2}-\frac{n-k}{v}.
\end{equation}
Our first goal is to write the evolution of $v$ in polar coordinates. We denote  $\mathbf{y}=y\omega$, where $y>0$ and $\omega=\omega(\varphi_{1},\dots \varphi_{k-1})$ is a local parametrization of sphere $S^{k-1}$ via normal coordinates $(\varphi_{1},\dots \varphi_{k-1})$ of ${S}^{k-1}$ at a given point. We will compute the evolution equation at this point. Note that $\partial_y$ and $y^{-1}\partial_{\varphi_m}$, for  $m=1,\cdots,k-1$, form an orthogonal basis of tangent space of Euclidean space. If $D$ denotes the flat connection (a.k.a. usual derivative) on $\mathbb{R}^n$, \bea |Dv|^2 = v_y^2 + \sum_{m=1}^{k-1}(y^{-1} v_{\varphi_m})^2, \quad\text{and}\quad  y_i v_{y_i} = Dv \cdot \mathbf{y} = y v_y .\eea 
In the following computations, repeated index represents summation from $1$ to $k-1$. e.g. $v_{\varphi_m}^2= \sum_{m=1}^{k-1} v_{\varphi_m}^2$, $v_{\varphi_m \varphi_m} = \sum_{m=1}^{k-1} v_{\varphi_m\varphi_m}$ and $v_{\varphi_\ell}v_{\varphi_m}v_{\varphi_{\ell }\varphi_m}= \sum_{m=1}^{k-1}\sum_{\ell =1}^{k-1}v_{\varphi_\ell}v_{\varphi_m}v_{\varphi_{\ell }\varphi_m}$. From the definition of the second fundamental form on the sphere,
\bea D^2 v(\partial_{\varphi_m}, \partial_{\varphi_\ell})&= v_{\varphi_m\varphi_\ell} - \langle D v , D _{{\varphi_\ell }}\partial_{ \varphi_{m}} \rangle  \\
&=v_{\varphi_m\varphi_\ell} + y v_y \delta_{m\ell} ,\eea 
and similarly 
\bea D^2 v(\partial_{y}, \partial_{\varphi_m})&= v_{\varphi_my} - \langle D v , D _{{\varphi_m }}\partial_{y} \rangle \\
&= v_{\varphi_my} - y^{-1} v_{\varphi_m} .\eea 
Therefore, we get the Laplacian in the polar coordinates 
\bea \Delta v &=  \nabla^2 v( \partial_y,\partial_y) + \nabla^2 v(y^{-1} \partial _{\varphi_m},y^{-1}\partial _{\varphi_m})  \\ &=
v_{yy}+(k-1)y^{-1}v_{y}+y^{-2}v_{\varphi_m\varphi_m},\eea
and $  v_{y_i} v_{y_j}v_{y_iy_j} = D^2 v(D v, Dv) $ can be written as  
\bea &\quad v_y^2 D^2 v( \partial_y, \partial_y) +2y^{-2} v_yv_{\varphi_m}D ^2 v (\partial_y, \partial_{\varphi_m}) + y^{-4}v_{\varphi_m} v_{\varphi_\ell }D^2 v ( \partial_{\varphi_m},  \partial_{\varphi_\ell})\\
  &=v_{y}^2v_{yy}
 +y^{-4}v_{\varphi_\ell \varphi_m}v_{\varphi_\ell}v_{\varphi_m}+2y^{-2}v_{y\varphi_m}v_{y}v_{\varphi_m}-y^{-3}v_{\varphi_m}v_{\varphi_m}v_{y}.\eea 
By plugging above into \eqref{eq-recallv}, at the given point we get
\bea \label{vtau_evolution_sphere}
v_{\tau}&= \frac{(y^2+v_{\varphi_m}^2)v_{yy}
 -2v_{y\varphi_m}v_{y}v_{\varphi_m}}{y^{2}(1+v^{2}_{y})+v_{\varphi_m}^2} + \frac{ v_{\varphi_m\varphi_m}}{y^2}-\frac{1}{y^2}\frac{ v_{\varphi_\ell \varphi_m}v_{\varphi_\ell}v_{\varphi_m}}{ y^{2}(1+v^{2}_{y})+v_{\varphi_m}^2}\\
 &+\Big  (\frac{k}{y^2}-\frac{1}{2}-\frac{1+v^2_{y}}{y^2(1+v^2_y)+v_{\varphi_m}^2}\Big  )yv_{y}+\frac{v}{2}-\frac{n-k}{v}.
\eea 
Finally, we derive an equation of $Y(v,\omega,\tau)$. Differentiating the identity $y= Y(v(y,\omega ,\tau),\omega,\tau)$, there are identities 
\bea &v_\tau =-\frac{Y_\tau }{Y_v}, \quad  v_y= \frac{1}{ Y_v},\quad   v_{\varphi_m}= - \frac{Y_{\varphi_m}}  {Y_v} ,\\ &v_{yy}= - \frac{Y_{vv}}{Y_v^3},\quad v_{y\varphi_{m}} =- \frac{Y_{\varphi_{m} v}}{Y_v^2}+ \frac{Y_{\varphi_m}Y_{vv}}{Y_v^3},\\ 
&v_{\varphi_m\varphi_\ell}=- \frac{Y_{\varphi_m\varphi_\ell}}{Y_v}+ { \frac{ Y_{\varphi_m} Y_{\varphi_{\ell} v}}{Y_v^2}+ \frac{ Y_{\varphi_\ell } Y_{\varphi_{m} v}}{Y_v^2} } - \frac{Y_{\varphi_{m}}Y_{\varphi_{\ell}}  Y_{vv}}{Y_v^3} .
\eea  
Plugging these into \eqref{vtau_evolution_sphere} yields \eqref{eq-Ytensor} at the given point written in the chosen coordinates: \bea  \label{eq-evolYcoord}
    Y_{\tau} &=\frac{(Y^2 + Y_{\varphi_m}^2)Y_{vv}  - 2Y_{\varphi_m}Y_v Y_{\varphi_m v} +(1+Y_v^2)Y_{\varphi_m\varphi_m}}{ Y^2\, (1 + Y_v^2)+Y_{\varphi_m}^2}  \\ &+\frac{1}{Y^2}\frac{Y_{\varphi_\ell}^2   Y_{\varphi_m\varphi_m}-Y_{\varphi_\ell \varphi_m}Y_{\varphi_\ell}Y_{\varphi_m}}{Y^{2}(1+Y^{2}_{v})+Y_{\varphi_m}^2}\\
    &+\Big (\frac{n-k}{v}-\frac{v}{2}\Big )Y_{v}-\frac{Y_{\varphi_m}^2}{Y[Y^2(1+Y_{v}^2)+Y_{\varphi_m}^2]}+\frac{Y}{2}-\frac{k-1}{Y}.
\eea  Recall that \eqref{eq-Ytensor} is independent of coordinate charts (tensorial). Since one can repeat above at any given points, we conclude \eqref{eq-Ytensor} holds.

\end{proof}
Next we derive an evolution equation of $W=Y-\bar{Y}$ when $Y(v,\omega,\tau)$ and $\bar Y(v,\omega,\tau)$ are two solutions. Though the evolution equation below can be stated without using charts, here we write it in normal coordinates of cylinder $(v,\varphi_m)$ for the sake of notational simplicity. One can recall Remark \ref{remark-cylindericalconnection} for basic identities.
\begin{proposition}[{evolution of $W$}]\label{lemma-ev-W-appendix} Let $\varphi_m$ be a normal coordinates of ${S}^{k-1}$ at a given point $\omega$. At this point, the function $W(v,\omega,\tau )$ evolves by
\begin{align}\label{eq-W}
    W_{\tau}
    &=\frac{(Y^2 + |\nabla_{S^{k-1}}Y|^2)W_{vv}  - 2 Y_{\varphi_m}Y_v W_{\varphi_m v} + (1+Y_v ^2)  W_{\varphi_m \varphi_m}   }{Y^2\, (1 + Y_v^2)+|\nabla_{S^{k-1}}Y|^2} \\\nonumber
    &  + \frac{1}{Y^2}\frac{ |\nabla _{{S}^{k-1}} Y|^2W_{\varphi_m \varphi_m }  - Y_{\varphi_\ell} Y_{\varphi_m}W_{\varphi_\ell \varphi_m}}{Y^{2}(1+ Y_v^2)+|\nabla_{S^{k-1}}Y|^2}+\langle a\partial_v +b_m \partial_{\varphi_m},\nabla W\rangle +cW\nonumber,
\end{align}
where $|\cdot|$ and $\nabla$ denote the metric and the connection of the round cylinder, and $\nabla_{{S}^{k-1}}G$ denote the spherical part of $\nabla G$.  The components of vector fields $a\partial_v $, $\sum_{m=1}^{k-1}b_m \partial_{\varphi_m}$, and the coefficient $c$ are as follows: \\
\bea \label{eq-anew}
    a&=\frac{n-k} {v} - \frac v 2 -2 \frac{\bar Y_{{\varphi_m} v}Y_{\varphi_m}}{D}  + \frac{\bar Y^2 (Y_v +\bar Y_v)}{D\bar D} \left[ \frac{|\nabla_{S^{k-1}}\bar Y|^2}{\bar Y} -\bar B\right],
\eea 
\bea \label{eq-bnew}
    b_{m}&= \frac{\bar Y_{vv} (Y_{\varphi_m}+\bar Y_{\varphi_m} )-2\bar Y_{{\varphi_m} v}\bar Y_v }{D} - \frac{Y_{\varphi_m} +\bar Y_{\varphi_m}}{YD}  \\
    &{-\frac{1}{DY^2 }\nabla_{\varphi_\ell}(Y+\bar{Y})\nabla^2_{\varphi_m\varphi_{\ell}}\bar{Y} }+\frac{ (Y_{\varphi_m}  +\bar Y_{\varphi_m} )}{D\bar D} \left[ \frac{|\nabla_{S^{k-1}}\bar Y|^2}{\bar Y} -\bar B\right],
\eea 
and
\bea \label{eq-cnew}
    c&=\frac12 +\frac{{k-1}} {Y\bar Y} + \frac{\bar Y_{vv} (Y+\bar Y ) }{D} + \frac{|\nabla_{S^{k-1}}\bar Y|^2}{DY\bar Y}-{\frac{ Y+\bar Y}{Y^{2}\bar Y^{2}}\Delta_{{S}^{k-1}} \bar Y} \\
&+\frac{Y+\bar{Y}}{DY^2\bar Y^2 }\nabla_{\varphi_\ell}\bar{Y}\nabla_{\varphi_m}\bar{Y}\nabla^2_{\varphi_m\varphi_{\ell}}\bar{Y}    
    + \frac{ (Y  +\bar Y )(1+Y_v^2)}{D\bar D} \left[ \frac{|\nabla_{S^{k-1}}\bar Y|^2}{\bar Y} -\bar B\right],
\eea
where 
\bea 
D&= Y^2\, (1 + Y_v^2)+|\nabla_{S^{k-1}}Y|^2\quad \bar D= \bar Y^2\, (1 + \bar Y_v^2)+|\nabla_{S^{k-1}}\bar Y|^2,
\eea and
\begin{equation}
    \bar{B}=(\bar{Y}^2 + |\nabla_{S^{k-1}}\bar{Y}|^2)\nabla^2_{vv}\bar{Y}  \!-\! 2 \nabla^2_{\varphi_m v}\bar{Y}\nabla_{\varphi_m}\bar{Y}\nabla_v \bar{Y} {-\bar{Y}^{-2}\nabla^2_{\varphi_m\varphi_{\ell}}\bar{Y}\nabla_{\varphi_\ell}\bar{Y}\nabla_{\varphi_m}\bar{Y}}\!.
\end{equation}
\end{proposition}
\begin{proof}
 Let us rearrange terms in the first and the second lines of the right hand side of \eqref{eq-Ytensor} (or \eqref{eq-evolYcoord}) and group the terms as 
\[\partial_\tau Y= Y_1^\tau +Y_2^\tau + Y_3^\tau, \]
where
\[Y^\tau_1:= \frac{(Y^2 + Y_{\varphi_m}^2)Y_{vv}  - 2Y_{\varphi_m}Y_v Y_{\varphi_m v} - Y^{-2} Y_{\varphi_m}Y_{\varphi_\ell} Y_{\varphi_m \varphi_\ell}}{ Y^2\, (1 + Y_v^2)+Y_{\varphi_m}^2} ,\]
\[Y^{\tau}_2:=\frac{1}{Y^2} Y_{\varphi_m \varphi_m} , \]
\[Y^{\tau}_3:= \Big   (\frac{n-k}{v}-\frac{v}{2}\Big   )Y_{v}-\frac{Y_{\varphi_m}^2}{Y[Y^2(1+Y_{v}^2)+Y_{\varphi_m}^2]}+\frac{Y}{2}-\frac{k-1}{Y}.\]
 Also, we similarly define $\bar Y^{\tau}_i$ by replacing $Y$ with $\bar Y$. 
With this notation, $W_\tau = \sum_{i=1}^3 (Y^\tau_i - \bar Y ^\tau _i)$, and we expand each $Y^\tau_i- \bar Y^\tau_i$. 
First, using the product rule for differences, 
\bea Y^{\tau}_2- \bar Y^{\tau}_2= \frac{1}{Y^2}\Delta_{{S}^{k-1}} W - \frac{Y+\bar Y}{Y^2 \bar Y^2 } W\Delta_{{S}^{k-1}} \bar Y.\eea 
Similarly,
\bea Y^{\tau}_3- \bar Y^{\tau}_3&= \Big  ( \frac{n-k} {v} - \frac v 2\Big ) W_v+ \Big  (\frac12 +\frac{k-1}{Y \bar Y} \Big  )W \\
&\, -  \frac{Y_{\varphi_m} +\bar Y_{\varphi_m}}{YD}W_{\varphi_m} +\bar Y_{\varphi_m}^2  \Big ( \frac{W}{DY\bar Y} +  \frac{D-\bar D}{\bar Y D \bar D}\Big  ) ,\eea 
where $D$ and $\bar D$ are as defined in the statement. 
\bea &\quad Y^\tau_1- \bar Y^{\tau}_1\\
&=\frac{(Y^2 + Y_{\varphi_m}^2)W_{vv}  - 2Y_{\varphi_m}Y_v W_{\varphi_m v} -Y^{-2}Y_{\varphi_\ell}Y_{\varphi_m}W_{\varphi_\ell \varphi_m}}{ D}  \\
&\, +\frac{\bar Y_{vv}[(Y+\bar Y)W + (Y_{\varphi_{m}}+\bar Y_{\varphi_{m}})W_{\varphi_m}]}{D} - 2\frac{ \bar Y_{\varphi_m v}( Y_{\varphi_m} W_v +\bar Y_{v} W_{\varphi_m}) }{D}\\
&\, - \frac{Y_\ell  +\bar Y_\ell  }{DY^2} \bar Y_{\varphi_m \varphi_\ell}W_{\varphi_m}  +\frac{Y+\bar Y}{DY^2 \bar Y^2} \bar Y_{\varphi_\ell} \bar Y_{\varphi_m} \bar Y _{\varphi_m\varphi_\ell} W - \frac{D-\bar D}{D\bar D} \bar B ,\eea 
where $\bar B$ is as defined in the statement. Finally, plugging in
 \bea
D-\bar D = \bar Y^2 (Y_v+\bar Y_v)W_v+(Y_{\varphi_m} +\bar Y_{\varphi_m})W_{\varphi_m} + (Y+\bar Y)(1+Y_v^2) W 
\eea
and collecting the coefficients of $W_v$, $W_\varphi$ and $W$ from the computations of $Y^{\tau}_i-\bar Y^\tau_i$ above, we finish the computation of evolution equation of $W$.
\end{proof}
 \section{A graphical estimate lemma}\label{graphical}

\begin{lemma}[graphical estimate]\label{graphical_estimates claim}
 There exist constants $\kappa>0$ small enough and $\tau_{*}>-\infty$ negative enough with the following significance: If two $k$-ovals $\mathcal{M}^1, \mathcal{M}^2 \in \bar{\mathcal{A}}_{\kappa}(\tau_0)$, namely they are $k$-ovals $\kappa$-quadratic at $\tau_0\leq \tau_{*}$ and satisfy orthogonality conditions \eqref{condition_positive}--\eqref{OC21}, then  for each fixed $m\in \mathbb{N}$, fixed $0<l<+\infty$, fixed $\theta>0$ and $\tau_0\le \tau_*$, there are $C(\tau_0, l, m, \theta)<\infty$ and $\delta(\tau_0, l, \theta)>0$ such that if  the condition $|\mathcal{E}(\tau_0)(\mathcal{M}^1)-\mathcal{E}(\tau_0)(\mathcal{M}^2)|\leq \delta(\tau_0, l, \theta)$ holds, then for $\tau\in [\tau_0-l, \tau_0]$ one can write renormalized flow $\bar{M}^1_{\tau}$ as a graph over $\bar{M}^2_{\tau}$  of a function $\zeta(\tau)$ satisfying
\begin{align}\label{graph-w-W}
&\quad||\zeta(\cdot,\tau)||_{C^{m}} \\
&\leq C(\tau_0, l, m, \theta) (||w_{\cC}(\cdot,\tau)||_{C^{m}(B^1)}  +  ||W_{\cT}(\cdot,\tau)||_{C^{m}(B_{17\theta/18})}),
\end{align}
\end{lemma}

\begin{proof}[Proof]

 
By definition one easily sees that  
    \begin{align}\label{p0w-Ediff}
       \|\mathfrak{p}_0 w_{\mathcal{C}}(\tau_0)\|_{\mathcal{H}} \leq |\mathcal{E}(\tau_0)(\mathcal{M}^1)-\mathcal{E}(\tau_0)(\mathcal{M}^2)| .
    \end{align}
Let $0<C(\tau_0, l, m, \theta)<+\infty$ be the constant from regularity estimate \eqref{C20-L2} and $0<C<+\infty$ be the constant from \eqref{part1}. 
We define
\begin{align}
    \delta(\tau_0, l, \theta) = \frac{1}{(1+C)(1+C(\tau_0, l, 2, \theta))(|\tau_0|+l)^{10}}.
\end{align}
Note that we fixed $m=2$ in definition of $\delta(\tau_0, l, \theta) $, that is because we only need smallness of low derivatives of $w$ and $W$ to ensure the existence of the graph function. The higher derivatives estimate of the graph function is irrelevant to $\delta(\tau_0, l, \theta)$.

In the following, we also use $C_1, C_2$ and so on  to denote constants that may change from line or line. 

Then our spectral assumption $|\mathcal{E}(\tau_0)(\mathcal{M}^1) -\mathcal{E}(\tau_0)(\mathcal{M}^2)| \leq \delta(\tau_0, l, 2, \theta) $ together with the following estimates from \eqref{C20-L2} 
\begin{align}
	&\quad\sup_{\tau\in [\tau_0-l, \tau_0]}[||w_{\cC}(\cdot,\tau)||_{C^{m}(B^1)}  +  ||W_{\cT}(\cdot,\tau)||_{C^{m}(B_{17\theta/18})}]\\
    &\leq
	C(\tau_0, l, m, \theta) \|\mathfrak{p}_0 w_{\mathcal{C}}(\tau_0)\|_{\mathcal{H}}.
\end{align}
and \eqref{p0w-Ediff} imply that
\begin{align}
    ||w_{\cC}(\cdot,\tau)||_{C^{2}(B^1)}  +  ||W_{\cT}(\cdot,\tau)||_{C^{2}(B_{17\theta/18})} \leq
	\frac{1}{|\tau|^{10}}.
\end{align}
whenever $\tau \in [\tau_0 - l, \tau_0]$.

We now show that $\bar{M}_{\tau}^1$ is a graph $\zeta$ over $\bar{M}_{\tau}^2$ with the desired inequality. The idea is based on implicit function theorem argument together with using the information from unique uniform sharp asymptotics of $k$-ovals in Proposition \ref{strong_sharp_asymptotics} and the gradient estimates from Lemma \ref{std_pde_der_est}, Proposition  \ref{lemma-mainY}.  In the detailed proof below, we will fix time $\tau \in [\tau_0-l, \tau_0]$ and for ease of notation, we omit the $\tau$ in the profile functions, and the constant would depend on $\tau_0$, $l$ and the number of derivatives $m$.

By convexity, it suffices to show that there is a function $\zeta$ on $\bar{M}_{\tau}^2$ such that   $p + \zeta(p)\nu_{ \bar{M}_{\tau}^2}(p) \in  \bar{M}_{\tau}^1$ for every point $p\in \bar{M}_{\tau}^2$.

We will find $\zeta$ in two different regions: $\{v_2\geq \frac{39\theta}{40}\}$ and $\{v_2\leq \frac{99}{100}\theta\}$.

\noindent\textbf{Case1:} Let us first consider the region $\{v_2\geq \frac{39\theta}{40}\}$.  
Recall that
\begin{align*}
	B^1 =& B(0,\sqrt{|\tau|(2(n-k)-(19\theta/20)^2)}) \subset\mathbb{R}^k,\\
    \ B^2 =& B(0,\sqrt{|\tau|(2(n-k)-(9\theta/10)^2)}) \subset\mathbb{R}^k
\end{align*}
and 
\begin{align*}
	\{v\geq \theta\} \subset B^1 \subset  B^2 \subset \{v\geq 7\theta/8\}.  
\end{align*} 
holds  when $\tau \leq \cT(\theta)$.

Now in  the region $\{v_2\geq \frac{39\theta}{40}\}$, we can represent $\bar{M}_{\tau}^i$ as $(\mathbf{y},v_i(\mathbf{y})\vartheta)$, where $\mathbf{y}\in B^1\subset \R^k$ and $\vartheta \in S^{n-k}$. 

{We first compute the outward unit normal vectors:
\begin{align}
	\nu_{\bar{M}_{\tau}^i}(\mathbf{y}, \vartheta) = \frac{(-Dv_i(\mathbf{y}),\vartheta)}{(1+|Dv_i(\mathbf{y})|^2)^{1/2}}.
\end{align} 
The graph over $\bar{M}_{\tau}^2$ with graph function $\zeta = \zeta(\mathbf{y})$ is represented by
\begin{align}
	\quad&(\mathbf{y},v_2(\mathbf{y})\vartheta) + \zeta(\mathbf{y})\nu_{\bar{M}_{\tau}^2}(\mathbf{y},\vartheta)\\
    =& \Big(\mathbf{y} - \frac{\zeta(\mathbf{y})Dv_2(\mathbf{y})}{(1+|Dv_2(\mathbf{y})|^2)^{1/2}}\ , \ \big(v_2(\mathbf{y}) + \frac{\zeta(\mathbf{y})}{(1+|Dv_2(\mathbf{y})|^2)^{1/2}}\big)\vartheta\Big).
\end{align} 
To ensure that the graph over $\bar{M}_{\tau}^2 \cap \{v_2\geq \frac{39\theta}{40}\}$ with graph function $\zeta $ is contained in $\bar{M}_{\tau}^1$, it suffices to require that
\begin{align}\label{graph-region1-require1}
v_1(\mathbf{y}+\mathbf{h}(\mathbf{y})) = v_2(\mathbf{y}) + \frac{\zeta(\mathbf{y})}{(1+|Dv_2(\mathbf{y})|^2)^{1/2}}, 
\end{align}
where 
\begin{align}
    \mathbf{h}(\mathbf{y}) = -\frac{\zeta(\mathbf{y})Dv_2(\mathbf{y})}{(1+|Dv_2(\mathbf{y})|^2)^{1/2}}.
\end{align}

We will show that for each $\mathbf{y}$, there exists $\zeta\in \mathbb{R}$  such that $|\zeta|\leq 2|w(\mathbf{y})|$ and \eqref{graph-region1-require1} holds. 
}

To this end, we consider the following function:
\begin{align}
	\mathbf{G}_{\mathbf{y}}(\zeta)  := w(\mathbf{y}) + v_1\Big(\mathbf{y} - \frac{\zeta Dv_2(\mathbf{y})}{(1+|Dv_2(\mathbf{y})|^2)^{1/2}}\Big) - v_1(\mathbf{y})  - \frac{\zeta}{(1+|Dv_2(\mathbf{y})|^2)^{1/2}}. 
\end{align}
where $\zeta \in \mathbb{R}$ and $\mathbf{y}\in B^1$.
By the unique sharp asymptotic of $k$-oval in Proposition \ref{strong_sharp_asymptotics} and convexity (see also Lemma \ref{std_pde_der_est}), the following gradient bound holds in $B^2$:
\begin{align}
	|D v_i|\leq \frac{C_1(\theta)}{|\tau|} < \frac{1}{4}.
\end{align}  
Since $|w(\mathbf{y})|\leq \frac{1}{|\tau|} \leq \frac{\theta}{1000(n-k)}< \frac{1}{2}\dist(B^1, \partial B^2)$, the point $\mathbf{y} - \frac{\zeta Dv_2(\mathbf{y})}{(1+|Dv_2(\mathbf{y})|^2)^{1/2}} $
stays in $B^2$ whenever $\mathbf{y}\in B^1$ and $|\zeta|\leq 2|w(\mathbf{y})|$.

Combining the gradient bound, we deduce that for all $\mathbf{y}\in B^1$ and $|\zeta|\leq 2|w(\mathbf{y})|$, the following inequality holds:
\begin{align}
	\Big|v_1\Big(\mathbf{y} - \frac{\zeta Dv_2(\mathbf{y})}{(1+|Dv_2(\mathbf{y})|^2)^{1/2}}\Big) - v_1(\mathbf{y})\Big|\leq \frac{|\zeta|}{4} \leq \frac{|w(\mathbf{y})|}{2}.
\end{align}
It follows that
\begin{align}
	\mathbf{G}_{\mathbf{y}}(2|w(\mathbf{y})|) \leq& w(\mathbf{y}) + \frac{1}{2}|w(\mathbf{y})| - \frac{2|w(\mathbf{y})|}{\sqrt{1+16^{-1}}} \leq 0, \\
	\mathbf{G}_{\mathbf{y}}(-2|w(\mathbf{y})|) \geq& w(\mathbf{y}) - \frac{1}{2}|w(\mathbf{y})| + \frac{2|w(\mathbf{y})|}{\sqrt{1+16^{-1}}} \geq 0. 
\end{align}
By the intermediate value theorem, there exists $\zeta \in [-2|w(\mathbf{y})|, 2|w(\mathbf{y})|]$ such that $\mathbf{G}_{\mathbf{y}}(\zeta) = 0$.

Note that for all $ \zeta \in [-2|w(\mathbf{y})|, 2|w(\mathbf{y})|]$, we can estimate
\begin{align}
	 \mathbf{G}_{\mathbf{y}}'(\zeta ) =& \frac{1}{(1+|Dv_2(\mathbf{y})|^2)^{1/2}}\Big(\Big\langle Dv_1\Big(\mathbf{y} - \frac{\zeta Dv_2(\mathbf{y})}{(1+|Dv_2(\mathbf{y})|^2)^{1/2}}\Big), {Dv_2(\mathbf{y})}\Big\rangle - 1\Big)\\
    \leq& -\frac{1}{2}.
\end{align}
Therefore for each $\mathbf{y}\in B^1$, the solution to $\mathbf{G}_{\mathbf{y}}(\zeta) = 0$ in $[-2|w(\mathbf{y})|, 2|w(\mathbf{y})|]$ is unique, which now we can denote by $\zeta(\mathbf{y})$. 

To estimate higher derivatives, we define $\mathbf{G}(\zeta, \mathbf{y}) := \mathbf{G}_{\mathbf{y}}(\zeta)$. By definition, it is clear that $\mathbf{G}(0,\mathbf{y}) =  w(\mathbf{y})$ for all $\mathbf{y}\in B^2$. 
By the derivative estimate in cylindrical region (see Corollary \ref{lemma-cylindrical}), for all multi-index $\alpha$ with $|{\alpha}|\leq m$ and ${\tau\in[\tau_0-l, \tau_0]}$, the following estimate holds in $[0,\frac{1}{|\tau|}]\times B^1$:
\begin{align}
	|D^{(\alpha)}\mathbf{G}| \leq ||w||_{C^m(B^1)} + C_1(\tau_0, l, m, \theta) |\zeta(\mathbf{y})| \leq  C_1(\tau_0, l, m, \theta)||w||_{C^m(B^1)}.
\end{align}
We have used the fact that $|\zeta| < 1$ here.
By implicit function theorem, the higher derivatives of $\zeta$ consists of polynomials of derivatives of $\mathbf{G}$ with order at most $m$, divided by powers of $\mathbf{G}_{\mathbf{y}}'$ with power at most $2m$. This implies that for all $|\alpha|\leq m$, ${\tau\in[\tau_0-l, \tau_0]}$
\begin{align}
	 \|\nabla^{\alpha}\zeta\| \leq C_1(\tau_0, l, m, \theta)||w||_{C^{m}(B^1)}
\end{align}
in $B^1$. 

To conclude, we have found $\zeta$ that satisfies \eqref{graph-w-W} and \eqref{graph-region1-require1} over $\bar{M}_{\tau}^2 \cap\{v_2 \geq\frac{39\theta}{40}\}$. Consequently $\bar{M}_{\tau}^1$ is a graph of $\zeta$ over $\bar{M}_{\tau}^2 \cap\{v_2 \geq\frac{39\theta}{40}\}$.  


\noindent\textbf{Case2:} We then consider the region $\{v_2\leq \frac{99\theta}{100}\}$. 
In this region, we can represent $\bar{M}_{\tau}^i$ as 
\begin{align}
	(Y_i(\mathbf{v,\omega})\omega, \mathbf{v}\vartheta),
\end{align}
where $\mathbf{v}\in [0,\theta]$, $\omega\in S^{k-1}$ and $\vartheta\in S^{n-k}$ and $i = 1, 2$.

Then again we compute the outward normal vector for $\bar{M}_{\tau}^2$:
\begin{align}
	\nu_{\bar{M}_{\tau}^2}(\mathbf{v},\omega) = \Big(1+(Y_2)_v^2 + \frac{|\nabla_{S^{k-1}}Y_2|^2}{Y_2^2}\Big)^{-\frac{1}{2}}\Big(\omega-\frac{\nabla_{S^{k-1}}Y_2}{Y_2}, -(Y_2)_{{v}}\vartheta\Big).
\end{align}
Therefore, the graph over $\bar{M}_{\tau}^2$ with graph function $\zeta$ is represented by
\begin{align}
	&(Y_2 \cdot \omega, \mathbf{v}\vartheta) + \zeta\nu_{\bar{M}_{\tau}^2}\\
	=& \Big(Y_2\cdot\omega+\frac{\zeta}{\Big(1+(Y_2)_v^2 + \frac{|\nabla_{S^{k-1}}Y_2|^2}{Y_2^2}\Big)^{\frac{1}{2}}}\omega - \frac{\zeta\nabla_{S^{k-1}}Y_2}{\Big(1+(Y_2)_v^2 + \frac{|\nabla_{S^{k-1}}Y_2|^2}{Y_2^2}\Big)^{\frac{1}{2}}Y_2}\\
	&, \mathbf{v}\vartheta - \frac{\zeta(Y_2)_{\mathbf{v}}}{\Big(1+(Y_2)_v^2 + \frac{|\nabla_{S^{k-1}}Y_2|^2}{Y_2^2}\Big)^{\frac{1}{2}}}\vartheta\Big),
\end{align}
To ensure that the graph over $\bar{M}_{\tau}^2 \cap \{v_2\leq \frac{99\theta}{100}\}$ with graph function $\zeta $ is contained in $\bar{M}_{\tau}^1$, it suffices to require that we require that:
 \begin{align}
 	(Y_2 \cdot \omega, \mathbf{v}\vartheta) + \zeta\nu_{\bar{M}_{\tau}^2}\big|_{\mathbf{(\mathbf{v},\omega})} = (Y_1 \cdot \bar{\omega}, \bar{\mathbf{v}}\vartheta)\big|_{\mathbf{(\bar{\mathbf{v}},\bar{\omega}})}.
 \end{align}
This is amount to solving the equations:
\begin{align}\label{graph-region2-require2}
&\quad\, Y_1\cdot\bar{\omega}\Big|_{(\bar{\mathbf{v}},\bar{\omega})}\\
&=Y_2\cdot\omega+\frac{\zeta}{\Big(1+(Y_2)_v^2 + \frac{|\nabla_{S^{k-1}}Y_2|^2}{Y_2^2}\Big)^{\frac{1}{2}}}\omega 
    - \frac{\zeta\nabla_{S^{k-1}}Y_2}{\Big(1+(Y_2)_v^2 + \frac{|\nabla_{S^{k-1}}Y_2|^2}{Y_2^2}\Big)^{\frac{1}{2}}Y_2}\Big|_{(\mathbf{v},\omega)} 
\end{align} 
and
\begin{align}\label{graph-region2-require3}
    \mathbf{v} - \frac{(Y_2)_{{v}}\zeta}{\Big(1+(Y_2)_v^2 + \frac{|\nabla_{S^{k-1}}Y_2|^2}{Y_2^2}\Big)^{\frac{1}{2}}}\Big|_{(\mathbf{v},\omega)} = \bar{\mathbf{v}}.
\end{align}
Denote $\eta = \eta(\mathbf{v},\omega) = \Big(1+(Y_2)_v^2 + \frac{|\nabla_{S^{k-1}}Y_2|^2}{Y_2^2}\Big)^{\frac{1}{2}}\Big|_{(\mathbf{v},\omega)}$. 

From equation \eqref{graph-region2-require2}, we can directly solve that
\begin{align}\label{graph-region2-require5}
    \bar{\omega} = \bar{\omega}(\zeta, \mathbf{y},\omega) =  \frac{Y_2\cdot\omega+\eta^{-1}\zeta\omega - \frac{\eta^{-1}\zeta\nabla_{S^{k-1}}Y_2}{ Y_2}}{|Y_2\cdot\omega+\eta^{-1}\zeta\omega - \frac{\eta^{-1}\zeta\nabla_{S^{k-1}}Y_2}{ Y_2}|}\ \Bigg|_{(\mathbf{v},\omega)}.
\end{align}

By the tip derivative estimate (Proposition \ref{lemma-mainY}) and unique sharp derivative (Proposition \ref{strong_sharp_asymptotics}), there is a constant $C_2 = C_2(n,k)$ such that
\begin{align}\label{eta-upper-bound}
	\eta = \Big(1+(Y_2)_v^2 + \frac{|\nabla_{S^{k-1}}Y_2|^2}{Y_2^2}\Big)^{\frac{1}{2}} \leq C_2|\tau|^{\frac{1}{2}}.
\end{align}
holds in $\bar{M}_{\tau}^2\cap \{v_2 \leq \theta\}$.

Using above notations, \eqref{graph-region2-require2} becomes 
\begin{align}\label{graph-region2-require2-2}
	Y_2\cdot\omega+\eta^{-1}\zeta\omega - \frac{\eta^{-1}\zeta\nabla_{S^{k-1}}Y_2}{ Y_2}\Big|_{(\mathbf{v},\omega)} = Y_1(\mathbf{v} - (Y_2)_{{v}}\eta^{-1}\zeta\ , \ \bar{\omega})\cdot\bar{\omega}.
\end{align}
Since we already solved $\bar{\omega}$ in \eqref{graph-region2-require5} so that the direction of both sides of \eqref{graph-region2-require2-2} matches, it suffices to make sure that the length of both sides matches. This is equivalent to solving:
\begin{align}\label{graph-region2-require4}
	\quad&\mathcal{G}_{(\mathbf{v},\omega)}(\zeta)\\
    =& \sqrt{(Y_2 + \eta^{-1}\zeta)^2 + \frac{\eta^{-2}\zeta^2|\nabla_{S^{k-1}}Y_2|^2}{Y_2^2}}\Big|_{(\mathbf{v},\omega)} -  Y_1(\mathbf{v} - (Y_2)_{{v}}\eta^{-1}\zeta\ , \ \bar{\omega})\\
    =& 0.
\end{align}

By Proposition \ref{strong_sharp_asymptotics} and \ref{lemma-mainY}, we find that $\eta\geq 1$, $Y_i\geq {|\tau|}^{\frac{1}{2}}$,\  $|DY_i|\leq |\tau|^{\frac{1}{2}}$  and $|W|\leq \frac{1}{|\tau|^{10}}$. Moreover, Proposition \ref{lemma-mainY} implies that 
$\frac{|\nabla_{S^{k-1}}Y_i|}{Y_i} \leq 1$. All these estimates hold in the region 
$\bar{M}_{\tau}^i \cap \{v_i \leq \theta\}$ for $i=1, 2$
(after possibly decreasing $\tau_{*}$).  Therefore the following  estimates holds when $|\zeta|\leq 2|W|$:
\begin{align}
	 | \bar{\omega}-\omega|\leq& \frac{\zeta}{20|\tau|^{\frac{1}{2}}\eta}.\\
\end{align}
Moreover, when $\mathbf{v}\in [0, \frac{99\theta}{100}]$ and $|\zeta|\leq 2|W|$ we have
\begin{align}
        |\mathbf{v} - (Y_2)_{{v}}\eta^{-1}\zeta| \leq& \frac{99\theta}{100} + |\tau|^{-8} <\theta.
\end{align}
This will ensure that $\bar{\mathbf{v}} \in [0, \theta]$.
Using again that $|DY_i|\leq |\tau|^{\frac{1}{2}}$ we have
\begin{align}
	|Y_1(\mathbf{v} - (Y_2)_{{v}}\eta^{-1}\zeta\ , \ \bar{\omega}) - Y_1(\mathbf{v} - (Y_2)_{{v}}\eta^{-1}\zeta\ , \ {\omega}) | \leq 10^{-1}\eta^{-1}\zeta.
\end{align}
We need further estimate about deviation of derivatives. For any $|s|\leq 1$ and $|\zeta|\leq 2|W|$, we estimate:
\begin{align}\label{Y1-Y2-diff-1}
	 &\Big|(Y_1)_v(\mathbf{v} - s(Y_2)_{{v}}\eta^{-1}\zeta\ , \ {\omega}) -  (Y_1)_v(\mathbf{v}, \ {\omega})\Big|\\ \leq& \max |(Y_1)_{vv}|\max|(Y_2)_v|\eta^{-1}\zeta
     \leq |\tau|^{-8},
\end{align}
here we used the second derivative estimate in Proposition \ref{lemma-mainY}. We also note that 
\begin{align}\label{Y1-Y2-diff-2}
    |(Y_1)_v - (Y_2)_v| = |W_v|\leq \frac{1}{|\tau|^{10}}.
\end{align} 
Adding \eqref{Y1-Y2-diff-1} and \eqref{Y1-Y2-diff-2} together gives:
\begin{align}
	\Big|(Y_1)_v(\mathbf{v} - s(Y_2)_{{v}}\eta^{-1}\zeta\ , \ {\omega}) -  (Y_2)_v(\mathbf{v}, \ {\omega})\Big| \leq  |\tau|^{-8}.
\end{align}
Now we can proceed to estimate the following
\begin{align}
	&\big|Y_1(\mathbf{v} - (Y_2)_{{v}}\eta^{-1}\zeta\ , \ \bar{\omega}) - Y_1(\mathbf{v}  , \ {\omega}) + (Y_2)_v^2\eta^{-1}\zeta\big| \\
	 \leq& 10^{-1}\eta^{-1}\zeta +\big|Y_1(\mathbf{v} - (Y_2)_{{v}}\eta^{-1}\zeta\ , \ {\omega}) - Y_1(\mathbf{v}  , \ {\omega}) +(Y_2)_v^2\eta^{-1}\zeta \big| \\
	\leq&  10^{-1}\eta^{-1}\zeta + \Big|\int_0^1 \frac{d}{ds}\Big( Y_1\big(\mathbf{v} - s(Y_2)_{{v}}\eta^{-1}\zeta\ , \ {\omega}) \Big)+ (Y_2)_v^2\eta^{-1}\zeta  ds\Big|\\
	\leq& 10^{-1}\eta^{-1}\zeta + \int_0^1  \Big|(Y_1)_v(\mathbf{v} - s(Y_2)_{{v}}\eta^{-1}\zeta\ , \ {\omega}) -  (Y_2)_v(\mathbf{v}, \ {\omega})\Big|\cdot (Y_2)_v\eta^{-1}\zeta ds\\
	\leq& \frac{1}{5}\eta^{-1}\zeta.  
\end{align}
Next, we use the fact that $\frac{\eta^{-2}|\nabla_{S^{k-1}}Y_2|^2}{Y_2^2} < 1$ and $|Y_2|\geq |\tau|^{\frac{1}{2}}$, $|W|\leq \frac{1}{|\tau|^{10}}$, the following elementary estimate holds when $|\zeta|\leq 2|W|$:
\begin{align}
	Y_2 + \eta^{-1}\zeta  - 10\zeta^2Y_2^{-1}<& \sqrt{(Y_2 + \eta^{-1}\zeta)^2 + \frac{\eta^{-2}\zeta^2|\nabla_{S^{k-1}}Y_2|^2}{Y_2^2}} \\
    <& Y_2 + \eta^{-1}\zeta  + 10\zeta^2Y_2^{-1}.
\end{align}
Then we can estimate
\begin{align}
	\quad&\mathcal{G}_{(\mathbf{v},\omega)}(2|W(\mathbf{v},\omega)|) \\
    \geq& \sqrt{(Y_2 + \eta^{-1}\zeta)^2 + \frac{\eta^{-2}\zeta^2|\nabla_{S^{k-1}}Y_2|^2}{Y_2^2}}  - Y_1 + (Y_2)_v^2\eta^{-1}\zeta\Bigg|_{(\mathbf{v},\omega),\ \zeta = 2|W(\mathbf{v},\omega)|}\\
	-& \big|Y_1(\mathbf{v} - (Y_2)_{{v}}\eta^{-1}\zeta\ , \ \bar{\omega}) - Y_1(\mathbf{v},\omega) + (Y_2)_v^2\eta^{-1}\zeta\big|\Bigg|_{(\mathbf{v},\omega),\ \zeta = 2|W(\mathbf{v},\omega)|}\\
	\geq& Y_2 + \eta^{-1}\zeta - 10\zeta^2Y_2^{-1}  - Y_1 + (Y_2)_v^2\eta^{-1}\zeta - \frac{1}{5}\eta^{-1}\zeta \Bigg|_{(\mathbf{v},\omega),\ \zeta = 2|W(\mathbf{v},\omega)|}\\
	\geq& W + \frac{4(1 + (Y_2)_v^2)}{5\eta} \zeta - 200^{-1}\zeta \Bigg|_{(\mathbf{v},\omega),\ \zeta = 2|W(\mathbf{v},\omega)|}.
 \end{align}
 Note that Proposition \ref{lemma-mainY} implies  $\frac{|\nabla_{S^{k-1}}Y_2|^2}{Y_2^2}\leq \frac{1}{100}$. Therefore
 \begin{align*}
 	\frac{4(1 + (Y_2)_v^2)}{5\eta} = \frac{4(1 + (Y_2)_v^2)}{5(1 + (Y_2)_v^2 + \frac{|\nabla_{S^{k-1}}Y_2|^2}{Y_2^2})^{\frac{1}{2}}} \geq \frac{4(1 + (Y_2)_v^2)}{5(1 + (Y_2)_v^2 + 10^{-2})^{\frac{1}{2}}}  >\frac{2}{3}.
 \end{align*}
 We deduce that
 \begin{align}
	\mathcal{G}_{(\mathbf{v},\omega)}(2|W(\mathbf{v},\omega)|)  
	\geq& -|W| + \frac{2}{3}\cdot 2|W| - 100^{-1}|W|\geq 0.
 \end{align}
In the same way we get:
 \begin{align}
	\quad &\mathcal{G}_{(\mathbf{v},\omega)}(-2|W(\mathbf{v},\omega)|) \\
    \leq& \sqrt{(Y_2 + \eta^{-1}\zeta)^2 + \frac{\eta^{-2}\zeta^2|\nabla_{S^{k-1}}Y_2|^2}{Y_2^2}}  - Y_1 + (Y_2)_v^2\eta^{-1}\zeta\Bigg|_{(\mathbf{v},\omega),\ \zeta = -2|W(\mathbf{v},\omega)|}\\
	-& \big|Y_1(\mathbf{v} - (Y_2)_{{v}}\eta^{-1}\zeta\ , \ \bar{\omega}) - Y_1(\mathbf{v},\omega) + (Y_2)_v^2\eta^{-1}\zeta\big|\Bigg|_{(\mathbf{v},\omega),\ \zeta = -2|W(\mathbf{v},\omega)|}\\
	\leq& Y_2 + \eta^{-1}\zeta + 10\zeta^2Y_2^{-1}  - Y_1 + (Y_2)_v^2\eta^{-1}\zeta + \frac{1}{5}\eta^{-1}\zeta \Bigg|_{(\mathbf{v},\omega),\ \zeta = -2|W(\mathbf{v},\omega)|}\\
	\leq& W + \frac{(1 + (Y_2)_v^2)}{\eta} \zeta + 200^{-1}\zeta \Bigg|_{(\mathbf{v},\omega),\ \zeta = -2|W(\mathbf{v},\omega)|} \\
    \leq& -\frac{1}{2}	|W| \leq 0.
 \end{align}
By the intermediate value theorem, there is a $\zeta \in [-2|W(\mathbf{v},\omega)|, 2|W(\mathbf{v},\omega)|]$ such that 
\begin{align}
	\mathcal{G}_{(\mathbf{v},\omega)}(\zeta) = 0.
\end{align}
We next compute the derivative so as to establish uniqueness. For simplicity, we denote
\begin{align}
	 SQ = \sqrt{(Y_2 + \eta^{-1}\zeta)^2 + \frac{\eta^{-2}\zeta^2|\nabla_{S^{k-1}}Y_2|^2}{Y_2^2}}\Big|_{(\mathbf{v},\omega)}.
\end{align}
Next we assume that $|\zeta| \leq 2|W(\mathbf{v},\omega)|$. First we have
\begin{align}
	Y_2 \leq SQ\leq 2Y_2, \quad 
\end{align}
Then we can estimate the derivative:
\begin{align}
	\quad&\mathcal{G}_{(\mathbf{v},\omega)}'(\zeta)\\
    =& \frac{\eta^{-1}(Y_2 + \eta^{-1}\zeta) + \frac{\eta^{-2}|\nabla_{S^{k-1}}Y_2|^2}{Y_2^2}\zeta}{SQ} + (Y_1)_v\Big|_{(\mathbf{v} - (Y_2)_{{v}}\eta^{-1}\zeta\ , \ \bar{\omega})}\cdot (Y_2)_v\big|_{(\mathbf{v},\omega)}\\
	\geq& \frac{\eta^{-1}(Y_2-2|\zeta|)}{SQ} - \Big|(Y_1)_v(\mathbf{v} - s(Y_2)_{{v}}\eta^{-1}\zeta\ , \ {\omega}) -  (Y_2)_v(\mathbf{v}, \ {\omega})\Big| |(Y_2)_v| + (Y_2)_v^2\\
	\geq& \frac{1}{4\eta} - |\tau|^{-8}.
\end{align} 
By \eqref{eta-upper-bound} we have $|\tau|^{-8}< \frac{1}{20\eta}$. Hence
\begin{align}
	\quad&\mathcal{G}_{(\mathbf{v},\omega)}'(\zeta)	\geq \frac{1}{4\eta} - |\tau|^{-8} \geq \frac{1}{5\eta} >0.
\end{align} 
This means that there is a unique solution in $ [-2|W(\mathbf{v},\omega)|, 2|W(\mathbf{v},\omega)|]$, which we now denote by $\zeta(\mathbf{v},\omega)$, such that 
\begin{align}
	\mathcal{G}_{(\mathbf{v},\omega)}(\zeta(\mathbf{v},\omega)) = 0.
\end{align}
Next, we explain how we get higher derivative estimates for $Y_i$. That is, for all $|\alpha|\leq m$:
\begin{align}
	|D^{\alpha} Y_i| \leq C_3(\tau_0, l, m).
\end{align}
Indeed, by convexity and the sharp asymptotics Proposition \ref{strong_sharp_asymptotics}, we have the curvature bound $\sup\limits_{\bar{M}_{\tau}}|A|\leq 10|\tau|^{\frac{1}{2}}$ for all $\tau\leq \tau_{*}$. By the curvature estimate, c.f. \cite[Theorem 1.12]{HaslhoferKleiner_meanconvex}, we deduce that $\sup\limits_{\bar{M}_{\tau}}|\nabla^{\alpha} A|\leq C_3(\tau_0, l, m)$ for all $|\alpha|\leq m$. Then we can express $|D^{\alpha} Y_i|$ using $|\nabla^{\alpha} A|$ and $|\nabla Y|$. Since  Proposition \ref{lemma-mainY} gives $|\nabla Y|\leq |\tau|^{\frac{1}{2}}\leq (|\tau_0|+ l)^{\frac{1}{2}}$, we can get the higher derivative estimate for $Y_i$.

Finally, we define $\mathcal{G}(\zeta, \mathbf{v},\omega)  = \mathcal{G}_{(\mathbf{v},\omega)}(\zeta)$ and observe that $\mathcal{G}(0, \mathbf{v},\omega) = W(\mathbf{v},\omega)$. Then for all multi-index $\alpha$ with $|{\alpha}|\leq m$ and ${\tau\in[\tau_0\!-\!l, \tau_0]}$, the following estimate holds when $|\zeta|  \leq \frac{1}{|\tau|^8}$ and $\mathbf{v}\leq \frac{99\theta}{100}$:
\begin{align*}
    |D^{(\alpha)}\mathcal{G} | \leq ||W||_{C^m(B_{\theta})} + C_3(\tau_0, l, m)|\zeta|\leq C_3(\tau_0, l, m)||W||_{C^{m}(B_{\theta})}.
\end{align*}
Applying implicit function theorem for $\mathcal{G}(\zeta, \mathbf{v},\omega) \!=\! 0$ as before,   we get that for all multi-index $\alpha$ with $|{\alpha}|\leq m$ and ${\tau\in[\tau_0\!-\!l, \tau_0]}$,
\begin{equation}
    \|\nabla^{\alpha} \zeta\| \leq C_3(\tau_0, l, m)||W||_{C^{m}(B_{\theta})}.
\end{equation}
holds when $v_2 \leq \frac{99\theta}{100}$.
We have thus found $\zeta$ that satisfies \eqref{graph-region2-require2} and \eqref{graph-region2-require3}. Consequently $\bar{M}_{\tau}^1$ is a graph of $\zeta$ over $\bar{M}_{\tau}^2 \cap\{v_2 \leq\frac{99\theta}{100}\}$.


Putting the above two cases together completes the proof of the Lemma.
\end{proof}

\end{appendix}
\bibliographystyle{alpha}
\bibliography{rigidity_of_ancient_ovals_CDZ}

\newcommand{\etalchar}[1]{$^{#1}$}
\newcommand{\noopsort}[1]{} \newcommand{\singleletter}[1]{#1}
\begin{thebibliography}{CHHW22}

\bibitem[ABDS22]{ABDS}
S.~Angenent, S.~Brendle, P.~Daskalopoulos, and N.~Sesum.
\newblock Unique asymptotics of compact ancient solutions to three-dimensional
  {R}icci flow.
\newblock {\em Comm. Pure Appl. Math}, 75(5):1032--1073, 2022.

\bibitem[ADS19]{ADS1}
S.~Angenent, P.~Daskalopoulos, and N.~Sesum.
\newblock Unique asymptotics of ancient convex mean curvature flow solutions.
\newblock {\em J. Differential Geom.}, 111(3):381--455, 2019.

\bibitem[ADS20]{ADS2}
S.~Angenent, P.~Daskalopoulos, and N.~Sesum.
\newblock Uniqueness of two-convex closed ancient solutions to the mean
  curvature flow.
\newblock {\em Ann. of Math. (2)}, 192(2):353--436, 2020.

\bibitem[ADS23]{ADS_dynamics}
S.~Angenent, P.~Daskalopoulos, and N.~Sesum.
\newblock Dynamics mean curvature flow.
\newblock {\em arXiv:2305.17272}, 2023.

\bibitem[And12]{Andrews_noncollapsing}
B.~Andrews.
\newblock Noncollapsing in mean-convex mean curvature flow.
\newblock {\em Geom. Topol.}, 16(3):1413--1418, 2012.

\bibitem[Ang91]{angenent1991formation}
Sigurd Angenent.
\newblock On the formation of singularities in the curve shortening flow.
\newblock {\em Journal of Differential Geometry}, 33(3):601--633, 1991.

\bibitem[AV97]{AV_degenerate_neckpinch}
S.~Angenent and J.~Velazquez.
\newblock Degenerate neckpinches in mean curvature flow.
\newblock {\em J. Reine Angew. Math.}, 482:15--66, 1997.

\bibitem[BC19]{BC1}
S.~Brendle and K.~Choi.
\newblock Uniqueness of convex ancient solutions to mean curvature flow in
  {$\mathbb{R}^3$}.
\newblock {\em Invent. Math.}, 217(1):35--76, 2019.

\bibitem[BC21]{BC2}
S.~Brendle and K.~Choi.
\newblock Uniqueness of convex ancient solutions to mean curvature flow in
  higher dimensions.
\newblock {\em Geom. Topol.}, 25(5):2195--2234, 2021.

\bibitem[BH16]{BH_surgery}
S.~Brendle and G.~Huisken.
\newblock Mean curvature flow with surgery of mean convex surfaces in
  {$\mathbb{R}^3$}.
\newblock {\em Invent. Math.}, 203(2):615--654, 2016.

\bibitem[BLT21]{BLT1}
T.~Bourni, M.~Langford, and G.~Tinaglia.
\newblock Collapsing ancient solutions of mean curvature flow.
\newblock {\em J. Differential Geom.}, 119(2):187--219, 2021.

\bibitem[BLT22]{BLT2}
T.~Bourni, M.~Langford, and G.~Tinaglia.
\newblock Ancient mean curvature flows out of polytopes.
\newblock {\em Geom. Topol.}, 26(4):1849--1905, 2022.

\bibitem[BN24]{brendle2024local}
S.~Brendle and K.~Naff.
\newblock A local noncollapsing estimate for mean curvature flow.
\newblock {\em American Journal of Mathematics}, 146(5):1463--1468, 2024.

\bibitem[Bre15]{Brendle_inscribed}
S.~Brendle.
\newblock A sharp bound for the inscribed radius under mean curvature flow.
\newblock {\em Invent. Math.}, 202(1):217--237, 2015.

\bibitem[Bre20]{brendle2020ancient}
S.~Brendle.
\newblock Ancient solutions to the ricci flow in dimension $3$.
\newblock {\em Acta Mathematica}, 225(1), 2020.

\bibitem[CDD{\etalchar{+}}]{CDDHS}
B.~Choi, P.~Daskalopoulos, W.~Du, R.~Haslhofer, and N.~Sesum.
\newblock Classification of bubble-sheet ovals in $\mathbb{R}^{4}$.
\newblock {\em Geom. Topol. (to appear)}.

\bibitem[CH24]{CH-classification-r4}
K.~Choi and R.~Haslhofer.
\newblock Classification of ancient noncollapsed flows in $\mathbb{R}^4$.
\newblock {\em arXiv preprint arXiv:2412.10581}, 2024.

\bibitem[CHH21]{CHH_ovals}
B.~Choi, R.~Haslhofer, and O.~Hershkovits.
\newblock A note on the self-similarity of limit flows.
\newblock {\em Proc. Amer. Math. Soc.}, 143(10):4433--4437, 2021.

\bibitem[CHH22]{CHH}
K.~Choi, R.~Haslhofer, and O.~Hershkovits.
\newblock Ancient low entropy flows, mean convex neighborhoods, and uniqueness.
\newblock {\em Acta Math.}, 228(2):217--301, 2022.

\bibitem[CHH23]{CHH_translator}
K.~Choi, R.~Haslhofer, and O.~Hershkovits.
\newblock Classification of noncollapsed translators in $\mathbb{R}^4$.
\newblock {\em Camb. J. Math.}, 11(3), 2023.

\bibitem[CHHW22]{CHHW}
K.~Choi, R.~Haslhofer, O.~Hershkovits, and B.~White.
\newblock Ancient asymptotically cylindrical flows and applications.
\newblock {\em Invent. Math.}, 229:139--241, 2022.

\bibitem[CM16]{CM_arrival}
T.~Colding and W.~Minicozzi.
\newblock Differentiability of the arrival time.
\newblock {\em Comm. Pure Appl. Math.}, 69(12):2349--2363, 2016.

\bibitem[DH]{DH_ovals}
W.~Du and R.~Haslhofer.
\newblock On uniqueness and nonuniqueness of ancient ovals.
\newblock {\em Amer. J. Math. (to appear)}.

\bibitem[DH21]{DH_blowdown}
W.~Du and R.~Haslhofer.
\newblock The blowdown of ancient noncollapsed mean curvature flows.
\newblock {\em arXiv: 2106.04042}, 2021.

\bibitem[DH23]{DH_no_rotation}
W.~Du and R.~Haslhofer.
\newblock A nonexistence result for rotating mean curvature flows in
  $\mathbb{R}^4$.
\newblock {\em Journal f{\"u}r die reine und angewandte Mathematik (Crelles
  Journal)}, 2023(802):275--285, 2023.

\bibitem[DH24]{DH_hearing_shape}
W.~Du and R.~Haslhofer.
\newblock Hearing the shape of ancient noncollapsed flows in {$\mathbb{R}^4$}.
\newblock {\em Comm. Pure Appl. Math.}, 77(1):543--582, 2024.

\bibitem[DZ22]{DZ_spectral_quantization}
W.~Du and J.~Zhu.
\newblock Spectral quantization for ancient asymptotically cylindrical flows.
\newblock {\em arXiv:2211.02595}, 2022.

\bibitem[Fed14]{federer2014geometric}
H.~Federer.
\newblock {\em Geometric measure theory}.
\newblock Springer, 2014.

\bibitem[Ham86]{Ham_pco}
R.~Hamilton.
\newblock Four-manifolds with positive curvature operator.
\newblock {\em J. Differential Geom}, 24(2):153--179, 1986.

\bibitem[Ham95]{Hamilton_Harnack}
R.~Hamilton.
\newblock Harnack estimate for the mean curvature flow.
\newblock {\em J. Differential Geom.}, 41(1):215--226, 1995.

\bibitem[Has24]{Haslhofer_4dRicci}
R.~Haslhofer.
\newblock On $\kappa$-solutions and canonical neighborhoods in 4d ricci flow.
\newblock {\em Journal f{\"u}r die reine und angewandte Mathematik (Crelles
  Journal)}, 2024(811):257--265, 2024.

\bibitem[HH16]{HaslhoferHershkovits_ancient}
R.~Haslhofer and O.~Hershkovits.
\newblock Ancient solutions of the mean curvature flow.
\newblock {\em Comm. Anal. Geom.}, 24(3):593--604, 2016.

\bibitem[HIMW19]{HIMW}
D.~Hoffman, T.~Ilmanen, F.~Mart\'{\i}n, and B.~White.
\newblock Graphical translators for mean curvature flow.
\newblock {\em Calc. Var. PDE}, 58(4):Paper No. 117, 29, 2019.

\bibitem[HK15]{HK_inscribed}
R.~Haslhofer and B.~Kleiner.
\newblock On {B}rendle's estimate for the inscribed radius under mean curvature
  flow.
\newblock {\em Int. Math. Res. Not.}, pages 6558--6561, 2015.

\bibitem[HK17a]{HaslhoferKleiner_meanconvex}
R.~Haslhofer and B.~Kleiner.
\newblock Mean curvature flow of mean convex hypersurfaces.
\newblock {\em Comm. Pure Appl. Math.}, 70(3):511--546, 2017.

\bibitem[HK17b]{HaslhoferKleiner_surgery}
R.~Haslhofer and B.~Kleiner.
\newblock Mean curvature flow with surgery.
\newblock {\em Duke Math. J.}, 166(9):1591--1626, 2017.

\bibitem[HS09]{HuiskenSinestrari_surgery}
G.~Huisken and C.~Sinestrari.
\newblock Mean curvature flow with surgeries of two-convex hypersurfaces.
\newblock {\em Invent. Math.}, 175(1):137--221, 2009.

\bibitem[Ilm03]{Ilmanen_problems}
T.~Ilmanen.
\newblock Problems in mean curvature flow.
\newblock {\em preprint}, 2003.

\bibitem[Jac20]{magic}
Jules Jacobs.
\newblock A magic determinant formula for symmetric polynomials of eigenvalues.
\newblock {\em arXiv preprint arXiv:2009.01345}, 2020.

\bibitem[Kon14]{ODE-book}
Qingkai Kong.
\newblock {\em A short course in ordinary differential equations}.
\newblock Springer, 2014.

\bibitem[Lai24]{Lai}
Y.~Lai.
\newblock A family of $3d$ steady gradient solitons that are flying wings.
\newblock {\em Journal of Differential Geometry}, 126(1):297--328, 2024.

\bibitem[Lie96]{Lieberman}
G.~Lieberman.
\newblock {\em Second order parabolic differential equations}.
\newblock World Scientific Publishing Co., Inc., River Edge, NJ, 1996.

\bibitem[Mat23]{Rectifiability}
P.~Mattila.
\newblock {\em Rectifiability: a survey}, volume 483.
\newblock Cambridge University Press, 2023.

\bibitem[SW09]{ShengWang}
W.~Sheng and X.~Wang.
\newblock Singularity profile in the mean curvature flow.
\newblock {\em Methods Appl. Anal.}, 16(2):139--155, 2009.

\bibitem[SWX25]{sun2025passing}
A.~Sun, Z.~Wang, and J.~Xue.
\newblock Passing through nondegenerate singularities in mean curvature flows.
\newblock {\em arXiv preprint arXiv:2501.16678}, 2025.

\bibitem[SX22]{sun2022generic}
A.~Sun and J.~Xue.
\newblock Generic mean curvature flows with cylindrical singularities.
\newblock {\em arXiv preprint arXiv:2210.00419}, 2022.

\bibitem[Whi03]{White_nature}
B.~White.
\newblock The nature of singularities in mean curvature flow of mean-convex
  sets.
\newblock {\em J. Amer. Math. Soc.}, 16(1):123--138, 2003.

\end{thebibliography}

\vspace{5mm}

{\sc Beomjun Choi, Department of Mathematics, POSTECH, Gyeongbuk, Korea

{\sc Wenkui Du, Department of Mathematics, Massachusetts Institute of Technology,  Massachusetts, 02139, USA}

{\sc Jingze Zhu, Department of Mathematics, Massachusetts Institute of Technology,  Massachusetts, 02139, USA}

\emph{E-mail:} bchoi@postech.ac.kr, 
duwenkui@mit.edu, zhujz@mit.edu.

\end{document}